\documentclass[12pt]{amsart}%
\usepackage{amsfonts}
\usepackage{amsmath}
\usepackage{amssymb}
\usepackage{graphicx}%
\providecommand{\U}[1]{\protect\rule{.1in}{.1in}}
\newtheorem{theorem}{Theorem}
\theoremstyle{plain}

\newtheorem{corollary}{Corollary}

\newtheorem{example}{Example}

\newtheorem{lemma}{Lemma}

\newtheorem{proposition}{Proposition}
\newtheorem{remark}{Remark}

\numberwithin{equation}{section}
\numberwithin{equation}{section}
\numberwithin{theorem}{section}
\numberwithin{lemma}{section}
\numberwithin{remark}{section}
\numberwithin{example}{section}
\numberwithin{proposition}{section}
\numberwithin{definition}{section}
\numberwithin{corollary}{section}
\begin{document}
\title[Noncommutative complex analytic geometry]{Noncommutative complex analytic geometry of a contractive quantum plane}
\author{Anar Dosi}
\address{College of Mathematical Sciences, Harbin Engineering University, Nangang
District, Harbin, 150001, China }
\email{(dosiev@yahoo.com), (dosiev@hrbeu.edu.cn)}
\date{October 25, 2024}

\begin{abstract}
In the paper we investigate the Banach space representations of Manin's
quantum $q$-plane for $\left\vert q\right\vert \neq1$. The Arens-Michael
envelope of the quantum plane is extended up to a Fr\'{e}chet algebra presheaf
over its spectrum. The obtained ringed space represents the geometry of the
quantum plane as a union of two irreducible components being copies of the
complex plane equipped with the $q$-topology and the disk topology,
respectively. It turns out that the Fr\'{e}chet algebra presheaf is
commutative modulo its Jacobson radical, which is decomposed into a
topological direct sum. The related noncommutative functional calculus problem
and the spectral mapping property are solved in terms of the noncommutative
Harte spectrum.

\end{abstract}
\maketitle

\section{Introduction\label{secInt}}

The main principle of noncommutative geometry is to assign a geometric object
to a noncommutative (associative) algebra over a fixed field, which can be
restored through the global sections of the structure sheaf of that geometric
space. By a geometric object we mean a ringed space $\left(  X,\mathcal{O}%
_{X}\right)  $ which consists of a topological space $X$ and the structure
sheaf $\mathcal{O}_{X}$ of noncommutative algebras whose algebra
$\Gamma\left(  X,\mathcal{O}_{X}\right)  $ of the global sections is reduced
to a given noncommutative algebra $A$. In algebraic geometry that relation
defines an anti-equivalence between the categories of the commutative rings
and the affine schemes. In the complex analytic geometry we have an
anti-equivalence between the categories of locally compact topological spaces
and the commutative $C^{\ast}$-algebras. That duality allows us to have some
geometric constructions in the algebraic interpretations and vice-verse.
Within the framework of noncommutative geometry, von Neumann algebras are
treated as noncommutative measures whereas the theory of $C^{\ast}$-algebras
is refereed to noncommutative topology.

Noncommutative complex analytic geometry deals with the Banach space
representations of a noncommutative complex algebra. A geometric space
$\left(  X,\mathcal{O}_{X}\right)  $ of a finitely generated noncommutative
complex algebra $A$ consists of an analytic space $X$ and a noncommutative
Fr\'{e}chet algebra (pre)sheaf $\mathcal{O}_{X}$ so that $\Gamma\left(
X,\mathcal{O}_{X}\right)  $ represents the noncommutative algebra of all
entire functions in the generators of $A$. The abstract characterization of
the algebra of all entire functions was proposed by A. Ya. Helemskii
\cite[5.2.21]{Hel} and J. L. Taylor \cite{Tay2}, which is well known as the
Arens-Michael envelope of an algebra. The Arens-Michael envelope of a complex
algebra $A$ is the completion of $A$ with respect to the family of all
multiplicative seminorms defined on $A$. If $A$ is the algebra of all
polynomial functions on a complex affine algebraic variety $X$, then the
Arens-Michael envelope of $A$ turns out to be the algebra of holomorphic
functions on $X$ \cite{Pir}. In particular, the Arens-Michael envelope of the
algebra $A=\mathbb{C}\left[  x_{1},\ldots,x_{n}\right]  $ of all complex
polynomials in $n$-variables is the Fr\'{e}chet algebra $\mathcal{O}\left(
\mathbb{C}^{n}\right)  $ of all entire functions on $\mathbb{C}^{n}$. In the
case of a noncommutative polynomial algebra $A$, its Arens-Michael envelope
represents the algebra of all entire functions in noncommutative variables
generating $A$. If $A=\mathcal{U}\left(  \mathfrak{g}\right)  $ is the
universal enveloping algebra of a finite dimensional nilpotent Lie algebra
$\mathfrak{g}$ then its Arens-Michael envelope $\mathcal{O}_{\mathfrak{g}}$ is
the algebra of all noncommutative entire function in elements of
$\mathfrak{g}$, and the absolute basis problem in $\mathcal{O}_{\mathfrak{g}}$
was investigated in \cite{DosIzv}. More detailed and deep analysis of the
Arens-Michael envelopes were reflected in \cite{Pir06}, \cite{Pir16} and
\cite{Pir} by A. Yu. Pirkovskii.

The present paper is devoted to a noncommutative complex analytic geometry of
Yu. I. Manin's quantum plane \cite{M} including the related noncommutative
functional calculus and the joint spectral theory. The quantum plane (or just
$q$-plane) is the free associative algebra
\[
\mathfrak{A}_{q}=\mathbb{C}\left\langle x,y\right\rangle /\left(
xy-q^{-1}yx\right)
\]
generated by $x$ and $y$ modulo the relation $xy=q^{-1}yx$, where
$q\in\mathbb{C}\backslash\left\{  0,1\right\}  $. The Arens-Michael envelope
of $\mathfrak{A}_{q}$ is denoted by $\mathcal{O}_{q}\left(  \mathbb{C}%
^{2}\right)  $. If $x$ and $y$ are invertible additionally, then the algebra
represents the quantum $2$-torus. The properties of the Arens-Michael
envelopes of the quantum tori were investigated in \cite{Pir09}. If
$\left\vert q\right\vert \neq1$, then we deal with the contractive quantum
plane. Our first task is to reveal a ringed space that stands for the
contractive $q$-plane $\mathfrak{A}_{q}$. The Arens-Michael envelope
$\mathcal{O}_{q}\left(  \mathbb{C}^{2}\right)  $ representing the algebra of
all noncommutative entire functions (see also \cite{Goor}) in $x$ and $y$
consists of the following absolutely convergent power series
\begin{align*}
\mathcal{O}_{q}\left(  \mathbb{C}^{2}\right)   &  =\left\{  f=\sum_{i,k}%
a_{ik}x^{i}y^{k}:\left\Vert f\right\Vert _{\rho}=\sum_{i,k}\left\vert
a_{ik}\right\vert \rho^{i+k}<\infty,\rho>0\right\}  \text{ if }\left\vert
q\right\vert \leq1,\\
\mathcal{O}_{q}\left(  \mathbb{C}^{2}\right)   &  =\left\{  f=\sum_{i,k}%
a_{ik}x^{i}y^{k}:\left\Vert f\right\Vert _{\rho}=\sum_{i,k}\left\vert
a_{ik}\right\vert \left\vert q\right\vert ^{-ik}\rho^{i+k}<\infty
,\rho>0\right\}  \text{ if }\left\vert q\right\vert \geq1
\end{align*}
(see \cite[5.14]{Pir}). The case of $\left\vert q\right\vert >1$ can be
reduced to the case of $\left\vert q\right\vert <1$ by flipping the variables
$x$ and $y$, thereby whatever construction over the quantum plane done for
$\left\vert q\right\vert <1$ can be conveyed to the case of $\left\vert
q\right\vert >1$ too. As in the commutative case, the space of all continuous
characters $\operatorname{Spec}\mathcal{O}_{q}\left(  \mathbb{C}^{2}\right)  $
of the Arens-Michael envelope of $\mathfrak{A}_{q}$ stands for the
noncommutative "analytic" space, whose structure would be given by a
(pre)sheaf $\mathcal{O}_{q}$ over $\operatorname{Spec}\mathcal{O}_{q}\left(
\mathbb{C}^{2}\right)  $ of noncommutative Fr\'{e}chet $\widehat{\otimes}%
$-algebras extending the algebra $\mathcal{O}_{q}\left(  \mathbb{C}%
^{2}\right)  $.

It turns out that $\mathcal{O}_{q}\left(  \mathbb{C}^{2}\right)
=\mathcal{I}_{x}\oplus\mathcal{I}_{xy}\oplus\mathcal{I}_{y}$ is a topological
direct sum of the closed unital subalgebra $\mathcal{I}_{x}$ generated by $x$,
the closed two-sided ideal $\mathcal{I}_{xy}$ generated by $xy$, and the
closed non-unital subalgebra $\mathcal{I}_{y}$ generated by $y$ (see
Subsection \ref{sAME}). In this case, $\operatorname{Rad}\mathcal{O}%
_{q}\left(  \mathbb{C}^{2}\right)  \subseteq\mathcal{I}_{xy}=\cap\left\{
\ker\left(  \lambda\right)  :\lambda\in\operatorname{Spec}\left(
\mathcal{O}_{q}\left(  \mathbb{C}^{2}\right)  \right)  \right\}  $ and%
\[
\operatorname{Spec}\left(  \mathcal{O}_{q}\left(  \mathbb{C}^{2}\right)
\right)  =\operatorname{Spec}\left(  \mathcal{O}_{q}\left(  \mathbb{C}%
^{2}\right)  /\mathcal{I}_{xy}\right)  =\mathbb{C}_{xy},
\]
where $\mathbb{C}_{xy}=\mathbb{C}_{x}\cup\mathbb{C}_{y}$, $\mathbb{C}%
_{x}=\mathbb{C\times}\left\{  0\right\}  \subseteq\mathbb{C}^{2}$,
$\mathbb{C}_{y}=\left\{  0\right\}  \times\mathbb{C\subseteq C}^{2}$. If
$\left\vert q\right\vert \neq1$ then the algebra $\mathcal{O}_{q}\left(
\mathbb{C}^{2}\right)  $ is commutative modulo its Jacobson radical
$\operatorname{Rad}\mathcal{O}_{q}\left(  \mathbb{C}^{2}\right)  $ and
$\mathcal{I}_{xy}=\operatorname{Rad}\mathcal{O}_{q}\left(  \mathbb{C}%
^{2}\right)  $, that is, the topological direct sum decomposition
\[
\mathcal{O}_{q}\left(  \mathbb{C}^{2}\right)  =\mathcal{O}\left(
\mathbb{C}_{x}\right)  \oplus\operatorname{Rad}\mathcal{O}_{q}\left(
\mathbb{C}^{2}\right)  \oplus\mathcal{I}_{y}%
\]
holds. Thus the noncommutative analytic space $\mathbb{C}_{xy}=\mathbb{C}%
_{x}\cup\mathbb{C}_{y}$ is a union of two copies of the complex plane. One
needs to equip $\mathbb{C}_{xy}$ with a suitable topology that would affiliate
to the noncommutative multiplication of the quantum plane. We endow
$\mathbb{C}_{x}$ with the $\mathfrak{q}$-topology (see Subsection
\ref{subsecQT}) given by the $q$-spiraling open subsets of $\mathbb{C}$,
whereas $\mathbb{C}_{y}$ is equipped with the disk topology $\mathfrak{d}$
given by all open disks in $\mathbb{C}$ centered at the origin. Both are
non-Hausdorff topologies with their unique generic point zero, and they are
weaker than the original topology of the complex plane. The space
$\mathbb{C}_{xy}$ is equipped with the final topology so that both embedding
\[
\left(  \mathbb{C}_{x},\mathfrak{q}\right)  \hookrightarrow\mathbb{C}%
_{xy}\hookleftarrow\left(  \mathbb{C}_{y},\mathfrak{d}\right)
\]
are continuous, which is called $\left(  \mathfrak{q},\mathfrak{d}\right)
$\textit{-topology of} $\mathbb{C}_{xy}$. It is the union $\mathbb{C}%
_{xy}=\mathbb{C}_{x}\cup\mathbb{C}_{y}$ of two irreducible components, whose
intersection is a unique generic point. The direct image of the standard
Fr\'{e}chet sheaf $\mathcal{O}$ of stalks of holomorphic functions on
$\mathbb{C}$ along the continuous identity map $\mathbb{C\rightarrow}\left(
\mathbb{C}_{x},\mathfrak{q}\right)  $ is denoted by $\mathcal{O}%
^{\mathfrak{q}}$. In a similar way, we have the sheaf $\mathcal{O}%
^{\mathfrak{d}}$ on $\left(  \mathbb{C}_{y},\mathfrak{d}\right)  $. The sheaf
$\mathcal{O}^{\mathfrak{q}}$ has the subsheaf $\mathcal{I}_{\mathfrak{q}}$ of
holomorphic functions whose stalks vanish at the origin. Its unitization
$\mathcal{I}_{\mathfrak{q}}^{+}$ is reduced to $\mathcal{O}^{\mathfrak{q}}$ on
$\mathbb{C}_{x}$. In a similar way, $\mathcal{I}_{\mathfrak{d}}^{+}%
=\mathcal{O}^{\mathfrak{d}}$ on $\mathbb{C}_{y}$.

By \textit{a noncommutative contractive }$q$-\textit{plane }we mean the ringed
space $\left(  \mathbb{C}_{xy},\mathcal{O}_{q}\right)  $ to be the unital
projective tensor product of the ringed spaces $\left(  \mathbb{C}%
_{x},\mathcal{I}_{\mathfrak{q}}\right)  $ and $\left(  \mathbb{C}%
_{y},\mathcal{I}_{\mathfrak{d}}\right)  $. The underline space $\mathbb{C}%
_{xy}$ is equipped with the $\left(  \mathfrak{q},\mathfrak{d}\right)
$\textit{-}topology and its structure presheaf is defined as
\[
\mathcal{O}_{q}=\mathcal{I}_{\mathfrak{q}}^{+}\widehat{\otimes}\mathcal{I}%
_{\mathfrak{d}}^{+}=\left(  \mathcal{O}^{\mathfrak{q}}|\mathbb{C}_{x}\right)
\widehat{\otimes}\left(  \mathcal{O}^{\mathfrak{d}}|\mathbb{C}_{y}\right)  .
\]
It turns out that $\mathcal{O}_{q}$ is a presheaf of noncommutative
Fr\'{e}chet $\widehat{\otimes}$-algebras equipped with the formal
$q$-multiplication (see Subsection \ref{subsecCXY}) such that $\mathcal{O}%
_{q}\left(  \mathbb{C}_{xy}\right)  =\mathcal{O}_{q}\left(  \mathbb{C}%
^{2}\right)  $ up to a topological isomorphism of the Fr\'{e}chet
$\widehat{\otimes}$-algebras. The first central result (see Theorem
\ref{tdecomO}) of the paper asserts that there is a natural topological direct
sum decomposition
\[
\mathcal{O}_{q}\left(  U\right)  =\mathcal{O}^{\mathfrak{q}}\left(
U_{x}\right)  \oplus\operatorname{Rad}\mathcal{O}_{q}\left(  U\right)
\oplus\mathcal{I}_{\mathfrak{d}}\left(  U_{y}\right)
\]
for every $\left(  \mathfrak{q},\mathfrak{d}\right)  $-open subset
$U=U_{x}\cup U_{y}\subseteq\mathbb{C}_{xy}$. In particular, the algebra
$\mathcal{O}_{q}\left(  U\right)  $ is commutative modulo its Jacobson radical
$\operatorname{Rad}\mathcal{O}_{q}\left(  U\right)  $ and $\operatorname{Spec}%
\left(  \mathcal{O}_{q}\left(  U\right)  \right)  =U$. Thus the sheaf
$\mathcal{O}_{q}$ admits the following decomposition $\mathcal{O}%
_{q}=\mathcal{O}^{\mathfrak{q}}\oplus\mathcal{I}_{xy}\oplus\mathcal{I}%
_{\mathfrak{d}}$ with the two-sided ideal subsheaf $\mathcal{I}_{xy}$
associated to the presheaf $\operatorname{Rad}\mathcal{O}_{q}$.

The next target is to solve the noncommutative functional calculus problem
within the noncommutative space $\left(  \mathbb{C}_{xy},\mathcal{O}%
_{q}\right)  $ taking into account that $\mathcal{O}_{q}$ represents the
stalks of noncommutative holomorphic functions in $x$ and $y$. One needs to
solve the problem when a left Banach $\mathcal{O}_{q}\left(  \mathbb{C}%
_{xy}\right)  $-module structure can be lifted up to a left Banach
$\mathcal{O}_{q}\left(  U\right)  $-module one for an open subset
$U\subseteq\mathbb{C}_{xy}$. That is the challenging task to analyze all
details including the joint spectrum that stands by and plays the key role in
noncommutative complex analytic geometry.

A Banach algebra representation of the $q$-plane $\mathfrak{A}_{q}$ is given
by a complex Banach algebra $\mathcal{B}$ and a couple of its elements
$T,S\in\mathcal{B}$ such that $TS=q^{-1}ST$. The pair $\left(  T,S\right)  $
automatically defines a unique continuous algebra homomorphism $\mathcal{O}%
_{q}\left(  \mathbb{C}_{xy}\right)  \rightarrow\mathcal{B}$. The next central
of the paper on the functional calculus (see Theorem \ref{thFunCal}) asserts
that $\mathcal{O}_{q}\left(  \mathbb{C}_{xy}\right)  \rightarrow\mathcal{B}$
extends up to a continuous algebra homomorphism $\mathcal{O}_{q}\left(
U\right)  \rightarrow\mathcal{B}$ for a $\left(  \mathfrak{q},\mathfrak{d}%
\right)  $-open subset $U\subseteq\mathbb{C}_{xy}$ if and only if the
inclusions $\sigma\left(  T\right)  \subseteq U_{x}$ and $\sigma\left(
S\right)  \subseteq U_{y}$ hold for the spectra of $T$ and $S$ in the Banach
algebra $\mathcal{B}$. The union $\sigma\left(  T\right)  \cup\sigma\left(
S\right)  $ plays the role of the joint spectrum of the pair $\left(
T,S\right)  $. To be precise, assume that $\mathcal{B}$ is a unital Banach
algebra containing the pair $\left(  T,S\right)  $, which is commutative
modulo its Jacobson radical. That is the case if $\mathcal{B}$ is the closed
subalgebra of the algebra $\mathcal{B}\left(  \mathfrak{X}\right)  $ of all
bounded linear operators acting on a complex Banach space $\mathfrak{X}$
generated by $\left(  T,S\right)  $ such that $T$ or $S$ is a compact
operator. In this case, the joint Harte spectrum $\sigma_{\operatorname{H}%
}\left(  T,S\right)  $ of the pair $\left(  T,S\right)  $ in the Banach
algebra $\mathcal{B}$ is reduced to the union $\sigma\left(  T\right)
\cup\sigma\left(  S\right)  $, and the following spectral mapping theorem%
\[
\sigma_{\operatorname{H}}\left(  f\left(  T,S\right)  \right)  =f\left(
\sigma_{\operatorname{H}}\left(  T,S\right)  \right)
\]
holds for every tuple $f\left(  x,y\right)  $ of noncommutative holomorphic
functions from $\mathcal{O}_{q}\left(  U\right)  $ (see Corollary
\ref{corFCST}). The action of every $f_{i}\left(  x,y\right)  \in
\mathcal{O}_{q}\left(  U\right)  $ on $U$ is given by $\lambda\left(
f_{i}\left(  x,y\right)  \right)  $, $\lambda\in$ $\operatorname{Spec}\left(
\mathcal{O}_{q}\left(  U\right)  \right)  $. The related examples for
illustration are considered in Subsection \ref{subsecEx}.

\section{Preliminaries\label{sPre}}

All considered vector spaces are assumed to be complex. If $A$ is a complex
associative algebra then $A^{\operatorname{op}}$ denotes the same algebra $A$
but with the opposite multiplication $a\cdot^{\operatorname{op}}b=b\cdot a$,
$a,b\in A$. The unitization of an algebra $A$ is denoted by $A^{+}$, that is,
$A^{+}=A\oplus\mathbb{C}1$ with the standard multiplication.

\subsection{The index notations\label{subsecIN}}

The set of all positive integers is denoted by $\mathbb{N}$ whereas
$\mathbb{Z}_{+}$ denotes the set of all nonnegative integers. The set of all
$s$-tuples over $\mathbb{Z}_{+}$ is denoted by $\mathbb{Z}_{+}^{s}$. If
$I=\left(  i_{1},\ldots,i_{s}\right)  \in\mathbb{Z}_{+}^{s}$ with $i_{k}\geq1$
for all $k$, then we write $I\in\mathbb{N}^{s}$. If $I,J\in\mathbb{Z}_{+}^{s}$
then $\left\langle I,J\right\rangle $ denotes the integer $\sum_{k=1}^{s}%
i_{k}j_{k}$, and $\left\vert I\right\vert =\sum_{k=1}^{s}i_{k}$. The following
notations will be used below. For $I\in\mathbb{Z}_{+}^{s}$ and $1\leq t\leq
s-1$, we put
\[
I_{\left(  t\right)  }=\left(  i_{1},\ldots,i_{t-1}\right)  \in\mathbb{Z}%
_{+}^{t-1}\quad\text{and}\quad I^{\left(  t\right)  }=\left(  i_{t+1}%
,\ldots,i_{s}\right)  \in\mathbb{Z}_{+}^{s-t}%
\]
($I_{\left(  1\right)  }=0$). For every tuple $I\in\mathbb{Z}_{+}^{s}$ we also
set $I^{\ast}=\left(  \left\vert I^{\left(  1\right)  }\right\vert ,\left\vert
I^{\left(  2\right)  }\right\vert ,\ldots,\left\vert I^{\left(  s-1\right)
}\right\vert \right)  $ to be an $(s-1)$-tuple. Thus
\[
I^{\ast}=\left(  i_{2}+\cdots+i_{s},i_{3}+\cdots+i_{s},\ldots,i_{s-1}%
+i_{s},i_{s}\right)  \in\mathbb{Z}_{+}^{s-1}\text{.}%
\]
If $I=\left(  i_{1},\ldots,i_{s}\right)  $, $K=\left(  k_{1},\ldots
,k_{s}\right)  \in\mathbb{Z}_{+}^{s}$ and $a_{i_{t}k_{t}}\in\mathbb{C}$,
$1\leq t\leq s$, then for brevity, we write $a_{IK}$ instead of the product
$\prod\limits_{t=1}^{s}a_{i_{t}k_{t}}$.

\subsection{Arens-Michael envelope}

Let $A$ be a complete polynormed (or locally convex) algebra. If the topology
of $A$ is defined by means of a family of multiplicative seminorms, then $A$
is called an \textit{Arens-Michael algebra} \cite[1.2.4]{Hel}. Fix a
polynormed algebra $A$ with its separately continuous multiplication. The
\textit{Arens-Michael envelope} \cite[5.2.21]{Hel} of $A$ is called a pair
$\left(  \widetilde{A},\omega\right)  $ of an Arens-Michael algebra
$\widetilde{A}$, and a continuous algebra homomorphism $\omega:A\rightarrow
\widetilde{A}$ with the following \textquotedblright
universal-projective\textquotedblright\ property: for all Arens-Michael
algebra $\mathcal{B}$ and a continuous algebra homomorphism $\pi
:A\rightarrow\mathcal{B}$, there exists a unique continuous algebra
homomorphism $\widetilde{\pi}:\widetilde{A}\rightarrow\mathcal{B}$ such that
$\widetilde{\pi}\cdot\omega=\pi$. It turns out that an Arens-Michael algebra
is an inverse limit of some Banach algebras \cite[5.2.10]{Hel}. Therefore in
the latter universal projective property it can be assumed that all considered
algebras $\mathcal{B}$ are Banach algebras.

The set of all continuous characters of an Arens-Michael algebra $\mathcal{A}$
is denoted by $\operatorname{Spec}\left(  \mathcal{A}\right)  $. If
$\lambda\in\operatorname{Spec}\left(  \mathcal{A}\right)  $ then the algebra
homomorphism $\lambda:\mathcal{A}\rightarrow\mathbb{C}$ defines $\mathcal{A}%
$-module structure on $\mathbb{C}$ via pull back along $\lambda$. This module
called \textit{a trivial module} is denoted by $\mathbb{C}\left(
\lambda\right)  $.

The spectrum of an element $a$ in an associative algebra $A$ is denoted by
$\sigma\left(  a\right)  $. An element $a\in A$ is said to be quasinilpotent
if $\sigma\left(  a\right)  =\left\{  0\right\}  $. The set of all
quasinilpotent elements in $A$ is denoted by $Q\left(  A\right)  $. For the
Jacobson radical of $A$ we use the notation $\operatorname{Rad}A$.

\subsection{The $q$-topology and the disk topology of the complex
plane\label{subsecQT}}

We fix $q\in\mathbb{C}\backslash\left\{  0\right\}  $ with $\left\vert
q\right\vert <1$. A subset $S\subseteq\mathbb{C}$ is called a $q$%
\textit{-spiraling set }if it contains the origin and $\left\{  q^{n}%
x:n\in\mathbb{Z}_{+}\right\}  \subseteq S$ for every $x\in S$. Thus $S$ is a
$q$-spiraling set iff $S_{q}=S$, where $S_{q}=\left\{  0\right\}  \cup\left(
\cup_{n=1}^{\infty}q^{n}S\right)  $ is the $q$\textit{-hull of }$S$. If
$S=B\left(  a,r\right)  $ is an open disk centered at $a\in\mathbb{C}$ of
radius $r>0$, then $q^{n}S=B\left(  q^{n}a,\left\vert q\right\vert
^{n}r\right)  $ are family of disks in $\mathbb{C}$ obtained by rotating and
squeezing $S$ about the origin. For a real $q$, a $q$-spiraling set is a
star-shaped set centered at the origin. If $S=\left\{  x\right\}  $ is a
singleton, then $\left\{  x\right\}  _{q}$ is a spiraling sequence which tends
to zero including its limit point, that is, $\left\{  x\right\}  _{q}=\left\{
q^{n}x:n\in\mathbb{Z}_{+}\right\}  \cup\left\{  0\right\}  $ is a compact set.

An open subset $U\subseteq\mathbb{C}$ is said to be \textit{a }$q$%
\textit{-open set} if it is an open subset of $\mathbb{C}$ in the standard
topology, which is also a $q$-spiraling set. So is the whole plane
$\mathbb{C}$, and the empty set is assumed to be $q$-open set. The family of
all $q$-open subsets defines a new topology $\mathfrak{q}$ in $\mathbb{C}$,
which is weaker than the original standard topology of the complex plane.
Every open disk $B\left(  0,r\right)  $ centered at the origin is a $q$-open
set. Thus the neighborhood filter base of the origin is the same in both
$\mathfrak{q}$-topology and the standard topology.

Notice that $\left\{  0\right\}  $ is a generic point of the topological space
$\left(  \mathbb{C},\mathfrak{q}\right)  $ being dense in the whole plane. One
can easily prove that if $x\in\mathbb{C}\backslash\left\{  0\right\}  $ then
its closure in $\left(  \mathbb{C},\mathfrak{q}\right)  $ is given by
\[
\left\{  x\right\}  ^{-}=\left\{  q^{-k}x:k\in\mathbb{Z}_{+}\right\}  .
\]
Thus the topological space $\left(  \mathbb{C},\mathfrak{q}\right)  $
satisfies the axiom $T_{0}$, and it turns out to be an irreducible topological
space \cite[Ch. II, 4.1]{BurComA}. One can easily see that $\left(
\mathbb{C},\mathfrak{q}\right)  $ is not quasicompact, in particular, it is
not noetherian. Every $\mathfrak{q}$-open subset of $\left(  \mathbb{C}%
,\mathfrak{q}\right)  $ is $\mathfrak{q}$-connected (see \cite[Ch. II, 4.1,
Proposition 1]{BurComA}) automatically.

If $K\subseteq\mathbb{C}$ is a compact subset in the standard topology of
$\mathbb{C}$ then it turns out to be a quasicompact subset of $\left(
\mathbb{C},\mathfrak{q}\right)  $, but never a $\mathfrak{q}$-closed subset.
In particular, all disks (open or closed) centered at the origin are
quasicompact (nonclosed) subsets of $\left(  \mathbb{C},\mathfrak{q}\right)
$. They are all dense subsets of $\left(  \mathbb{C},\mathfrak{q}\right)  $.
Every closure $\left\{  x\right\}  ^{-}$ of a point $x\in\mathbb{C}$ is not quasicompact.

The family $\left\{  B\left(  0,r\right)  :r\in\mathbb{R}_{+}\right\}  $ of
all open disks in $\mathbb{C}$ centered at the origin defines a new topology
$\mathfrak{d}$ called \textit{the disk topology}. Since every $B\left(
0,r\right)  $ is $q$-open, the disk topology $\mathfrak{d}$ is weaker than
$\mathfrak{q}$, that is, $\mathfrak{d\preceq q}$. The closure $\left\{
x\right\}  ^{-}$ in $\left(  \mathbb{C},\mathfrak{d}\right)  $ of every
$x\in\mathbb{C}$ with $\left\vert x\right\vert =\rho\geq0$ is reduced to
$\mathbb{C}\backslash B\left(  0,\rho\right)  $. In particular, $\left\{
0\right\}  $ is a generic point in $\left(  \mathbb{C},\mathfrak{d}\right)  $ too.

\begin{lemma}
\label{lemQC}The disk topology $\mathfrak{d}$ and the $\mathfrak{q}$-topology
have the same neighborhood filter base at the origin. The quasicompact subsets
of $\left(  \mathbb{C},\mathfrak{d}\right)  $ are exactly the bounded subsets
of $\mathbb{C}$. A nonempty subset $K\subseteq\left(  \mathbb{C}%
,\mathfrak{q}\right)  $ is quasicompact iff so is its $q$-hull $K_{q}$. In
this case, $K$ is bounded automatically.
\end{lemma}

\begin{proof}
Since every $q$-open subset $U\subseteq\left(  \mathbb{C},\mathfrak{q}\right)
$ contains a small disk $B\left(  0,\varepsilon\right)  $, it follows that the
$\mathfrak{q}$-neighborhood filter base is weaker than the $\mathfrak{d}$-one,
that is, they are equal.

If $K\subseteq\left(  \mathbb{C},\mathfrak{q}\right)  \mathfrak{\ }$is a
quasicompact subset then it is quasicompact in $\left(  \mathbb{C}%
,\mathfrak{d}\right)  $, for the identity mapping $\left(  \mathbb{C}%
,\mathfrak{q}\right)  \rightarrow\left(  \mathbb{C},\mathfrak{d}\right)  $ is
continuous. But the family $\left\{  B\left(  0,n\right)  :n\in\mathbb{N}%
\right\}  $ of open disks in $\mathbb{C}$ is a $\mathfrak{d}$-open covering of
$\left(  \mathbb{C},\mathfrak{d}\right)  $. Since $K\subseteq\cup_{n}B\left(
0,n\right)  $, it follows that $K\subseteq B\left(  0,n\right)  $ for a large
$n$, that is, $K$ is bounded. Conversely, every bounded subset $K\subseteq
\left(  \mathbb{C},\mathfrak{d}\right)  $ is obviously quasicompact.

Now let $K\subseteq\mathbb{C}$ be a nonempty subset. If $K\subseteq\cup
_{i}U_{i}$ is an open covering in $\left(  \mathbb{C},\mathfrak{q}\right)  $,
and $z\in K$, then $z\in U_{i}$ for some $i$, and $\left\{  z\right\}
_{q}\subseteq U_{i}$, for $U_{i}$ is $q$-spiraling set. It follows that
$K_{q}\subseteq\cup_{i}U_{i}$ too. In particular, $K\subseteq\left(
\mathbb{C},\mathfrak{q}\right)  $ is quasicompact iff so is its $q$-hull
$K_{q}$.
\end{proof}

\begin{remark}
Not every bounded subset $K\subseteq\mathbb{C}$ is quasicompact in $\left(
\mathbb{C},\mathfrak{q}\right)  $ automatically. For example, suppose
$K=S^{1}-\left\{  1\right\}  $ is the unit circle with the removed point
$\left\{  1\right\}  $. There exists an infinite open covering of $K$ with
small disks $B\left(  z_{n},\varepsilon_{n}\right)  $, $n\in\mathbb{N}$, which
has no finite subcovering of $K$, for it is not a compact set in the standard
topology of $\mathbb{C}$. Put $U_{n}=B\left(  z_{n},\varepsilon_{n}\right)
_{q}$ to be $q$-open subsets covering $K$. One can easily seen that $K$ can
not covered by a finite of $U_{n}$ for a small $q$.
\end{remark}

We can consider the same topologies in the case of $\left\vert q\right\vert
>1$ too. One needs just to swap the poles $z\mapsto1/z$ in the Riemann sphere
and do the same.

Let $\mathcal{O}$ be the standard Fr\'{e}chet sheaf of stalks of the
holomorphic functions on $\mathbb{C}$ and let $\operatorname{id}%
:\mathbb{C\rightarrow}\left(  \mathbb{C},\mathfrak{q}\right)  $ be the
identity mapping, which is a continuous mapping between the standard complex
plane and $\left(  \mathbb{C},\mathfrak{q}\right)  $. The direct image
$\operatorname{id}_{\ast}\mathcal{O}$ of the sheaf $\mathcal{O}$ along the
identity mapping (see \cite[2.1]{Harts}) is denoted by $\mathcal{O}%
^{\mathfrak{q}}$. It turns out to be a Fr\'{e}chet $\widehat{\otimes}$-algebra
sheaf on the topological space $\left(  \mathbb{C},\mathfrak{q}\right)  $. For
every $q$-open set $U$ and its quasicompact subset $K\subseteq U$ we define
the related seminorm $\left\Vert f\right\Vert _{K}=\sup\left\vert f\left(
K\right)  \right\vert $, $f\in\mathcal{O}\left(  U\right)  $ on the algebra
$\mathcal{O}^{\mathfrak{q}}\left(  U\right)  $. By Lemma \ref{lemQC}, the
family $\left\{  \left\Vert \cdot\right\Vert _{K}\right\}  $ of seminorms over
all $q$-compact subsets $K\subseteq U$ (that is, $K=K_{q}$) defines the same
original Fr\'{e}chet topology of $\mathcal{O}\left(  U\right)  $, that is,
$\mathcal{O}^{\mathfrak{q}}\left(  U\right)  =\mathcal{O}\left(  U\right)  $
as the Fr\'{e}chet $\widehat{\otimes}$-algebras. But $\mathcal{O}%
^{\mathfrak{q}}$ and $\mathcal{O}$ are totally different sheaves having quite
different stalks (see below Lemma \ref{lemQC2}).

In a similar way, we can define the Fr\'{e}chet sheaf $\mathcal{O}%
^{\mathfrak{d}}$ as the direct image of $\mathcal{O}^{\mathfrak{q}}$ along the
identity (continuous) mapping $\left(  \mathbb{C},\mathfrak{q}\right)
\rightarrow\left(  \mathbb{C},\mathfrak{d}\right)  $. If $U=B\left(
0,r\right)  $ then $\mathcal{O}^{\mathfrak{d}}\left(  U\right)  =\mathcal{O}%
\left(  U\right)  $ as the Fr\'{e}chet $\widehat{\otimes}$-algebras.

\subsection{The Banach algebras $\mathcal{A}\left(  K\right)  $ and
$\mathcal{A}\left(  \rho\right)  $\label{subsecAKR}}

In this subsection we present some material of the mathematical folklore just
to prevent a possible discomfort for a reader.

Let $K\subseteq\mathbb{C}$ be a compact subset (in the standard topology). The
algebra of all stalks of the holomorphic functions on $K$ is denoted by
$\mathcal{O}\left(  K\right)  $. If $K$ is infinite, then one can identify
$\mathcal{O}\left(  K\right)  $ through the restriction map with the unital
subalgebra of the $C^{\ast}$-algebra $C\left(  K\right)  $ of all complex
continuous functions on $K$. The norm closure of $\mathcal{O}\left(  K\right)
$ in $C\left(  K\right)  $ is denoted by $\mathcal{A}\left(  K\right)  $.
Notice that $\left\Vert \cdot\right\Vert _{K}$ defines a norm on
$\mathcal{O}\left(  K\right)  $ by the uniqueness property of the holomorphic
functions, and $\mathcal{A}\left(  K\right)  $ is a Banach algebra, which is
the closure of the rational functions $\mathcal{R}\left(  K\right)  $ whose
poles are located outside of $K$. In particular, it contains every continuous
function on $K$ having a holomorphic extension on a neighborhood of $K$. Note
that $\mathcal{A}\left(  K\right)  \subseteq C\left(  K\right)  $ is an
inverse closed subalgebra and $\sigma\left(  z\right)  =K$ in $\mathcal{A}%
\left(  K\right)  $.

If $U\subseteq\mathbb{C}$ is an open subset and $\mathcal{R}\left(  U\right)
\subseteq\mathcal{O}\left(  U\right)  $ is the subalgebra of all rational
functions then $\mathcal{R}\left(  U\right)  =\underleftarrow{\lim}\left\{
\mathcal{R}\left(  K\right)  :K\subseteq U\right\}  $, which means that
$\mathcal{O}\left(  U\right)  $ is continuously embedded into the algebra
$\mathcal{A}\left(  U\right)  =\underleftarrow{\lim}\left\{  \mathcal{A}%
\left(  K\right)  :K\subseteq U\right\}  $. The Arens-Michael algebra
$\mathcal{A}\left(  U\right)  $ is in turn embedded into $C\left(  U\right)  $
by extending the canonical inclusion $\mathcal{O}\left(  U\right)  \subseteq
C\left(  U\right)  $. In the case of a compact disk $K=\mathbb{D}_{\rho}$ of
radius $\rho$ centered at the origin, along with $\mathcal{A}\left(  K\right)
$ we have also the Banach algebra
\[
\mathcal{A}\left(  \rho\right)  =\left\{  f=\sum_{n\in\mathbb{Z}_{+}}%
a_{n}z^{n}:\left\Vert f\right\Vert _{\rho}<\infty\right\}
\]
of all absolutely convergent series on $\mathbb{D}_{\rho}$ equipped with the
norm $\left\Vert f\right\Vert _{\rho}=\sum_{n\in\mathbb{Z}_{+}}\left\vert
a_{n}\right\vert \rho^{n}$. The identity mapping over polynomials is extended
up to a contractive homomorphism $\mathcal{A}\left(  \rho\right)
\rightarrow\mathcal{A}\left(  \mathbb{D}_{\rho}\right)  $ of the Banach algebras.

\begin{lemma}
\label{lemAK0}Let $U\subseteq\mathbb{C}$ be an open subset. Then
$\mathcal{O}\left(  U\right)  =\mathcal{A}\left(  U\right)  $ up to a
topological isomorphism. If $U=B\left(  0,r\right)  $ then the equality
\[
\mathcal{O}\left(  B\left(  0,r\right)  \right)  =\underleftarrow{\lim
}\left\{  \mathcal{A}\left(  \rho\right)  :\rho<r\right\}
\]
holds up to a topological isomorphism. In particular,
\[
\mathcal{O}\left(  U\right)  \widehat{\otimes}\mathcal{O}\left(  B\left(
0,r\right)  \right)  =\underleftarrow{\lim}\left\{  \mathcal{A}\left(
K\right)  \widehat{\otimes}\mathcal{A}\left(  \rho\right)  :K\subset
U,\rho<r\right\}  .
\]

\end{lemma}

\begin{proof}
First pick a bounded open subset $V\subseteq U$ with its compact closure
$K=V^{-}$ in the standard topology of $\mathbb{C}$, and $f\in\mathcal{A}%
\left(  U\right)  $, which turns out to be a continuous function on $U$. Since
$f|K\in\mathcal{A}\left(  K\right)  $, it follows that $f|K=\lim_{n}g_{n}$ for
a certain sequence $\left\{  g_{n}\right\}  \subseteq\mathcal{R}\left(
K\right)  $. But $\left\{  g_{n}|V\right\}  \subseteq\mathcal{R}\left(
V\right)  \subseteq\mathcal{O}\left(  V\right)  $ and $\left\Vert
f|\omega-g_{n}|\omega\right\Vert _{\omega}=\left\Vert f-g_{n}\right\Vert
_{\omega}\leq\left\Vert f|K-g_{n}\right\Vert _{K}\rightarrow0$ as
$n\rightarrow\infty$ for every compact subset $\omega\subset V$. It follows
that $f|V=\lim_{n}g_{n}|V$ in the Fr\'{e}chet space $\mathcal{O}\left(
V\right)  $, which means that $f|V\in\mathcal{O}\left(  V\right)  $. Hence
$f\in\mathcal{O}\left(  U\right)  $. Since the Hausdorff completion of the
seminormed space $\left(  \mathcal{R}\left(  U\right)  ,\left\Vert
\cdot\right\Vert _{K}\right)  $ is continuously included into $\mathcal{A}%
\left(  K\right)  $ for every infinite compact subset $K\subseteq U$, it
follows that $\mathcal{O}\left(  U\right)  =\mathcal{A}\left(  U\right)  $ up
to a topological isomorphism.

Now assume that $K=\mathbb{D}_{\rho}$ and $\varepsilon<\rho$. Since
$\left\Vert f\right\Vert _{K}\leq\left\Vert f\right\Vert _{\rho}$ for all
$f\in\mathcal{O}\left(  K\right)  $, it follows that $\mathcal{A}\left(
\rho\right)  \subseteq\mathcal{A}\left(  \mathbb{D}_{\rho}\right)  $ is a
contraction of Banach algebras. Since $\sigma\left(  z\right)  =\mathbb{D}%
_{\varepsilon}$ in the Banach algebra $\mathcal{A}\left(  \varepsilon\right)
$, there is a (holomorphic) functional calculus
\[
\mathcal{A}\left(  \mathbb{D}_{\rho}\right)  \rightarrow\mathcal{A}\left(
\varepsilon\right)  ,\quad f\mapsto f\left(  z\right)  =\dfrac{1}{2\pi i}%
\oint\limits_{\partial\mathbb{D}_{\rho}}f\left(  \lambda\right)  \left(
\lambda-z\right)  ^{-1}d\lambda,
\]
which extends the restriction map $\mathcal{O}\left(  \mathbb{D}_{\rho
}\right)  \rightarrow\mathcal{O}\left(  \mathbb{D}_{\varepsilon}\right)  $ due
to the Cauchy Integral Theorem. In particular, there are canonical continuous
maps $\mathcal{O}\left(  \mathbb{D}_{\rho}\right)  \hookrightarrow
\mathcal{A}\left(  \rho\right)  \hookrightarrow\mathcal{A}\left(
\mathbb{D}_{\rho}\right)  \rightarrow\mathcal{A}\left(  \varepsilon\right)
\hookrightarrow\mathcal{A}\left(  \mathbb{D}_{\varepsilon}\right)  $. It
follows that%
\[
\mathcal{O}\left(  B\left(  0,r\right)  \right)  =\underleftarrow{\lim
}\left\{  \mathcal{A}\left(  \mathbb{D}_{\rho}\right)  :\rho<r\right\}
=\underleftarrow{\lim}\left\{  \mathcal{A}\left(  \rho\right)  :\rho
<r\right\}
\]
up to a topological isomorphism of the Fr\'{e}chet $\widehat{\otimes}$-algebras.

Finally, the family of continuous linear maps $\mathcal{O}\left(  U\right)
\widehat{\otimes}\mathcal{O}\left(  B\left(  0,r\right)  \right)
\rightarrow\mathcal{A}\left(  K\right)  \widehat{\otimes}\mathcal{A}\left(
\rho\right)  $ generates a canonical continuous linear map $\pi:\mathcal{O}%
\left(  U\right)  \widehat{\otimes}\mathcal{O}\left(  B\left(  0,r\right)
\right)  \rightarrow\mathcal{F}$ of the Fr\'{e}chet spaces, where
$\mathcal{F=}\underleftarrow{\lim}\left\{  \mathcal{A}\left(  K\right)
\widehat{\otimes}\mathcal{A}\left(  \rho\right)  :K\subset U,\rho<r\right\}
$. But $\left\{  z^{n}:n\in\mathbb{Z}_{+}\right\}  $ is an absolute basis in
all algebras $\mathcal{O}\left(  B\left(  0,r\right)  \right)  $ and
$\mathcal{A}\left(  \rho\right)  $, $\rho<r$. If $f=\left(  f_{K,\rho}\right)
\in\mathcal{F}$ then every $f_{K,\rho}=\sum_{n}f_{K,n}\otimes z^{n}$ has a
unique absolutely convergent series expansion in $\mathcal{A}\left(  K\right)
\widehat{\otimes}\mathcal{A}\left(  \rho\right)  $. It follows that
$f_{n}=\left(  f_{K,n}\right)  _{K}\in\underleftarrow{\lim}\left\{
\mathcal{A}\left(  K\right)  \right\}  =\mathcal{O}\left(  U\right)  $ for
every $n$, and $g=\sum f_{n}\otimes z^{n}\in\mathcal{O}\left(  U\right)
\widehat{\otimes}\mathcal{O}\left(  B\left(  0,r\right)  \right)  $ converges
absolutely. Indeed,
\[
\left(  \left\Vert \cdot\right\Vert _{K}\otimes\left\Vert \cdot\right\Vert
_{\rho}\right)  \left(  g\right)  \leq\sum_{n}\left\Vert f_{n}\right\Vert
_{K}\rho^{n}=\sum_{n}\left\Vert f_{K,n}\right\Vert _{K}\rho^{n}<\infty
\]
for all $K$ and $\rho$, where $\left\Vert \cdot\right\Vert _{K}\otimes
\left\Vert \cdot\right\Vert _{\rho}$ is the projective tensor product of the
related seminorms. Since $\pi\left(  g\right)  =f$, we obtain that $\pi$ is a
continuous bijection of the Fr\'{e}chet spaces. By the Open Mapping Theorem,
we conclude that $\pi$ is a topological isomorphism.
\end{proof}

\section{Banach quantum planes\label{sBQP}}

Fix a complex number $q\in\mathbb{C}\backslash\left\{  0,1\right\}  $. A
unital associative algebra $A_{q}$ generated by two elements $x$ and $y$ with
the relation $xy=q^{-1}yx$ is called \textit{a quantum plane} or
$q$-\textit{plane} in our context. Thus a $q$-plane is a quotient of the
algebra $\mathfrak{A}_{q}$ (see Section \ref{secInt}) modulo its two-sided
ideal. Notice that $A_{q}^{\operatorname{op}}=A_{q^{-1}}$.

\subsection{Quantum planes}

If $A_{q}$ is a $q$-plane then we use the following notations $I_{x}$,
$I_{xy}$ and $I_{y}$ for the unital\ subalgebra in $A_{q}$ generated by $x$,
two sided ideal in $A_{q}$ generated by $xy$, and the subalgebra (without
unit) in $A_{q}$ generated by $y$, respectively. Thus $I_{x}=\left\{
\sum_{i\in\mathbb{Z}_{+}}a_{i}x^{i}:a_{i}\in\mathbb{C}\right\}  $,
$I_{xy}=\left\langle xy\right\rangle $ and $I_{y}=\left\{  \sum_{k\in
\mathbb{N}}a_{k}y^{k}:a_{k}\in\mathbb{C}\right\}  $. Let us notify that
$A_{q}$ is the linear span of all non-ordered monomials in $x$ and $y$, which
can be converted into the ordered ones.

\begin{lemma}
\label{lyx}If $A_{q}$ is a $q$-plane then $A_{q}=\operatorname{span}\left\{
x^{i}y^{k}:i,k\in\mathbb{Z}_{+}\right\}  $ is the linear span of all ordered
monomials. Moreover,
\[
I_{xy}=\left\{  \sum_{i,k\in\mathbb{N}}a_{ik}x^{i}y^{k}:a_{ik}\in
\mathbb{C}\right\}  \quad\text{and \quad}A_{q}=I_{x}+I_{xy}+I_{y}.
\]

\end{lemma}

\begin{proof}
Since $yx=qxy$, it follows that
\begin{equation}
y^{k}x^{i}=q^{ik}x^{i}y^{k} \label{yx}%
\end{equation}
for all $i,k\geq0$. In particular, each nonordered monomial in $A_{q}$ taken
by $x$ and $y$ can be converted into an ordered one. Hence $A_{q}%
=\operatorname{span}\left\{  x^{i}y^{k}:i,k\in\mathbb{Z}_{+}\right\}  $.
Further,%
\[
\sum_{i,k\in\mathbb{N}}a_{ik}x^{i}y^{k}=\sum_{i,k\in\mathbb{N}}a_{ik}%
x^{i-1}\left(  xy\right)  y^{k-1}\in I_{xy}%
\]
for all $a_{ik}\in\mathbb{C}$. Conversely, take $f\left(  xy\right)  g\in
I_{xy}$ for some polynomials $f,g\in A_{q}$. Using (\ref{yx}), we deduce that
$f\left(  xy\right)  g=\sum_{i,k\in\mathbb{N}}a_{ik}x^{i}y^{k}$ for some
$a_{ik}\in\mathbb{C}$. Thus $I_{xy}=\left\{  \sum_{i,k\in\mathbb{N}}%
a_{ik}x^{i}y^{k}:a_{ik}\in\mathbb{C}\right\}  $. Finally, for every
$f=\sum_{i,k}a_{ik}x^{i}y^{k}\in A_{q}$ we have
\[
f=\sum_{i\in\mathbb{Z}_{+}}a_{i0}x^{i}+\sum_{i,k\in\mathbb{N}}a_{ik}x^{i}%
y^{k}+\sum_{k\in\mathbb{N}}a_{0k}y^{k},
\]
which justifies the equality $A_{q}=I_{x}+I_{xy}+I_{y}$.
\end{proof}

\begin{lemma}
\label{lT}If $f=\sum_{i,k}a_{ik}x^{i}y^{k}\in A_{q}$ and $s\in\mathbb{N}$,
then
\[
f^{s}=\sum_{I,K\in\mathbb{Z}_{+}^{s}}a_{IK}q^{\left\langle I^{\ast},K_{\left(
s\right)  }\right\rangle }x^{\left\vert I\right\vert }y^{\left\vert
K\right\vert },
\]
where $K_{\left(  s\right)  }=\left(  k_{1},\ldots,k_{s-1}\right)
\in\mathbb{Z}_{+}^{s-1}$, $I^{\ast}=\left(  \left\vert I^{\left(  1\right)
}\right\vert ,\left\vert I^{\left(  2\right)  }\right\vert ,\ldots,\left\vert
I^{\left(  s-1\right)  }\right\vert \right)  \in\mathbb{Z}_{+}^{s-1}$,
$I^{\left(  t\right)  }=\left(  i_{t+1},\ldots,i_{s}\right)  \in\mathbb{Z}%
_{+}^{s-t}$, $1\leq t\leq s-1$ (see Subsection \ref{subsecIN}).
\end{lemma}

\begin{proof}
For $s=2$ we have
\begin{align*}
f^{2}  &  =\sum_{i_{1},i_{2},k_{1},k_{2}}a_{i_{1}k_{1}}a_{i_{2}k_{2}}x^{i_{1}%
}y^{k_{1}}x^{i_{2}}y^{k_{2}}=\sum_{i_{1},i_{2},k_{1},k_{2}}a_{i_{1}k_{1}%
}a_{i_{2}k_{2}}q^{i_{2}k_{1}}x^{i_{1}+i_{2}}y^{k_{1}+k_{2}}\\
&  =\sum_{I,K\in\mathbb{Z}_{+}^{2}}a_{IK}q^{\left\langle I^{\ast},K_{\left(
2\right)  }\right\rangle }x^{\left\vert I\right\vert }y^{\left\vert
K\right\vert }%
\end{align*}
thanks to (\ref{yx}). Using again (\ref{yx}), by induction on $s$, we deduce
that
\begin{align*}
f^{s}  &  =\sum_{I,K\in\mathbb{Z}_{+}^{s-1}}a_{IK}q^{\left\langle I^{\ast
},K_{\left(  s-1\right)  }\right\rangle }x^{\left\vert I\right\vert
}y^{\left\vert K\right\vert }f=\sum_{I,K\in\mathbb{Z}_{+}^{s}}a_{IK}%
q^{\left\langle I_{\left(  s\right)  }^{\ast},K_{\left(  s-1\right)
}\right\rangle }x^{\left\vert I_{\left(  s\right)  }\right\vert }y^{\left\vert
K_{\left(  s\right)  }\right\vert }x^{i_{s}}y^{k_{s}}\\
&  =\sum_{I,K\in\mathbb{Z}_{+}^{s}}a_{IK}q^{\left\langle I_{\left(  s\right)
}^{\ast},K_{\left(  s-1\right)  }\right\rangle }q^{i_{s}\left\vert K_{\left(
s\right)  }\right\vert }x^{\left\vert I_{\left(  s\right)  }\right\vert
+i_{s}}y^{\left\vert K_{\left(  s\right)  }\right\vert +k_{s}}\\
&  =\sum_{I,K\in\mathbb{Z}_{+}^{s}}a_{IK}q^{\left\langle I_{\left(  s\right)
}^{\ast},K_{\left(  s-1\right)  }\right\rangle +i_{s}\left\vert K_{\left(
s\right)  }\right\vert }x^{\left\vert I\right\vert }y^{\left\vert K\right\vert
}.
\end{align*}
Using the index notations from Subsection \ref{subsecIN}, we deduce that
\begin{align*}
\left\langle I_{\left(  s\right)  }^{\ast},K_{\left(  s-1\right)
}\right\rangle +i_{s}\left\vert K_{\left(  s\right)  }\right\vert  &
=\left\vert I_{\left(  s\right)  }^{\left(  1\right)  }\right\vert
k_{1}+\left\vert I_{\left(  s\right)  }^{\left(  2\right)  }\right\vert
k_{2}+\cdots+\left\vert I_{\left(  s\right)  }^{\left(  s-2\right)
}\right\vert k_{s-2}+i_{s}\left(  k_{1}+\cdots+k_{s-1}\right) \\
&  =\left(  \left\vert I_{\left(  s\right)  }^{\left(  1\right)  }\right\vert
+i_{s}\right)  k_{1}+\left(  \left\vert I_{\left(  s\right)  }^{\left(
2\right)  }\right\vert +i_{s}\right)  k_{2}+\cdots+\left(  \left\vert
I_{\left(  s\right)  }^{\left(  s-2\right)  }\right\vert +i_{s}\right)
k_{s-2}+i_{s}k_{s-1}\\
&  =\left\vert I^{\left(  1\right)  }\right\vert k_{1}+\left\vert I^{\left(
2\right)  }\right\vert k_{2}+\cdots+\left\vert I^{\left(  s-2\right)
}\right\vert k_{s-2}+i_{s}k_{s-1}\\
&  =\left\vert I^{\left(  1\right)  }\right\vert k_{1}+\left\vert I^{\left(
2\right)  }\right\vert k_{2}+\cdots+\left\vert I^{\left(  s-2\right)
}\right\vert k_{s-2}+\left\vert I^{\left(  s-1\right)  }\right\vert k_{s-1}\\
&  =\left\langle I^{\ast},K_{\left(  s\right)  }\right\rangle
\end{align*}
for all $I,K\in\mathbb{Z}_{+}^{s}$. Consequently, the formula for the powers
of $f$ holds.
\end{proof}

\subsection{The contractive Banach $q$-planes}

A Banach algebra norm completion of a quantum plane is called \textit{a Banach
quantum plane} or \textit{Banach }$q$\textit{-plane} whenever $q$ is fixed.
For the norm closures of $I_{x}$, $I_{xy}$ and $I_{y}$ in a Banach $q$-plane
$\mathcal{A}_{q}$ we use the notations $\mathcal{I}_{x}$, $\mathcal{I}_{xy}$
and $\mathcal{I}_{y}$, respectively. Thus $\mathcal{I}_{xy}$ is a closed two
sided ideal in $\mathcal{A}_{q}$ whereas $\mathcal{I}_{x}$ and $\mathcal{I}%
_{y}$ are its closed subalgebras. If $\left\vert q\right\vert <1$ then we say
that $\mathcal{A}_{q}$ is a \textit{contractive Banach }$q$-\textit{plane}.

\begin{lemma}
\label{lAq}Let $\mathcal{A}_{q}$ be a contractive Banach $q$-plane generated
by $x$ and $y$. If $f\in I_{xy}$ then $f$ is quasinilpotent in $\mathcal{A}%
_{q}$, that is, $I_{xy}\subseteq Q\left(  \mathcal{A}_{q}\right)  $.
\end{lemma}

\begin{proof}
Using Lemma \ref{lyx}, we conclude that $f=\sum_{i,k\in\mathbb{N}}a_{ik}%
x^{i}y^{k}$ is a polynomial in $A_{q}$. Put $\rho_{f}=\sum_{i,k\in\mathbb{N}%
}\left\vert a_{ik}\right\vert \rho^{i+k}$, where $\rho\geq\max\left\{
\left\Vert x\right\Vert ,\left\Vert y\right\Vert \right\}  $. By Lemma
\ref{lT}, we infer that
\[
f^{s}=\sum_{I,K\in\mathbb{N}^{s}}a_{IK}q^{\left\langle I^{\ast},K_{\left(
s\right)  }\right\rangle }x^{\left\vert I\right\vert }y^{\left\vert
K\right\vert }.
\]
It follows that $\left\Vert f^{s}\right\Vert \leq\sum_{I,K\in\mathbb{N}^{s}%
}\left\vert a_{IK}\right\vert \left\vert q\right\vert ^{\left\langle I^{\ast
},K_{\left(  s\right)  }\right\rangle }\rho^{\left\vert I\right\vert
+\left\vert K\right\vert }$. Since $\left\vert q\right\vert <1$, it follows
that
\[
\left\vert q\right\vert ^{\left\langle I^{\ast},K_{\left(  s\right)
}\right\rangle }=\left\vert q\right\vert ^{\left\vert I^{\left(  1\right)
}\right\vert k_{1}+\cdots+\left\vert I^{\left(  s-1\right)  }\right\vert
k_{s-1}}\leq\left\vert q\right\vert ^{\left(  s-1\right)  k_{1}+\left(
s-2\right)  k_{2}+\cdots+k_{s-1}}\leq\left\vert q\right\vert ^{\left(
s-1\right)  +\left(  s-2\right)  +\cdots+1}=\left\vert q\right\vert ^{s\left(
s-1\right)  /2}.
\]
It follows that
\[
\left\Vert f^{s}\right\Vert \leq\left\vert q\right\vert ^{s\left(  s-1\right)
/2}\sum_{I,K\in\mathbb{N}^{s}}\left\vert a_{IK}\right\vert \rho^{\left\vert
I\right\vert +\left\vert K\right\vert }=\left\vert q\right\vert ^{s\left(
s-1\right)  /2}\rho_{f}^{s}.
\]
Consequently,
\begin{equation}
\left\Vert f^{s}\right\Vert ^{1/s}\leq\left\vert q\right\vert ^{\left(
s-1\right)  /2}\rho_{f}. \label{Ts}%
\end{equation}
Whence $\inf\left\{  \left\Vert f^{s}\right\Vert ^{1/s}:s\in\mathbb{N}%
\right\}  =0$, that is, $f$ is a quasinilpotent element of $\mathcal{A}_{q}$.
\end{proof}

The fact $xy\in Q\left(  \mathcal{A}_{q}\right)  $ can also be deduced from
the simple algebraic equalities $\sigma\left(  xy\right)  \cup\left\{
0\right\}  =\sigma\left(  yx\right)  \cup\left\{  0\right\}  =q\sigma\left(
xy\right)  \cup\left\{  0\right\}  $ (see \cite[2.1.8]{Hel}) which means that
$\sigma\left(  xy\right)  =\left\{  0\right\}  $ whenever $\left\vert
q\right\vert \neq1$. Note that (\ref{Ts}) shows the decay rate of the powers
in $I_{xy}$. We don't know whether $\mathcal{I}_{xy}\subseteq Q\left(
\mathcal{A}_{q}\right)  $ (or $\mathcal{I}_{xy}\subseteq\operatorname{Rad}%
\mathcal{A}_{q}$) holds too. The following assertion demonstrates that the
problem is reduced to the property to be commutative modulo its Jacobson
radical of a contractive Banach quantum plane.

\begin{proposition}
\label{pequi}Let $\mathcal{A}_{q}$ be a contractive Banach $q$-plane. The
following assertions are equivalent:

$\left(  i\right)  $ $\mathcal{I}_{xy}\subseteq\operatorname{Rad}%
\mathcal{A}_{q}$;

$\left(  ii\right)  $ $\operatorname{Rad}\mathcal{A}_{q}=Q\left(
\mathcal{A}_{q}\right)  ;$

$\left(  iii\right)  $ $\mathcal{A}_{q}/\operatorname{Rad}\mathcal{A}_{q}$ is
a commutative Banach algebra.
\end{proposition}

\begin{proof}
$\left(  i\right)  \Rightarrow\left(  ii\right)  $ Assume that $\mathcal{I}%
_{xy}\subseteq\operatorname{Rad}\mathcal{A}_{q}$. Thus $\mathcal{I}_{xy}$ is a
closed two sided ideal of quasinilpotent elements. Therefore, if
$\pi:\mathcal{A}_{q}\rightarrow\mathcal{A}_{q}/\mathcal{I}_{xy}$, $\pi\left(
a\right)  =a^{\sim}$, is the quotient mapping, then $\sigma\left(  a^{\sim
}\right)  =\sigma\left(  a\right)  $ for all $a\in\mathcal{A}_{q}$. The Banach
algebra $\mathcal{A}_{q}/\mathcal{I}_{xy}$ is generated by the elements
$x^{\sim}$ and $y^{\sim}$. Since $x^{\sim}y^{\sim}=0^{\sim}$, it follows that
$\mathcal{A}_{q}/\mathcal{I}_{xy}$ is commutative. Therefore
$\operatorname{Rad}\mathcal{A}_{q}/\mathcal{I}_{xy}=Q\left(  \mathcal{A}%
_{q}/\mathcal{I}_{xy}\right)  $ \cite[2.1.34]{Hel}. Take $b\in Q\left(
\mathcal{A}_{q}\right)  $ and $a\in\mathcal{A}_{q}$. Then $\sigma\left(
b^{\sim}\right)  =\sigma\left(  b\right)  =\left\{  0\right\}  $, that is,
$b^{\sim}\in Q\left(  \mathcal{A}_{q}/\mathcal{I}_{xy}\right)  $. Therefore
$b^{\sim}\in\operatorname{Rad}\mathcal{A}_{q}/\mathcal{I}_{xy}$ and $a^{\sim
}b^{\sim}\in Q\left(  \mathcal{A}_{q}/\mathcal{I}_{xy}\right)  $. It follows
that $\sigma\left(  ab\right)  =\sigma\left(  a^{\sim}b^{\sim}\right)
=\left\{  0\right\}  $, that is, $ab\in Q\left(  \mathcal{A}_{q}\right)  $.
Hence $b\in\operatorname{Rad}\mathcal{A}_{q}$.

$\left(  ii\right)  \Rightarrow\left(  i\right)  $ Now assume that
$\operatorname{Rad}\mathcal{A}_{q}=Q\left(  \mathcal{A}_{q}\right)  $. By
Lemma \ref{lAq}, the inclusion $I_{xy}\subseteq Q\left(  \mathcal{A}%
_{q}\right)  $ holds, that is, $I_{xy}\subseteq\operatorname{Rad}%
\mathcal{A}_{q}$. It follows that $\mathcal{I}_{xy}\subseteq\operatorname{Rad}%
\mathcal{A}_{q}$.

$\left(  iii\right)  \Rightarrow\left(  ii\right)  $ Take $b\in Q\left(
\mathcal{A}_{q}\right)  $. Since $\mathcal{A}_{q}/\operatorname{Rad}%
\mathcal{A}_{q}$ is commutative, it follows that $a^{\sim}b^{\sim}\in Q\left(
\mathcal{A}_{q}/\operatorname{Rad}\mathcal{A}_{q}\right)  $ for each
$a\in\mathcal{A}_{q}$. Hence $ab\in Q\left(  \mathcal{A}_{q}\right)  $ for
each $a\in\mathcal{A}_{q}$. Thus $b\in\operatorname{Rad}\mathcal{A}_{q}$.

$\left(  i\right)  \Rightarrow\left(  iii\right)  $ Take $a,b\in
\mathcal{A}_{q}$. Being $\mathcal{A}_{q}/\mathcal{I}_{xy}$ a commutative
Banach algebra, we conclude that $ab-ba\in\mathcal{I}_{xy}\subseteq
\operatorname{Rad}\mathcal{A}_{q}$. Whence $\mathcal{A}_{q}/\operatorname{Rad}%
\mathcal{A}_{q}$ is commutative.
\end{proof}

\subsection{The $x$-inverse closed hull of a quantum plane}

As above let $A_{q}$ be a contractive $q$-plane in a Banach algebra
$\mathcal{B}$. Assume that the spectrum $\sigma\left(  x\right)  $ of $x$ in
$\mathcal{B}$ is a $q$-compact subset of $\mathbb{C}$ in the sense of that
$\sigma\left(  x\right)  _{q}=\sigma\left(  x\right)  $, where $\sigma\left(
x\right)  _{q}=\left\{  0\right\}  \cup\left(  \cup_{i=1}^{\infty}q^{i}%
\sigma\left(  x\right)  \right)  $ is the $q$-hull of the compact set
$\sigma\left(  x\right)  $ (see Subsection \ref{subsecQT}). There is a natural
holomorphic functional $\mathcal{O}\left(  \sigma\left(  x\right)  \right)
\rightarrow\mathcal{B}$, $f\left(  z\right)  \mapsto f\left(  x\right)  $ for
the element $x$, where $\mathcal{O}\left(  \sigma\left(  x\right)  \right)  $
is the algebra of all stalks of holomorphic functions on the spectrum
$\sigma\left(  x\right)  $. As above in Subsection \ref{subsecAKR}, the
subalgebra of stalks of all rational functions on the spectrum $\sigma\left(
x\right)  $ is denoted by $\mathcal{R}\left(  \sigma\left(  x\right)  \right)
$. The image $\mathcal{R}\left(  x\right)  $ of $\mathcal{R}\left(
\sigma\left(  x\right)  \right)  $ through the functional calculus is the
inverse closed subalgebra in $\mathcal{B}$ generated by $x$, whereas the image
$\mathcal{O}\left(  x\right)  $ of the holomorphic functional calculus is
contained in the closure of the inverse closed hull of $x$ in $\mathcal{B}$.
If $\lambda\notin\sigma\left(  x\right)  $ then $\left(  q^{i}z-\lambda
\right)  ^{-1}\in\mathcal{R}\left(  \sigma\left(  x\right)  \right)  $,
therefore $\left(  q^{i}x-\lambda\right)  ^{-1}\in\mathcal{B}$ for all
$i\geq0$. Note that $q^{-i}\lambda\notin\sigma\left(  x\right)  $, for
$\sigma\left(  x\right)  _{q}=\sigma\left(  x\right)  $.

The formula (\ref{yx}) can be generalized in the following way%
\begin{equation}
y^{k}x^{i}\left(  x-\lambda\right)  ^{-m}=q^{ik}x^{i}\left(  q^{k}%
x-\lambda\right)  ^{-m}y^{k} \label{yxm}%
\end{equation}
Indeed, since $y\left(  x-\mu\right)  =\left(  qx-\mu\right)  y$, $\mu
\notin\sigma\left(  x\right)  $, and $\left(  qx-\mu\right)  ^{-1}%
\in\mathcal{B}$, it follows that $y\left(  x-\mu\right)  ^{-1}=\left(
qx-\mu\right)  ^{-1}y$. By iterating, we obtain that $y\left(  x-\mu\right)
^{-m}=\left(  qx-\mu\right)  ^{-m}y$ for all $m\geq1$. Using (\ref{yx}), we
proceed by induction on the pairs $\left(  k,m\right)  $. Namely, taking into
account that $\mu=q^{1-k}\lambda\notin\sigma\left(  x\right)  $, we deduce
that
\begin{align*}
y^{k}x^{i}\left(  x-\lambda\right)  ^{-m}  &  =q^{ik}x^{i}y^{k}\left(
x-\lambda\right)  ^{-m}=q^{ik}x^{i}y\left(  q^{k-1}x-\lambda\right)
^{-m}y^{k-1}\\
&  =q^{ik}x^{i}q^{\left(  1-k\right)  m}y\left(  x-q^{1-k}\lambda\right)
^{-m}y^{k-1}=q^{ik}x^{i}q^{\left(  1-k\right)  m}\left(  qx-q^{1-k}%
\lambda\right)  ^{-m}y^{k}\\
&  =q^{ik}x^{i}\left(  q^{k}x-\lambda\right)  ^{-m}y^{k},
\end{align*}
that is, the formula (\ref{yxm}) holds.

Now consider the $x$-inverse closed subalgebra $A_{q,x}\subseteq\mathcal{B}$
generated by the $q$-plane $A_{q}$. That is the subalgebra in $\mathcal{B}$
generated by $\mathcal{R}\left(  x\right)  $ and $y$. Using Lemma \ref{lyx}
and (\ref{yxm}), and the fact that every rational function can be decomposed
into a sum of simple ratios, we obtain that%
\[
A_{q,x}=\operatorname{span}\left\{  x^{i}\left(  x-\lambda\right)  ^{-m}%
y^{k}:i,m,k\in\mathbb{Z}_{+},\lambda\notin\sigma\left(  x\right)  \right\}
=I_{x}+I_{xy}+I_{y},
\]
where $I_{x}=\mathcal{R}\left(  x\right)  $ is the inverse closed subalgebra
in $\mathcal{B}$ generated by $x$, $I_{y}$ is the (non-unital) subalgebra in
$\mathcal{B}$ generated by $y$, and
\begin{align*}
I_{xy}  &  =\left\{  \sum_{k\in\mathbb{N}}r_{k}\left(  x\right)  y^{k}%
:r_{k}\left(  z\right)  \in\mathcal{R}\left(  \sigma\left(  x\right)  \right)
,r_{k}\left(  0\right)  =0\right\} \\
&  =\left\{  \sum_{m\in\mathbb{Z}_{+}}\sum_{i,k\in\mathbb{N}}a_{imk}%
x^{i}\left(  x-\lambda\right)  ^{-m}y^{k}:a_{imk}\in\mathbb{C}\right\}
\end{align*}
is the two-sided ideal in $A_{q,x}$ generated by $xy$. Notice that if
$r\left(  0\right)  \neq0$ for some $r\left(  z\right)  \in\mathcal{R}\left(
\sigma\left(  x\right)  \right)  $, then $r\left(  x\right)  y^{k}=\left(
r\left(  x\right)  -r\left(  0\right)  \right)  y^{k}+r\left(  0\right)
y^{k}\in I_{xy}+I_{y}$ for all $k\geq1$.

The norm closure of $A_{q,x}$ in $\mathcal{B}$ denoted by $\mathcal{A}_{q,x}$
is a unital Banach algebra.

\begin{lemma}
\label{lemxIC1}If $f\in I_{xy}$ then $f$ is quasinilpotent in $\mathcal{A}%
_{q,x}$, that is, $I_{xy}\subseteq Q\left(  \mathcal{A}_{q,x}\right)  $.
\end{lemma}

\begin{proof}
First note that if $f=\sum_{i,m,k}a_{imk}x^{i}\left(  x-\lambda_{m}\right)
^{-m}y^{k}\in A_{q,x}$ with $\left\{  \lambda_{m}\right\}  \cap\sigma\left(
x\right)  =\varnothing$ and $\left\{  a_{imk}\right\}  \subseteq\mathbb{C}$,
then using (\ref{yxm}), we deduce that%
\begin{align*}
f^{2}  &  =\sum_{i_{1},i_{2},k_{1},k_{2}}a_{i_{1}m_{1}k_{1}}a_{i_{2}m_{2}%
k_{2}}x^{i_{1}}\left(  x-\lambda_{m_{1}}\right)  ^{-m_{1}}y^{k_{1}}x^{i_{2}%
}\left(  x-\lambda_{m_{2}}\right)  ^{-m_{2}}y^{k_{2}}\\
&  =\sum_{i_{1},i_{2},k_{1},k_{2}}a_{i_{1}m_{1}k_{1}}a_{i_{2}m_{2}k_{2}%
}q^{i_{2}k_{1}}x^{i_{1}+i_{2}}\left(  x-\lambda_{m_{1}}\right)  ^{-m_{1}%
}\left(  q^{k_{1}}x-\lambda_{m_{2}}\right)  ^{-m_{2}}y^{k_{1}+k_{2}}.
\end{align*}
As in the proof of Lemma \ref{lT}, we derive that
\[
f^{s}=\sum_{I,M,K\in\mathbb{Z}_{+}^{s}}a_{IMK}q^{\left\langle I^{\ast
},K_{\left(  s\right)  }\right\rangle }x^{\left\vert I\right\vert }%
\prod\limits_{t=1}^{s}\left(  q^{\left\vert K_{\left(  t\right)  }\right\vert
}x-\lambda_{m_{t}}\right)  ^{-m_{t}}y^{\left\vert K\right\vert },
\]
where $K_{\left(  t\right)  }=\left(  k_{1},\ldots,k_{t-1}\right)
\in\mathbb{Z}_{+}^{t-1}$, $I^{\ast}=\left(  \left\vert I^{\left(  1\right)
}\right\vert ,\left\vert I^{\left(  2\right)  }\right\vert ,\ldots,\left\vert
I^{\left(  s-1\right)  }\right\vert \right)  \in\mathbb{Z}_{+}^{s-1}$,
$I^{\left(  t\right)  }=\left(  i_{t+1},\ldots,i_{s}\right)  \in\mathbb{Z}%
_{+}^{s-t}$, $1\leq t\leq s-1$ are the same index notations from Subsection
\ref{subsecIN}. Since $0\in\sigma\left(  x\right)  $, it follows that
$\min\left\{  \left\vert \lambda_{m}\right\vert \right\}  \geq\varepsilon>0$
for some $0<\varepsilon\leq1$, and $\left\Vert q^{l}x\right\Vert
\leq\varepsilon/2$ for all large $l\geq n$. Let $F$ be the set of all tuples
$J=\left(  j_{1},\ldots,j_{p}\right)  $ of positive integers so that
$\left\vert J\right\vert <n$, and put
\[
C_{f,x}=1\vee\max_{J\in F}\left\Vert \prod\limits_{v=1}^{p}\left(
q^{\left\vert J_{\left(  v\right)  }\right\vert }x-\lambda_{m_{v}}\right)
^{-m_{v}}\right\Vert
\]
to be a positive constant, which depends on $f$ and $x$ only. Then
\begin{align*}
\left\Vert \prod\limits_{t=1}^{s}\left(  q^{\left\vert K_{\left(  t\right)
}\right\vert }x-\lambda_{m_{t}}\right)  ^{-m_{t}}\right\Vert  &  \leq
C_{f,x}\left\Vert \prod\limits_{\left\vert K_{\left(  t\right)  }\right\vert
\geq n}^{s}\left(  q^{\left\vert K_{\left(  t\right)  }\right\vert }%
x-\lambda_{m_{t}}\right)  ^{-m_{t}}\right\Vert \\
&  \leq C_{f,x}\prod\limits_{\left\vert K_{\left(  t\right)  }\right\vert \geq
n}^{s}\left\vert \lambda_{m_{t}}\right\vert ^{-m_{t}}\left(  1-\left\Vert
\lambda_{m_{t}}^{-1}q^{\left\vert K_{\left(  t\right)  }\right\vert
}x\right\Vert \right)  ^{-m_{t}}\\
&  \leq C_{f,x}\left(  2/\varepsilon\right)  ^{\left\vert M\right\vert }.
\end{align*}
Now assume that $f\in I_{xy}$. Using Lemma \ref{lAq} (see to the proof), we
conclude that%
\[
\left\Vert f^{s}\right\Vert \leq C_{f,x}\left\vert q\right\vert ^{s\left(
s-1\right)  /2}\sum_{M}\sum_{I,K\in\mathbb{N}^{s}}\left\vert a_{IMK}%
\right\vert \rho^{\left\vert I\right\vert +\left\vert M\right\vert +\left\vert
K\right\vert }=C_{f,x}\left\vert q\right\vert ^{s\left(  s-1\right)  /2}%
\rho_{f}^{s},
\]
where $\rho_{f}=\sum_{m}\sum_{i,k\in\mathbb{N}}\left\vert a_{imk}\right\vert
\rho^{i+k+m}$ with $\rho\geq\max\left\{  \left\Vert x\right\Vert ,\left\Vert
y\right\Vert ,2/\varepsilon\right\}  $. Hence
\[
\left\Vert f^{s}\right\Vert ^{1/s}\leq C_{f,x}^{1/s}\left\vert q\right\vert
^{\left(  s-1\right)  /2}\rho_{f},
\]
which means that $f\in Q\left(  \mathcal{A}_{q,x}\right)  $.
\end{proof}

As above we use the notations $\mathcal{I}_{x}$, $\mathcal{I}_{xy}$ and
$\mathcal{I}_{y}$ for the related norm closures of $I_{x}$, $I_{xy}$ and
$I_{y}$ in the Banach algebra $\mathcal{A}_{q,x}$. Thus $\mathcal{I}%
_{x}=\mathcal{R}\left(  x\right)  ^{-}$ is the inverse closed subalgebra in
$\mathcal{B}$ generated by $x$, $\mathcal{I}_{xy}$ is a closed two sided ideal
in $\mathcal{A}_{q,x}$, and $\mathcal{I}_{y}$ is the closed (non-unital)
subalgebra generated by $y$. The statement of Lemma \ref{lemxIC1} can be
generalized in the following way.

\begin{proposition}
\label{propTQA}Let $A_{q}$ be a contractive $q$-plane in a Banach algebra
$\mathcal{B}$ with the $q$-compact spectrum $\sigma\left(  x\right)  $,
$U\subseteq\mathbb{C}$ a $q$-open subset containing $\sigma\left(  x\right)  $
whose topological boundary $\partial U$ consists of a finite union of
piecewise smooth curves, $K=U^{-}$ the closure of $U$ in the standard topology
of $\mathbb{C}$, and let $f=\sum_{k\in\mathbb{N}}f_{k}\left(  x\right)  y^{k}$
be an element of $\mathcal{I}_{xy}$ in $\mathcal{A}_{q,x}$ with $\left\{
f_{k}\right\}  \subseteq\mathcal{O}\left(  K\right)  $, $f_{k}\left(
0\right)  =0$ and $\sum_{k}\left\Vert f_{k}\right\Vert _{K}\left\Vert
y\right\Vert ^{k}<\infty$. Then $f\in Q\left(  \mathcal{A}_{q,x}\right)  $.
\end{proposition}

\begin{proof}
Recall that $\mathcal{A}\left(  K\right)  $ is the closure of the algebra
$\mathcal{O}\left(  K\right)  $ in $C\left(  K\right)  $ (see Subsection
\ref{subsecQT}). Since $\sigma\left(  x\right)  \subseteq U\subseteq U^{-}=K$,
it follows that the functional calculus
\[
\mathcal{A}\left(  K\right)  \rightarrow\mathcal{B},\quad g\mapsto g\left(
x\right)  =\dfrac{1}{2\pi i}\oint\limits_{\partial U}g\left(  \lambda\right)
\left(  \lambda-x\right)  ^{-1}d\lambda
\]
turns out to be a unital bounded homomorphism (see Lemma \ref{lemAK0}) with
its norm $C>0$. Using (\ref{yxm}), we derive that
\[
y^{k}g\left(  x\right)  =\dfrac{1}{2\pi i}\oint\limits_{\partial U}g\left(
\lambda\right)  y^{k}\left(  \lambda-x\right)  ^{-1}d\lambda=\dfrac{1}{2\pi
i}\oint\limits_{\partial U}g\left(  \lambda\right)  \left(  \lambda
-q^{k}x\right)  ^{-1}y^{k}d\lambda=g\left(  q^{k}x\right)  y^{k}%
\]
for all $g\in\mathcal{A}\left(  K\right)  $. Note that $\sigma\left(
q^{k}x\right)  =q^{k}\sigma\left(  x\right)  \subseteq\sigma\left(  x\right)
_{q}=\sigma\left(  x\right)  \subseteq U$, and the functional calculus holds
for $q^{k}x$ too. Actually, $g\left(  q^{k}x\right)  $ is the result of the
functional calculus of $x$ applied to the function $g\left(  q^{k}z\right)  $
from $\mathcal{A}\left(  K\right)  $ (recall that $U$ is $q$-open). Further,
$\left\Vert f\right\Vert \leq\sum_{k\in\mathbb{N}}\left\Vert f_{k}\left(
x\right)  \right\Vert \left\Vert y\right\Vert ^{k}\leq C\sum_{k}\left\Vert
f_{k}\right\Vert _{K}\left\Vert y\right\Vert ^{k}<\infty$, that is,
$f=\sum_{k\in\mathbb{N}}f_{k}\left(  x\right)  y^{k}$ is an absolutely
convergent series in $\mathcal{B}$. As in the proof of Lemma \ref{lemxIC1}, we
have
\begin{align*}
f^{s}  &  =\sum f_{k_{1}}\left(  x\right)  f_{k_{2}}\left(  q^{k_{1}}x\right)
f_{k_{3}}\left(  q^{k_{1}+k_{2}}x\right)  \cdots f_{k_{s}}\left(
q^{k_{1}+\cdots+k_{s-1}}x\right)  y^{k_{1}+\cdots+k_{s}}\\
&  =\sum_{K}q^{\left(  s-1\right)  k_{1}+\left(  s-2\right)  k_{2}%
+\cdots+k_{s-1}}x^{s}g_{k_{1}}\left(  x\right)  g_{k_{2}}\left(  q^{k_{1}%
}x\right)  g_{k_{3}}\left(  q^{k_{1}+k_{2}}x\right)  \cdots g_{k_{s}}\left(
q^{k_{1}+\cdots+k_{s}}x\right)  y^{k_{1}+\cdots+k_{s}},
\end{align*}
where $f_{k}\left(  z\right)  =zg_{k}\left(  z\right)  $ with $g_{k}\left(
z\right)  \in\mathcal{O}\left(  K\right)  $. By assumption, $\mathbb{D}%
_{\varepsilon}\subseteq U$ for a small $\varepsilon>0$. But $\left\Vert
g_{k}\right\Vert _{\mathbb{D}_{\varepsilon}}=\left\vert g_{k}\left(
z_{k}\right)  \right\vert $ for some $z_{k}\in\mathbb{D}_{\varepsilon}$ with
$\left\vert z_{k}\right\vert =\varepsilon$. Therefore $\left\Vert
g_{k}\right\Vert _{\mathbb{D}_{\varepsilon}}=\varepsilon\left\vert
g_{k}\left(  z_{k}\right)  \right\vert \varepsilon^{-1}=\left\vert
f_{k}\left(  z_{k}\right)  \right\vert \varepsilon^{-1}\leq\left\Vert
f_{k}\right\Vert _{K}\varepsilon^{-1}$. If $z\in K\backslash\mathbb{D}%
_{\varepsilon}$ then $\left\vert g_{k}\left(  z\right)  \right\vert
=\left\vert f_{k}\left(  z\right)  \right\vert \left\vert z\right\vert
^{-1}\leq\left\vert f_{k}\left(  z\right)  \right\vert \varepsilon^{-1}%
\leq\left\Vert f_{k}\right\Vert _{K}\varepsilon^{-1}$. Hence $\left\Vert
g_{k}\right\Vert _{K}\leq\left\Vert f_{k}\right\Vert _{K}\varepsilon^{-1}$ for
all $k$. It follows that
\[
\left\Vert f^{s}\right\Vert \leq\left\Vert x\right\Vert ^{s}\left\vert
q\right\vert ^{s\left(  s-1\right)  /2}C^{s}\sum\left\Vert g_{k_{1}%
}\right\Vert _{K}\cdots\left\Vert g_{k_{s}}\right\Vert _{K}\left\Vert
y\right\Vert ^{k_{1}+\cdots+k_{s}}=C^{s}\left\Vert x\right\Vert ^{s}\left\vert
q\right\vert ^{s\left(  s-1\right)  /2}\varepsilon^{-s}\rho_{f}^{s},
\]
where $\rho_{f}=\sum_{k}\left\Vert f_{k}\right\Vert _{K}\left\Vert
y\right\Vert ^{k}$. In particular,
\[
\left\Vert f^{s}\right\Vert ^{1/s}\leq C\left\Vert x\right\Vert \left\vert
q\right\vert ^{\left(  s-1\right)  /2}\varepsilon^{-1}\rho_{f},
\]
which means that $f\in Q\left(  \mathcal{A}_{q,x}\right)  $.
\end{proof}

\subsection{The Banach $q$-plane $\mathcal{A}_{q}\left(  \rho\right)
$\label{subsecBqA}}

Now let $\mathcal{A}_{q}$ be a Banach $q$-plane with $\left\vert q\right\vert
\leq1$, and let $\rho$ be a positive real with $\rho\geq\max\left\{
\left\Vert x\right\Vert ,\left\Vert y\right\Vert \right\}  $. Suppose that the
correspondence
\[
f=\sum_{i,k\in\mathbb{Z}_{+}}a_{ik}x^{i}y^{k}\mapsto\left\Vert f\right\Vert
_{\rho}=\sum_{i,k\in\mathbb{Z}_{+}}\left\vert a_{ik}\right\vert \rho^{i+k}%
\]
(see to the proof of Lemma \ref{lAq}) defines a continuous mapping over all
polynomials $A_{q}$ of the Banach $q$-plane $\mathcal{A}_{q}$. Then
$\left\Vert \cdot\right\Vert _{\rho}$ defines a continuous seminorm on $A_{q}$
such that $\left\Vert f\right\Vert \leq\left\Vert f\right\Vert _{\rho}\leq
C\left\Vert f\right\Vert $, $f\in A_{q}$ for some $C>0$. It turns out that
$A_{q}=\mathfrak{A}_{q}$ and $\left\Vert \cdot\right\Vert _{\rho}$ defines a
norm on $\mathcal{A}_{q}$ which is equivalent to the original one. We denote
this Banach $q$-plane by $\mathcal{A}_{q}\left(  \rho\right)  $. Actually,
$\left\Vert \cdot\right\Vert _{\rho}$ is multiplicative \cite[Section
5.3]{Pir}, and
\[
\mathcal{A}_{q}\left(  \rho\right)  =\left\{  f=\sum_{i,k\in\mathbb{Z}_{+}%
}a_{ik}x^{i}y^{k}:\left\Vert f\right\Vert _{\rho}=\sum_{i,k\in\mathbb{Z}_{+}%
}\left\vert a_{ik}\right\vert \rho^{i+k}<\infty\right\}
\]
is the Banach space completion of the quantum plane $\mathfrak{A}_{q}$
($\left\vert q\right\vert \leq1$) (see \cite[Corollary 5.14]{Pir}). If
$\left\vert q\right\vert \geq1$ we obtain (see \cite{Pir}) the following
Banach $q$-plane%
\[
\mathcal{A}_{q}\left(  \rho\right)  =\left\{  f=\sum_{i,k\in\mathbb{Z}_{+}%
}a_{ik}x^{i}y^{k}:\left\Vert f\right\Vert _{\rho}=\sum_{i,k\in\mathbb{Z}_{+}%
}\left\vert a_{ik}\right\vert \left\vert q\right\vert ^{-ik}\rho^{i+k}%
<\infty\right\}  .
\]
The ideal $\mathcal{I}_{xy}$ and subalgebras $\mathcal{I}_{x}$, $\mathcal{I}%
_{y}$ in the algebra $\mathcal{A}_{q}\left(  \rho\right)  $ introduced above
are denoted by $\mathcal{I}_{xy}^{\left(  \rho\right)  }$ and $\mathcal{I}%
_{x}^{\left(  \rho\right)  }$, $\mathcal{I}_{y}^{\left(  \rho\right)  }$, respectively.

\begin{proposition}
\label{pAfree}The Banach $q$-plane $\mathcal{A}_{q}\left(  \rho\right)  $
admits the following decomposition%
\[
\mathcal{A}_{q}\left(  \rho\right)  =\mathcal{I}_{x}^{\left(  \rho\right)
}\oplus\mathcal{I}_{xy}^{\left(  \rho\right)  }\oplus\mathcal{I}_{y}^{\left(
\rho\right)  }\quad\text{and\quad}\mathcal{I}_{xy}^{\left(  \rho\right)
}=\cap\left\{  \ker\left(  \lambda\right)  :\lambda\in\operatorname{Spec}%
\left(  \mathcal{A}_{q}\left(  \rho\right)  \right)  \right\}  .
\]
If $\left\vert q\right\vert \neq1$ then $\mathcal{I}_{xy}^{\left(
\rho\right)  }=\operatorname{Rad}\mathcal{A}_{q}\left(  \rho\right)  $ and
$\mathcal{A}_{q}\left(  \rho\right)  $ is commutative modulo its Jacobson
radical. If $\left\vert q\right\vert =1$ then $Q\left(  \mathcal{A}_{q}\left(
\rho\right)  \right)  \subseteq\mathcal{I}_{xy}^{\left(  \rho\right)  }$ and
$x^{i}y^{k}\in\mathcal{I}_{xy}^{\left(  \rho\right)  }\backslash Q\left(
\mathcal{A}_{q}\left(  \rho\right)  \right)  $ for all $i,k\in\mathbb{N}$.
\end{proposition}

\begin{proof}
First, note that $\mathcal{I}_{xy}^{\left(  \rho\right)  }=\left\{
f=\sum_{i,k\in\mathbb{N}}a_{ik}x^{i}y^{k}:\left\Vert f\right\Vert _{\rho
}<\infty\right\}  $. Indeed, the indicated set in the right hand side is a
closed two sided ideal in $\mathcal{A}_{q}\left(  \rho\right)  $ containing
$I_{xy}$. Moreover, each series like $\sum_{i,k\in\mathbb{N}}a_{ik}x^{i}y^{k}$
belongs to $\mathcal{I}_{xy}^{\left(  \rho\right)  }$, therefore the equality
follows. Since we deal with the absolutely convergent series, it follows that
$\mathcal{A}_{q}\left(  \rho\right)  =\mathcal{I}_{x}^{\left(  \rho\right)
}\oplus\mathcal{I}_{xy}^{\left(  \rho\right)  }\oplus\mathcal{I}_{y}^{\left(
\rho\right)  }$ is the direct $\ell^{1}$-sum (see Lemma \ref{lyx}) with the
closed subalgebras%
\[
\mathcal{I}_{x}^{\left(  \rho\right)  }=\left\{  f=\sum_{i\in\mathbb{Z}_{+}%
}a_{i}x^{i}:\sum_{i\in\mathbb{Z}_{+}}\left\vert a_{i}\right\vert \rho
^{i}<\infty\right\}  \quad\text{and\quad}\mathcal{I}_{y}^{\left(  \rho\right)
}=\left\{  g=\sum_{k\in\mathbb{N}}a_{k}y^{k}:\sum_{k\in\mathbb{N}}\left\vert
a_{k}\right\vert \rho^{k}<\infty\right\}  .
\]
If $\lambda\in\operatorname{Spec}\left(  \mathcal{A}_{q}\left(  \rho\right)
\right)  $ then $\left(  1-q^{-1}\right)  \lambda\left(  x\right)
\lambda\left(  y\right)  =\lambda\left(  xy-q^{-1}yx\right)  =0$, that is,
$\lambda\left(  xy\right)  =0$. Thus $\mathcal{I}_{xy}^{\left(  \rho\right)
}\subseteq\ker\left(  \lambda\right)  $. Consider the quotient algebra
$\mathcal{A}_{q}\left(  \rho\right)  /\mathcal{I}_{xy}^{\left(  \rho\right)
}$ and the related quotient mapping $\pi:\mathcal{A}_{q}\left(  \rho\right)
\rightarrow\mathcal{A}_{q}\left(  \rho\right)  /\mathcal{I}_{xy}^{\left(
\rho\right)  }$, $\pi\left(  g\right)  =g^{\sim}$. The commutative Banach
algebra $\mathcal{A}_{q}\left(  \rho\right)  /\mathcal{I}_{xy}^{\left(
\rho\right)  }$ is decomposed into the direct sum $\mathcal{A}_{q}\left(
\rho\right)  /\mathcal{I}_{xy}^{\left(  \rho\right)  }=\mathcal{I}_{x^{\sim}%
}^{\left(  \rho\right)  }\oplus\mathcal{I}_{y^{\sim}}^{\left(  \rho\right)  }$
of its closed subalgebras such that $x^{\sim}y^{\sim}=0^{\sim}$. Moreover,
$\mathcal{I}_{x^{\sim}}^{\left(  \rho\right)  }$ and $\mathcal{I}_{y^{\sim}%
}^{\left(  \rho\right)  }$ are identified with the subalgebras $\mathcal{I}%
_{x}^{\left(  \rho\right)  }$ and $\mathcal{I}_{y}^{\left(  \rho\right)  }$,
respectively. It follows that $\operatorname{Spec}\left(  \mathcal{A}%
_{q}\left(  \rho\right)  \right)  =\operatorname{Spec}\left(  \mathcal{A}%
_{q}\left(  \rho\right)  /\mathcal{I}_{xy}^{\left(  \rho\right)  }\right)  $
and for each $\lambda\in\operatorname{Spec}\left(  \mathcal{A}_{q}\left(
\rho\right)  /\mathcal{I}_{xy}^{\left(  \rho\right)  }\right)  $ we have
either $\lambda\left(  x^{\sim}\right)  =0$ or $\lambda\left(  y^{\sim
}\right)  =0$. Note that $\mathcal{I}_{x^{\sim}}^{\left(  \rho\right)
}=\mathcal{A}\left(  \rho\right)  $ and $\mathcal{I}_{y^{\sim}}^{\left(
\rho\right)  }\oplus\mathbb{C}=\mathcal{A}\left(  \rho\right)  $ are in turn
identified with the Banach algebras of holomorphic functions on the closed
disk $\mathbb{D}_{\rho}$ (see Subsection \ref{subsecAKR}). Therefore they do
not have nontrivial quasinilpotents. Thus%
\[
\operatorname{Spec}\left(  \mathcal{A}_{q}\left(  \rho\right)  /\mathcal{I}%
_{xy}^{\left(  \rho\right)  }\right)  =\sigma\left(  x^{\sim},y^{\sim}\right)
=\left(  \sigma\left(  x^{\sim}\right)  \times\left\{  0\right\}  \right)
\cup\left(  \left\{  0\right\}  \times\sigma\left(  y^{\sim}\right)  \right)
=\left(  \mathbb{D}_{\rho}\times\left\{  0\right\}  \right)  \cup\left(
\left\{  0\right\}  \times\mathbb{D}_{\rho}\right)  ,
\]
where $\sigma\left(  x^{\sim},y^{\sim}\right)  $ is the joint spectrum (see
\cite[1.3.5]{BourST}) of $\left\{  x^{\sim},y^{\sim}\right\}  $ in
$\mathcal{A}_{q}\left(  \rho\right)  /\mathcal{I}_{xy}^{\left(  \rho\right)
}$.

If $\sigma\left(  f^{\sim}+g^{\sim}\right)  =\left\{  0\right\}  $ for some
$f^{\sim}\in\mathcal{I}_{x^{\sim}}^{\left(  \rho\right)  }$, $g^{\sim}%
\in\mathcal{I}_{y^{\sim}}^{\left(  \rho\right)  }$, then by using the spectral
mapping theorem (see \cite[1.4.7]{BourST}) for the joint spectrum, we derive
that $\lambda\left(  f^{\sim}\right)  =0$ and $\mu\left(  g^{\sim}\right)  =0$
for all $\lambda\in\mathbb{D}_{\rho}\times\left\{  0\right\}  $ and $\mu
\in\left\{  0\right\}  \times\mathbb{D}_{\rho}$. It follows that $f^{\sim
}=g^{\sim}=0$. Thus $\mathcal{A}_{q}\left(  \rho\right)  /\mathcal{I}%
_{xy}^{\left(  \rho\right)  }$ is a semisimple commutative Banach algebra.
Hence $\operatorname{Rad}\mathcal{A}_{q}\left(  \rho\right)  \subseteq
\mathcal{I}_{xy}^{\left(  \rho\right)  }$ and $\cap\left\{  \ker\left(
\lambda\right)  :\lambda\in\operatorname{Spec}\left(  \mathcal{A}_{q}\left(
\rho\right)  /\mathcal{I}_{xy}^{\left(  \rho\right)  }\right)  \right\}
=\left\{  0\right\}  $. Consequently, $\mathcal{I}_{xy}^{\left(  \rho\right)
}=\cap\left\{  \ker\left(  \lambda\right)  :\lambda\in\operatorname{Spec}%
\left(  \mathcal{A}_{q}\left(  \rho\right)  \right)  \right\}  $.

Now let us prove that each $f=\sum_{i,k\in\mathbb{N}}a_{ik}x^{i}y^{k}%
\in\mathcal{I}_{xy}^{\left(  \rho\right)  }$ is quasinilpotent whenever
$\left\vert q\right\vert \neq1$. First assume that $\left\vert q\right\vert
<1$. If $f$ is a polynomial the result follows from Lemma \ref{lAq}. In the
general case, we use (\ref{Ts}) from the proof of Lemma \ref{lAq}. Namely, put
$f_{n}=\sum_{1\leq i,k\leq n}a_{ik}x^{i}y^{k}\in I_{xy}$, $n\in\mathbb{N}$.
Using (\ref{Ts}) and the fact that $\mathcal{A}_{q}\left(  \rho\right)  $ is a
Banach algebra, we infer that
\[
\left\Vert f^{s}\right\Vert _{\rho}^{1/s}\leq\lim_{n}\left\Vert f_{n}%
^{s}\right\Vert _{\rho}^{1/s}\leq\lim_{n}\left\vert q\right\vert ^{\left(
s-1\right)  /2}\rho_{f_{n}}=\lim_{n}\left\vert q\right\vert ^{\left(
s-1\right)  /2}\left\Vert f_{n}\right\Vert _{\rho}=\left\vert q\right\vert
^{\left(  s-1\right)  /2}\left\Vert f\right\Vert _{\rho}.
\]
With $\left\vert q\right\vert <1$ in mind, we deduce that $f$ is quasinilpotent.

Now assume that $\left\vert q\right\vert >1$. Using Lemma \ref{lT}, we derive
that
\[
\left\Vert f_{n}^{s}\right\Vert _{\rho}\leq\sum_{1\leq I,K\leq n}\left\vert
a_{IK}\right\vert \left\vert q\right\vert ^{\left\langle I^{\ast},K_{\left(
s\right)  }\right\rangle }\left\vert q\right\vert ^{-\left\vert I\right\vert
\left\vert K\right\vert }\rho^{\left\vert I\right\vert +\left\vert
K\right\vert }.
\]
Put $p_{IK}=\left\vert I\right\vert \left\vert K\right\vert -\left\langle
I^{\ast},K_{\left(  s\right)  }\right\rangle $ for $I=\left(  i_{1}%
,\ldots,i_{s}\right)  $ and $K=\left(  k_{1},\ldots,k_{s}\right)  $. Note
that
\begin{align*}
p_{IK}  &  =\left\vert I\right\vert \left\vert K\right\vert -\sum_{t=1}%
^{s-1}\left\vert I^{\left(  t\right)  }\right\vert k_{t}=\left(
i_{1}+\left\vert I^{\left(  1\right)  }\right\vert \right)  \left(
k_{1}+\left\vert K^{\left(  1\right)  }\right\vert \right)  -\sum_{t=1}%
^{s-1}\left\vert I^{\left(  t\right)  }\right\vert k_{t}\\
&  =i_{1}k_{1}+\left\vert K^{\left(  1\right)  }\right\vert i_{1}+\left\vert
I^{\left(  1\right)  }\right\vert \left\vert K^{\left(  1\right)  }\right\vert
-\sum_{t=2}^{s-1}\left\vert I^{\left(  t\right)  }\right\vert k_{t}\\
&  =i_{1}k_{1}+\left\vert K^{\left(  1\right)  }\right\vert i_{1}+i_{2}%
k_{2}+\left\vert K^{\left(  2\right)  }\right\vert i_{2}+\left\vert I^{\left(
2\right)  }\right\vert \left\vert K^{\left(  2\right)  }\right\vert
-\sum_{t=3}^{s-1}\left\vert I^{\left(  t\right)  }\right\vert k_{t}\\
&  =\sum_{t=1}^{s-1}i_{t}k_{t}+\left\vert K^{\left(  t\right)  }\right\vert
i_{t}+i_{s}k_{s}=\left\langle K^{\ast},I_{\left(  s\right)  }\right\rangle
+\left\langle I,K\right\rangle
\end{align*}
(see Subsection \ref{subsecIN}). It follows that
\[
\left\vert q\right\vert ^{\left\langle I^{\ast},K_{\left(  s\right)
}\right\rangle }\left\vert q\right\vert ^{-\left\vert I\right\vert \left\vert
K\right\vert }=\left\vert q^{-1}\right\vert ^{p_{IK}}=\left\vert
q^{-1}\right\vert ^{\left\langle K^{\ast},I_{\left(  s\right)  }\right\rangle
+\left\langle I,K\right\rangle }\leq\left\vert q^{-1}\right\vert ^{s\left(
s-1\right)  /2}\left\vert q^{-1}\right\vert ^{\left\langle I,K\right\rangle }%
\]
(see to the proof of Lemma \ref{lAq}). Hence
\[
\left\Vert f_{n}^{s}\right\Vert _{\rho}\leq\left\vert q^{-1}\right\vert
^{s\left(  s-1\right)  /2}\sum_{1\leq I,K\leq n}\left\vert a_{IK}\right\vert
\left\vert q^{-1}\right\vert ^{\left\langle I,K\right\rangle }\rho^{\left\vert
I\right\vert +\left\vert K\right\vert }=\left\vert q^{-1}\right\vert
^{s\left(  s-1\right)  /2}\left\Vert f_{n}\right\Vert _{\rho}^{s},
\]
which in turn implies that
\[
\left\Vert f^{s}\right\Vert _{\rho}^{1/s}\leq\lim_{n}\left\Vert f_{n}%
^{s}\right\Vert _{\rho}^{1/s}\leq\lim_{n}\left\vert q^{-1}\right\vert
^{\left(  s-1\right)  /2}\left\Vert f_{n}\right\Vert _{\rho}=\left\vert
q^{-1}\right\vert ^{\left(  s-1\right)  /2}\left\Vert f\right\Vert _{\rho},
\]
that is, $f$ is quasinilpotent. Thus $\mathcal{I}_{xy}^{\left(  \rho\right)
}$ is a closed two sided ideal in $\mathcal{A}_{q}\left(  \rho\right)  $,
which consists of quasinilpotent elements whenever $\left\vert q\right\vert
\neq1$. Therefore $\mathcal{I}_{xy}^{\left(  \rho\right)  }\subseteq
\operatorname{Rad}\mathcal{A}_{q}\left(  \rho\right)  $. By Proposition
\ref{pequi}, we in turn deduce that $\mathcal{A}_{q}\left(  \rho\right)  $ is
commutative modulo its Jacobson radical. As we have proved above
$\operatorname{Rad}\mathcal{A}_{q}\left(  \rho\right)  \subseteq
\mathcal{I}_{xy}^{\left(  \rho\right)  }$, therefore $\mathcal{I}%
_{xy}^{\left(  \rho\right)  }=\operatorname{Rad}\mathcal{A}_{q}\left(
\rho\right)  $.

Finally, assume that $\left\vert q\right\vert =1$. Take an ordered monomial
$g=x^{i}y^{k}\in\mathcal{I}_{xy}^{\left(  \rho\right)  }$, $i,k\in\mathbb{N}$.
Using Lemma \ref{lT}, we derive that $g^{s}=q^{ik\left(  s-1\right)
s/2}x^{is}y^{ks}$ and
\[
\left\Vert g^{s}\right\Vert _{\rho}=\left\vert q\right\vert ^{ik\left(
s-1\right)  s/2}\rho^{\left(  i+k\right)  s}=\rho^{\left(  i+k\right)
s}=\left\Vert g\right\Vert _{\rho}^{s},
\]
which means that $\left\Vert g\right\Vert _{\rho}$ is just the spectral radius
of $g$ in the Banach algebra $\mathcal{A}_{q}\left(  \rho\right)  $. Hence
$g\notin Q\left(  \mathcal{A}_{q}\left(  \rho\right)  \right)  $ and certainly
$g\notin\operatorname{Rad}\mathcal{A}_{q}\left(  \rho\right)  $.

Finally, let us prove that $Q\left(  \mathcal{A}_{q}\left(  \rho\right)
\right)  \subseteq\mathcal{I}_{xy}^{\left(  \rho\right)  }$. As we have seen
above $\left(  \mathcal{I}_{x}^{\left(  \rho\right)  }+\mathcal{I}%
_{y}^{\left(  \rho\right)  }\right)  \cap Q\left(  \mathcal{A}_{q}\left(
\rho\right)  \right)  =\left\{  0\right\}  $. Take $u\in Q\left(
\mathcal{A}_{q}\left(  \rho\right)  \right)  $. Using the $\ell^{1}%
$-decomposition $\mathcal{A}_{q}\left(  \rho\right)  =\mathcal{I}_{x}^{\left(
\rho\right)  }\oplus\mathcal{I}_{xy}^{\left(  \rho\right)  }\oplus
\mathcal{I}_{y}^{\left(  \rho\right)  }$, we have $u=a+h+b$ for uniquely
defined $a\in\mathcal{I}_{x}^{\left(  \rho\right)  }$, $h\in\mathcal{I}%
_{xy}^{\left(  \rho\right)  }$ and $b\in\mathcal{I}_{y}^{\left(  \rho\right)
}$. Since $u^{s}=a^{s}+h_{s}+b^{s}$ for some $h_{s}\in\mathcal{I}%
_{xy}^{\left(  \rho\right)  }$, it follows that%
\[
0=\inf\left\{  \left\Vert u^{s}\right\Vert _{\rho}^{1/s}\right\}
=\inf\left\{  \left(  \left\Vert a^{s}\right\Vert _{\rho}+\left\Vert
h_{s}\right\Vert _{\rho}+\left\Vert b^{s}\right\Vert _{\rho}\right)
^{1/s}\right\}  \geq\inf\left\{  \max\left\{  \left\Vert a^{s}\right\Vert
_{\rho}^{1/s},\left\Vert b^{s}\right\Vert _{\rho}^{1/s}\right\}  \right\}  ,
\]
which means that $a=b=0$. Whence $u\in\mathcal{I}_{xy}^{\left(  \rho\right)
}$.
\end{proof}

\subsection{The operator $q$-planes}

By an operator $q$-plane we mean a Banach $q$-plane in the algebra
$\mathcal{B}\left(  \mathfrak{X}\right)  $ of bounded linear operators acting
on a complex Banach space $\mathfrak{X}$.

\begin{proposition}
\label{pAcom}Let $\mathcal{A}_{q}$ be a contractive Banach $q$-plane which is
a subalgebra of $\mathcal{B}\left(  \mathfrak{X}\right)  $ for a certain
Banach space $\mathfrak{X}$. If one of the generators $x$, $y$ of
$\mathcal{A}_{q}$ is a compact operator then $\mathcal{I}_{xy}\subseteq
\operatorname{Rad}\mathcal{A}_{q}$. In particular, $\mathcal{A}_{q}$ is
commutative modulo its Jacobson radical.
\end{proposition}

\begin{proof}
If $x$ or $y$ is a compact operator then $I_{xy}$ consists of the compact
operators on $\mathfrak{X}$. Therefore so are the elements of the closure
$\mathcal{I}_{xy}$. By Lemma \ref{lAq}, we have $I_{xy}\subseteq Q\left(
\mathcal{A}_{q}\right)  $. Using the continuity of the spectrum over all
compact operators \cite{Ll}, we conclude that the inclusion $\mathcal{I}%
_{xy}\subseteq Q\left(  \mathcal{A}_{q}\right)  $ holds too. Being
$\mathcal{I}_{xy}$ a two-sided ideal, we deduce that $\mathcal{I}%
_{xy}\subseteq\operatorname{Rad}\mathcal{A}_{q}$. It remains to apply
Proposition \ref{pequi}.
\end{proof}

\begin{example}
\label{exCom}Let $\mathfrak{X}=\ell_{2}$ be a separable Hilbert space with its
standard basis $\left\{  e_{n}\right\}  _{n\in\mathbb{Z}_{+}}$. Consider the
operators $T,S\in\mathcal{B}\left(  \ell_{2}\right)  $ acting by the rules%
\begin{align*}
T\left(  e_{n}\right)   &  =e_{n+1},\\
S\left(  e_{n}\right)   &  =q^{n}e_{n},
\end{align*}
for all $n\in\mathbb{Z}_{+}$. One can easily verify that $TS=q^{-1}ST$. Let
$\mathcal{A}_{q}$ be the closed associative hull of the set $\left\{
T,S\right\}  $ in $\mathcal{B}\left(  \ell_{2}\right)  $. Thus $\mathcal{A}%
_{q}$ is a Banach $q$-plane. If $\left\vert q\right\vert <1$ then
$\mathcal{A}_{q}$ is a contractive Banach $q$-plane, and in this case $S$ is a
compact operator. Using Proposition \ref{pAcom}, we conclude that
$\mathcal{I}_{TS}\subseteq\operatorname{Rad}\mathcal{A}_{q}$, and
$\mathcal{A}_{q}$ is commutative modulo its Jacobson radical.
\end{example}

\section{Arens-Michael envelope of $\mathfrak{A}_{q}$ and noncommutative plane
$\mathbb{C}_{xy}$}

In this section we construct a presheaf of noncommutative Fr\'{e}chet
$\widehat{\otimes}$-algebras on the noncommutative space $\mathbb{C}_{xy}$ of
$\mathfrak{A}_{q}$, $\left\vert q\right\vert \neq1$ which extends the
Arens-Michael envelope of $\mathfrak{A}_{q}$.

\subsection{The Fr\'{e}chet algebra $\mathcal{O}_{q}\left(  \mathbb{C}%
_{xy}\right)  $\label{sAME}}

The Arens-Michael envelope of the $q$-plane $\mathfrak{A}_{q}$ is denoted by
$\mathcal{O}_{q}\left(  \mathbb{C}^{2}\right)  $. As shown in \cite{Pir}, the
algebra $\mathcal{O}_{q}\left(  \mathbb{C}^{2}\right)  $ is the inverse limit
of the Banach algebras $\mathcal{A}_{q}\left(  \rho\right)  $, $\rho<\infty$
considered in Subsection \ref{subsecBqA}. For a couple of positive real $\rho$
and $r$ with $\rho<r$, the identity map on the $q$-plane $\mathfrak{A}_{q}$ is
extended up to a bounded (contractive) algebra homomorphism $u_{\rho
r}:\mathcal{A}_{q}\left(  r\right)  \rightarrow\mathcal{A}_{q}\left(
\rho\right)  $. This defines the inverse system $\mathcal{S}_{q}=\left\{
\mathcal{A}_{q}\left(  \rho\right)  ,u_{\rho r}:\rho<r\right\}  $ of the
Banach $q$-planes such that
\[
\mathcal{O}_{q}\left(  \mathbb{C}^{2}\right)  =\underleftarrow{\lim
}\mathcal{S}_{q}=\left\{  f=\sum_{i,k\in\mathbb{Z}_{+}}a_{ik}x^{i}%
y^{k}:\left\Vert f\right\Vert _{\rho}<\infty,\forall\rho>0\right\}  \text{.}%
\]
As above the closures of $I_{xy}$, $I_{x}$ and $I_{y}$ in $\mathcal{O}%
_{q}\left(  \mathbb{C}^{2}\right)  $ are denoted by $\mathcal{I}_{xy}$,
$\mathcal{I}_{x}$ and $\mathcal{I}_{y}$, respectively.

\begin{proposition}
\label{tO}The Fr\'{e}chet algebra $\mathcal{O}_{q}\left(  \mathbb{C}%
^{2}\right)  $ admits the following topological direct sum decomposition
\begin{align*}
\mathcal{O}_{q}\left(  \mathbb{C}^{2}\right)   &  =\mathcal{I}_{x}%
\oplus\mathcal{I}_{xy}\oplus\mathcal{I}_{y},\text{\quad}\mathcal{I}_{xy}%
=\cap\left\{  \ker\left(  \lambda\right)  :\lambda\in\operatorname{Spec}%
\left(  \mathcal{O}_{q}\left(  \mathbb{C}^{2}\right)  \right)  \right\}  ,\\
\operatorname{Rad}\mathcal{O}_{q}\left(  \mathbb{C}^{2}\right)   &
\subseteq\mathcal{I}_{xy},\quad\operatorname{Spec}\left(  \mathcal{O}%
_{q}\left(  \mathbb{C}^{2}\right)  \right)  =\operatorname{Spec}\left(
\mathcal{O}_{q}\left(  \mathbb{C}^{2}\right)  /\mathcal{I}_{xy}\right)
=\mathbb{C}_{xy},
\end{align*}
where $\mathbb{C}_{x}=\mathbb{C\times}\left\{  0\right\}  \subseteq
\mathbb{C}^{2}$, $\mathbb{C}_{y}=\left\{  0\right\}  \times\mathbb{C\subseteq
C}^{2}$ and $\mathbb{C}_{xy}=\mathbb{C}_{x}\cup\mathbb{C}_{y}$. If $\left\vert
q\right\vert \neq1$ then the algebra $\mathcal{O}_{q}\left(  \mathbb{C}%
^{2}\right)  $ is commutative modulo its Jacobson radical $\operatorname{Rad}%
\mathcal{O}_{q}\left(  \mathbb{C}^{2}\right)  $ and $\mathcal{I}%
_{xy}=\operatorname{Rad}\mathcal{O}_{q}\left(  \mathbb{C}^{2}\right)  $. If
$\left\vert q\right\vert =1$ then $Q\left(  \mathcal{O}_{q}\left(
\mathbb{C}^{2}\right)  \right)  \subseteq\mathcal{I}_{xy}$ and $x^{i}y^{k}%
\in\mathcal{I}_{xy}\backslash Q\left(  \mathcal{O}_{q}\left(  \mathbb{C}%
^{2}\right)  \right)  $ for all $i,k\in\mathbb{N}$.
\end{proposition}

\begin{proof}
Take $\rho,r\in\mathbb{R}$ with $0<\rho<r$. Obviously, $u_{\rho r}\left(
\mathcal{I}_{x}^{\left(  r\right)  }\right)  \subseteq\mathcal{I}_{x}^{\left(
\rho\right)  }$, $u_{\rho r}\left(  \mathcal{I}_{xy}^{\left(  r\right)
}\right)  \subseteq\mathcal{I}_{xy}^{\left(  \rho\right)  }$, and $u_{\rho
r}\left(  \mathcal{I}_{y}^{\left(  r\right)  }\right)  \subseteq
\mathcal{I}_{y}^{\left(  \rho\right)  }$ (see Subsection \ref{subsecBqA}).
Using Proposition \ref{pAfree}, we deduce that $\mathcal{O}_{q}\left(
\mathbb{C}^{2}\right)  =\mathcal{I}_{x}\oplus\mathcal{I}_{xy}\oplus
\mathcal{I}_{y}$. Since each continuous character $\lambda:\mathcal{O}%
_{q}\left(  \mathbb{C}^{2}\right)  \rightarrow\mathbb{C}$ can be factorized
through a character $\mathcal{A}_{q}\left(  \rho\right)  \rightarrow
\mathbb{C}$, we conclude that $\mathcal{I}_{xy}=\cap\left\{  \ker\left(
\lambda\right)  :\lambda\in\operatorname{Spec}\left(  \mathcal{O}_{q}\left(
\mathbb{C}^{2}\right)  \right)  \right\}  $ thanks to Proposition
\ref{pAfree}. In particular, $\operatorname{Rad}\mathcal{O}_{q}\left(
\mathbb{C}^{2}\right)  \subseteq\mathcal{I}_{xy}$.

The quotient algebra $\mathcal{O}_{q}\left(  \mathbb{C}^{2}\right)
/\mathcal{I}_{xy}$ is commutative and $\operatorname{Spec}\left(
\mathcal{O}_{q}\left(  \mathbb{C}^{2}\right)  \right)  =\operatorname{Spec}%
\left(  \mathcal{O}_{q}\left(  \mathbb{C}^{2}\right)  /\mathcal{I}%
_{xy}\right)  $. Actually, $\mathcal{O}_{q}\left(  \mathbb{C}^{2}\right)
/\mathcal{I}_{xy}=\mathcal{I}_{x^{\sim}}\oplus\mathcal{I}_{y^{\sim}}$, where
$\mathcal{I}_{x^{\sim}}$ and $\mathcal{I}_{y^{\sim}}$ are the closed
subalgebras generated by $x^{\sim}$ and $y^{\sim}$ respectively. Note that
$\mathcal{I}_{x^{\sim}}$ is unital, and $x^{\sim}y^{\sim}=0$. Moreover,
$\sigma\left(  x^{\sim}\right)  =\sigma\left(  y^{\sim}\right)  =\bigcup
\limits_{\rho}\mathbb{D}_{\rho}=\mathbb{C}$, that is, $\mathcal{I}_{x^{\sim}%
}=\mathcal{O}\left(  \mathbb{C}\right)  $ and $\mathcal{I}_{y^{\sim}}%
\oplus\mathbb{C}1=\mathcal{O}\left(  \mathbb{C}\right)  $ are the algebras of
entire functions (see Lemma \ref{lemAK0}). It follows that
\begin{align*}
\operatorname{Spec}\left(  \mathcal{O}_{q}\left(  \mathbb{C}^{2}\right)
/\mathcal{I}_{xy}\right)   &  =\bigcup\limits_{\rho}\operatorname{Spec}\left(
\mathcal{A}_{q}\left(  \rho\right)  /\mathcal{I}_{xy}^{\left(  \rho\right)
}\right)  =\bigcup\limits_{\rho}\left(  \mathbb{D}_{\rho}\times\left\{
0\right\}  \right)  \cup\left(  \left\{  0\right\}  \times\mathbb{D}_{\rho
}\right) \\
&  =\left(  \mathbb{C\times}\left\{  0\right\}  \right)  \cup\left(  \left\{
0\right\}  \times\mathbb{C}\right)  =\mathbb{C}_{x}\cup\mathbb{C}%
_{y}=\mathbb{C}_{xy}\text{,}%
\end{align*}
that is, $\operatorname{Spec}\left(  \mathcal{O}_{q}\left(  \mathbb{C}%
^{2}\right)  \right)  =\operatorname{Spec}\left(  \mathcal{O}_{q}\left(
\mathbb{C}^{2}\right)  /\mathcal{I}_{xy}\right)  =\mathbb{C}_{xy}$.

Now assume that $\left\vert q\right\vert \neq1$. Take $g=\left(  g_{\rho
}\right)  \in\mathcal{I}_{xy}$ with $g_{\rho}\in\mathcal{I}_{xy}^{\left(
\rho\right)  }$, and take $f=\left(  f_{\rho}\right)  \in\mathcal{O}%
_{q}\left(  \mathbb{C}^{2}\right)  $. Using Proposition \ref{pAfree} and
\cite[5.2.12]{Hel}, we deduce that $\sigma\left(  fg\right)  =\bigcup
\limits_{\rho}\sigma\left(  f_{\rho}g_{\rho}\right)  =\left\{  0\right\}  $,
that is, $g\in\operatorname{Rad}\mathcal{O}_{q}\left(  \mathbb{C}^{2}\right)
$. Hence $\mathcal{I}_{xy}=\operatorname{Rad}\mathcal{O}_{q}\left(
\mathbb{C}^{2}\right)  $ and $\mathcal{O}_{q}\left(  \mathbb{C}^{2}\right)  $
is commutative modulo $\operatorname{Rad}\mathcal{O}_{q}\left(  \mathbb{C}%
^{2}\right)  $.

Finally, assume that $\left\vert q\right\vert =1$. Take $g=\left(  g_{\rho
}\right)  \in Q\left(  \mathcal{O}_{q}\left(  \mathbb{C}^{2}\right)  \right)
$. Since $\sigma\left(  g_{\rho}\right)  \subseteq\sigma\left(  g\right)
=\left\{  0\right\}  $, it follows that $g_{\rho}\in Q\left(  \mathcal{A}%
_{q}\left(  \rho\right)  \right)  $ for each $\rho$. Using again Proposition
\ref{pAfree}, we obtain that $g_{\rho}\in\mathcal{I}_{xy}^{\left(
\rho\right)  }$ for every $\rho$. Hence $g\in\mathcal{I}_{xy}$. The fact
$x^{i}y^{k}\in\mathcal{I}_{xy}\backslash Q\left(  \mathcal{O}_{q}\left(
\mathbb{C}^{2}\right)  \right)  $ for all $i,k\in\mathbb{N}$ follows from
Proposition \ref{pAfree} as well.
\end{proof}

\begin{remark}
\label{remCXCY}The character space $\mathbb{C}_{xy}$\ of the Fr\'{e}chet
algebra $\mathcal{O}_{q}\left(  \mathbb{C}^{2}\right)  $ is the union
$\mathbb{C}_{x}\cup\mathbb{C}_{y}$ of two copies of the complex plane whose
intersection consist of the trivial character that responds to $\left(
0,0\right)  $. One can swap the role of $x$ and $y$ having $\mathbb{C}_{yx}$
as the character space of $\mathcal{O}_{q^{-1}}\left(  \mathbb{C}^{2}\right)
$. Based on the just proven Proposition \ref{tO}, it is reasonable to use the
notation $\mathcal{O}_{q}\left(  \mathbb{C}_{xy}\right)  $ instead of
$\mathcal{O}_{q}\left(  \mathbb{C}^{2}\right)  $ whenever $\left\vert
q\right\vert \neq1$.
\end{remark}

Below the algebra $\mathcal{O}_{q}\left(  \mathbb{C}_{xy}\right)  $ is
extended up to a reasonable structure presheaf of the noncommutative plane
$\mathbb{C}_{xy}$.

\subsection{The sheaf of stalks of holomorphic functions on $\left(
\mathbb{C},\mathfrak{q}\right)  $ and $\left(  \mathbb{C},\mathfrak{d}\right)
$}

As in Subsection \ref{subsecQT}, we denote by $\mathcal{O}^{\mathfrak{q}}$ and
$\mathcal{O}^{\mathfrak{d}}$ the Fr\'{e}chet sheaf of stalks of the
holomorphic functions on $\left(  \mathbb{C},\mathfrak{q}\right)  $ and
$\left(  \mathbb{C},\mathfrak{d}\right)  $, respectively.

\begin{lemma}
\label{lemQC2}The stalks of the sheaves $\mathcal{O}^{\mathfrak{q}}$ and
$\mathcal{O}^{\mathfrak{d}}$ at zero are the same stalk $\mathcal{O}_{0}$ of
the standard sheaf $\mathcal{O}$ with its residue field $\mathbb{C}%
_{0}=\mathbb{C}$, whereas
\[
\mathcal{O}_{\lambda}^{\mathfrak{q}}=\mathcal{O}\left(  \left\{
\lambda\right\}  _{q}\right)  =\mathcal{O}_{0}+\sum_{n\in\mathbb{Z}_{+}%
}\mathcal{O}_{q^{n}\lambda}\quad\text{and\quad}\mathcal{O}_{\lambda
}^{\mathfrak{d}}=\mathcal{O}\left(  \mathbb{D}_{\left\vert \lambda\right\vert
}\right)
\]
at every $\lambda\in\mathbb{C}\backslash\left\{  0\right\}  $.
\end{lemma}

\begin{proof}
Since the $\mathfrak{q}$-topology, $\mathfrak{d}$-topology and the standard
topology have the same neighborhood filter base of the origin (see Lemma
\ref{lemQC}), the equalities $\mathcal{O}_{0}^{\mathfrak{q}}=\mathcal{O}%
_{0}=\mathcal{O}_{0}^{\mathfrak{d}}$ are immediate.

Take $\lambda\in\mathbb{C}\backslash\left\{  0\right\}  $ and its
$\mathfrak{q}$-neighborhood $U$, which is a $q$-open subset containing
$\lambda$. Then the $q$-hull $\left\{  \lambda\right\}  _{q}$ of $\lambda$ is
contained in $U$. But $B\left(  0,\varepsilon\right)  \subseteq U$ for a small
$\varepsilon>0$, and there are just finitely many of $q^{m}\lambda$ outside of
the disk $B\left(  0,\varepsilon\right)  $, say $\left\{  \lambda
,q\lambda,\ldots q^{n}\lambda\right\}  $. Pick up small disks $B\left(
q^{m}\lambda,\left\vert q\right\vert ^{m}\delta\right)  \subseteq U$, $0\leq
m\leq n$ for $\delta>0$. Then $U_{\varepsilon,\delta}=B\left(  0,\varepsilon
\right)  \cup\bigcup\limits_{m=0}^{n}B\left(  q^{m}\lambda,\left\vert
q\right\vert ^{m}\delta\right)  $ is a $q$-open neighborhood of $\left\{
\lambda\right\}  _{q}$ contained in $U$. Actually, the latter union can be
assumed to be disjoint. In particular, every $s\in\mathcal{O}_{0}%
\oplus\bigoplus\limits_{m\in\mathbb{Z}_{+}}\mathcal{O}_{q^{m}\lambda}$ can be
written as a finite sum $s=\left\langle B\left(  0,\varepsilon\right)
,f_{0}\right\rangle +\sum_{m=0}^{n}\left\langle B\left(  q^{m}\lambda
,\left\vert q\right\vert ^{m}\delta\right)  ,g_{m}\right\rangle $ for some
$f_{0}\in\mathcal{O}\left(  B\left(  0,\varepsilon\right)  \right)  $ and
$g_{m}\in\mathcal{O}\left(  B\left(  q^{m}\lambda,\left\vert q\right\vert
^{m}\delta\right)  \right)  $. Using the standard sheaf property (of
$\mathcal{O}$), we obtain a unique $f\in\mathcal{O}\left(  U_{\varepsilon
,\delta}\right)  $ such that $\left\langle U_{\varepsilon,\delta
},f\right\rangle =\left\langle B\left(  0,\varepsilon\right)  ,f_{0}%
\right\rangle $ and $\left\langle U_{\varepsilon,\delta},f\right\rangle
=\left\langle B\left(  q^{m}\lambda,\left\vert q\right\vert ^{m}\delta\right)
,g_{m}\right\rangle $ for all $m$. It follows that there is a well defined
homomorphism
\[
\mathcal{O}_{0}\oplus\bigoplus\limits_{m\in\mathbb{Z}_{+}}\mathcal{O}%
_{q^{m}\lambda}\rightarrow\mathcal{O}\left(  \left\{  \lambda\right\}
_{q}\right)  ,\quad s=\left\langle B\left(  0,\varepsilon\right)
,f_{0}\right\rangle +\sum_{m=0}^{n}\left\langle B\left(  q^{m}\lambda
,\left\vert q\right\vert ^{m}\delta\right)  ,g_{m}\right\rangle \mapsto
\left\langle U_{\varepsilon,\delta},f\right\rangle ,
\]
where $\mathcal{O}\left(  \left\{  \lambda\right\}  _{q}\right)  $ is the
algebra of all holomorphic stalks on the compact set $\left\{  \lambda
\right\}  _{q}$. Conversely, $\left\langle U,f\right\rangle \in\mathcal{O}%
_{\lambda}^{\mathfrak{q}}$ if and only if $\left\langle U,f\right\rangle
=\left\langle U_{\varepsilon,\delta},f\right\rangle \in\mathcal{O}\left(
\left\{  \lambda\right\}  _{q}\right)  $. In this case, $s=\left\langle
B\left(  0,\varepsilon\right)  ,f\right\rangle +\sum_{m=0}^{n}\left\langle
B\left(  q^{m}\lambda,\left\vert q\right\vert ^{m}\delta\right)
,f\right\rangle $ is representing the stalk $\left\langle U_{\varepsilon
,\delta},f\right\rangle $. Thus $\mathcal{O}_{\lambda}^{\mathfrak{q}%
}=\mathcal{O}\left(  \left\{  \lambda\right\}  _{q}\right)  $ and it is a
quotient of the direct algebraic sum $\mathcal{O}_{0}\oplus\bigoplus
\limits_{m\in\mathbb{Z}_{+}}\mathcal{O}_{q^{m}\lambda}$. In particular,
$\mathcal{O}_{\lambda}^{\mathfrak{q}}=\mathcal{O}_{0}+\sum_{m\in\mathbb{Z}%
_{+}}\mathcal{O}_{q^{m}\lambda}$. Finally, $\mathcal{O}_{\lambda
}^{\mathfrak{d}}=\underrightarrow{\lim}\left\{  \mathcal{O}\left(  B\left(
0,r\right)  \right)  :\left\vert \lambda\right\vert <r\right\}  =\mathcal{O}%
\left(  \mathbb{D}_{\left\vert \lambda\right\vert }\right)  $, that is,
$\mathcal{O}_{\lambda}\subseteq\mathcal{O}_{\lambda}^{\mathfrak{q}%
}=\mathcal{O}\left(  \left\{  \lambda\right\}  _{q}\right)  \subseteq
\mathcal{O}_{\lambda}^{\mathfrak{d}}=\mathcal{O}\left(  \mathbb{D}_{\left\vert
\lambda\right\vert }\right)  $.
\end{proof}

\begin{remark}
The algebra $\mathcal{O}_{\lambda}^{\mathfrak{q}}$ is not local for
$\lambda\in\mathbb{C}\backslash\left\{  0\right\}  $. It has an ideal of those
stalks $\left\langle U,f\right\rangle \in\mathcal{O}_{\lambda}^{\mathfrak{q}}$
with $f\left(  \left\{  \lambda\right\}  _{q}\right)  =\left\{  0\right\}  $.
The related quotient (as a residue "field") equals to $\mathbb{C}_{0}%
+\sum_{n\in\mathbb{Z}_{+}}\mathbb{C}_{q^{n}\lambda}$, where $\mathbb{C}%
_{q^{n}\lambda}=\mathbb{C}$ is the residue field of $\mathcal{O}$ at
$q^{n}\lambda$. The same holds for the algebra $\mathcal{O}_{\lambda
}^{\mathfrak{d}}$.
\end{remark}

The character space $\mathbb{C}_{xy}$ being the union $\mathbb{C}_{x}%
\cup\mathbb{C}_{y}$ can be equipped with the final topology so that both
embeddings
\[
\left(  \mathbb{C}_{x},\mathfrak{q}\right)  \hookrightarrow\mathbb{C}%
_{xy}\hookleftarrow\left(  \mathbb{C}_{y},\mathfrak{d}\right)
\]
are continuous, which is called the $\left(  \mathfrak{q},\mathfrak{d}\right)
$\textit{-topology of} $\mathbb{C}_{xy}$. The topology base in $\mathbb{C}%
_{xy}$ consists of all open subsets $U=U_{x}\cup U_{y}$ with $q$-open set
$U_{x}\subseteq\mathbb{C}_{x}$ and a disk $U_{y}=B\left(  0,r_{y}\right)
\subseteq\mathbb{C}_{y}$. In this case, $\mathbb{C}_{xy}=\mathbb{C}_{x}%
\cup\mathbb{C}_{y}$ is the union of two irreducible components, whose
intersection is a unique generic point (see Lemma \ref{lemQC}).

Let us consider the subsheaf $\mathcal{I}_{\mathfrak{q}}$ of $\mathcal{O}%
^{\mathfrak{q}}$ on $\mathbb{C}_{x}$ whose stalks vanish at the origin. Notice
that $\mathcal{I}_{\mathfrak{q}}^{+}=\mathcal{O}^{\mathfrak{q}}$ on
$\mathbb{C}_{x}$, where $\mathcal{I}_{\mathfrak{q}}^{+}$ is the unitization of
the sheaf $\mathcal{I}_{\mathfrak{q}}$. In a similar way, we have
$\mathcal{I}_{\mathfrak{d}}^{+}=\mathcal{O}^{\mathfrak{d}}$ on $\mathbb{C}%
_{y}$. Thus there are ringed spaces (see \cite[2.2]{Harts}) $\left(
\mathbb{C}_{x},\mathcal{I}_{\mathfrak{q}}\right)  $ and $\left(
\mathbb{C}_{y},\mathcal{I}_{\mathfrak{d}}\right)  $, whose unitizations are
reduced to $\left(  \mathbb{C}_{x},\mathcal{O}^{\mathfrak{q}}\right)  $ and
$\left(  \mathbb{C}_{y},\mathcal{O}^{\mathfrak{d}}\right)  $, respectively.

\subsection{The structure presheaf of $\mathbb{C}_{xy}$\label{subsecCXY}}

Fix $\left\vert q\right\vert <1$. We define the ringed space $\left(
\mathbb{C}_{xy},\mathcal{O}_{q}\right)  $ to be the unital projective tensor
product of the ringed spaces $\left(  \mathbb{C}_{x},\mathcal{I}%
_{\mathfrak{q}}\right)  $ and $\left(  \mathbb{C}_{y},\mathcal{I}%
_{\mathfrak{d}}\right)  $ in the following way. The underline space
$\mathbb{C}_{xy}$ is equipped with the $\left(  \mathfrak{q},\mathfrak{d}%
\right)  $\textit{-}topology and its structure presheaf is defined as
\[
\mathcal{O}_{q}=\mathcal{I}_{\mathfrak{q}}^{+}\widehat{\otimes}\mathcal{I}%
_{\mathfrak{d}}^{+}=\left(  \mathcal{O}^{\mathfrak{q}}|\mathbb{C}_{x}\right)
\widehat{\otimes}\left(  \mathcal{O}^{\mathfrak{d}}|\mathbb{C}_{y}\right)  .
\]
Thus $\mathcal{O}_{q}\left(  U\right)  =\mathcal{O}^{\mathfrak{q}}\left(
U_{x}\right)  \widehat{\otimes}\mathcal{O}^{\mathfrak{d}}\left(  U_{y}\right)
$ for every $\left(  \mathfrak{q},\mathfrak{d}\right)  $\textit{-}open subset
$U=U_{x}\cup U_{y}\subseteq\mathbb{C}_{xy}$. We assume that $\mathcal{O}%
_{q}\left(  \varnothing\right)  =\left\{  0\right\}  $, and use the same
notation for the presheaf and the associated sheaf. Notice that
\begin{align*}
\mathcal{O}_{q}\left(  U_{x}\right)   &  =\mathcal{O}^{\mathfrak{q}}\left(
U_{x}\right)  \widehat{\otimes}\mathcal{I}_{\mathfrak{d}}^{+}\left(
\varnothing\right)  =\mathcal{O}^{\mathfrak{q}}\left(  U_{x}\right)
\widehat{\otimes}\mathbb{C=}\mathcal{O}^{\mathfrak{q}}\left(  U_{x}\right)
,\text{ }\\
\mathcal{O}_{q}\left(  U_{y}\right)   &  =\mathcal{I}_{\mathfrak{q}}%
^{+}\left(  \varnothing\right)  \widehat{\otimes}\mathcal{I}_{\mathfrak{d}%
}^{+}\left(  U_{y}\right)  =\mathbb{C}\widehat{\otimes}\mathcal{O}%
^{\mathfrak{d}}\left(  U_{y}\right)  =\mathcal{O}^{\mathfrak{d}}\left(
U_{y}\right)  ,
\end{align*}
which means that $\mathcal{O}_{q}|\mathbb{C}_{x}=\mathcal{O}^{\mathfrak{q}%
}|\mathbb{C}_{x}$ and $\mathcal{O}_{q}|\mathbb{C}_{y}=\mathcal{O}%
^{\mathfrak{d}}|\mathbb{C}_{y}$. In particular, using Lemma \ref{lemQC2}, we
deduce that
\[
\left(  \mathcal{O}_{q}\right)  _{\lambda}=\mathcal{O}_{\lambda}%
^{\mathfrak{q}}=\mathcal{O}_{0}\oplus\bigoplus\limits_{n\in\mathbb{Z}_{+}%
}\mathcal{O}_{q^{n}\lambda}\text{\quad and\quad}\left(  \mathcal{O}%
_{q}\right)  _{\mu}=\mathcal{O}_{\mu}^{\mathfrak{d}}=\mathcal{O}\left(
\mathbb{D}_{\left\vert \mu\right\vert }\right)
\]
for every $\lambda\in\mathbb{C}_{x}\backslash\left\{  0\right\}  $ and $\mu
\in\mathbb{C}_{y}\backslash\left\{  0\right\}  $. But
\[
\left(  \mathcal{O}_{q}\right)  _{0}=\underrightarrow{\lim}\left\{
\mathcal{O}_{q}\left(  U\right)  \right\}  =\underrightarrow{\lim}\left\{
\mathcal{O}\left(  U_{x}\right)  \widehat{\otimes}\mathcal{O}\left(
U_{y}\right)  \right\}  =\underrightarrow{\lim}\left\{  \mathcal{O}\left(
U_{x}\times U_{y}\right)  \right\}
\]
(see \cite[Ch. II, 4.15]{HelHom} ), where $U$ is running over all $\left(
\mathfrak{q},\mathfrak{d}\right)  $\textit{-}open subset $U\subseteq
\mathbb{C}_{xy}$.

\begin{lemma}
\label{lemQC3}If $U=U_{x}\cup U_{y}\subseteq\mathbb{C}_{xy}$ is a $\left(
\mathfrak{q},\mathfrak{d}\right)  $\textit{-}open subset of $\mathbb{C}_{xy}$
with $U_{y}=B\left(  0,r_{y}\right)  $, then
\[
\mathcal{O}_{q}\left(  U\right)  =\left\{  f=\sum_{n}f_{n}\left(  x\right)
y^{n}\in\mathcal{O}\left(  U_{x}\right)  \left[  \left[  y\right]  \right]
:f_{n}\in\mathcal{O}\left(  U_{x}\right)  ,p_{K,\rho}\left(  f\right)
<\infty,K\subseteq U_{x},\rho<r_{y}\right\}
\]
as the Fr\'{e}chet spaces, where $p_{K,\rho}\left(  f\right)  =\sum
_{n}\left\Vert f_{n}\right\Vert _{K}\rho^{n}$ and $K\subseteq U_{x}$ is a
quasicompact subset. The family $\left\{  p_{K,\rho}:K\subseteq U_{x}%
,0<\rho<r_{y}\right\}  $ of seminorms defines the Fr\'{e}chet topology of
$\mathcal{O}_{q}\left(  U\right)  $, and
\[
\mathcal{O}_{q}\left(  U\right)  =\underleftarrow{\lim}\left\{  \mathcal{A}%
\left(  K\right)  \widehat{\otimes}\mathcal{A}\left(  \rho\right)  :K\subseteq
U_{x},\rho<r_{y}\right\}  .
\]

\end{lemma}

\begin{proof}
By Lemma \ref{lemAK0}, $\left\{  \left\Vert \cdot\right\Vert _{K}%
\otimes\left\Vert \cdot\right\Vert _{\rho}:K\subseteq U_{x},\rho
<r_{y}\right\}  $ is a defining family of seminorms of the projective tensor
product $\mathcal{O}^{\mathfrak{q}}\left(  U_{x}\right)  \widehat{\otimes
}\mathcal{O}^{\mathfrak{d}}\left(  U_{y}\right)  $, and the equality
$\mathcal{O}_{q}\left(  U\right)  =\underleftarrow{\lim}\left\{
\mathcal{A}\left(  K\right)  \widehat{\otimes}\mathcal{A}\left(  \rho\right)
\right\}  $ holds up to a topological isomorphism.

Now let $\mathcal{F}$ be the polynormed space of all absolutely convergent
series $f=\sum_{n}f_{n}\left(  x\right)  y^{n}\in\mathcal{O}\left(
U_{x}\right)  \left[  \left[  y\right]  \right]  $ with $p_{K,\rho}\left(
f\right)  <\infty$, $K\subseteq U_{x}$, $\rho<r_{y}$. Since $\left(
\left\Vert \cdot\right\Vert _{K}\otimes\left\Vert \cdot\right\Vert _{\rho
}\right)  \left(  f_{n}\otimes z^{n}\right)  =\left\Vert f_{n}\right\Vert
_{K}\rho^{n}$ (see \cite[Ch. II, 1.16]{HelHom}), it follows that the series
$\sum_{n}f_{n}\otimes z^{n}$ converges absolutely in $\mathcal{O}\left(
U_{x}\right)  \widehat{\otimes}\mathcal{O}\left(  U_{y}\right)  $ once
$f\in\mathcal{F}$. It follows that the linear map%
\[
\alpha:\mathcal{F\longrightarrow O}\left(  U_{x}\right)  \widehat{\otimes
}\mathcal{O}\left(  U_{y}\right)  \text{, }\alpha\left(  f\right)  =\sum
_{n}f_{n}\otimes z^{n}%
\]
is continuous, namely, $\left(  \left\Vert \cdot\right\Vert _{K}%
\otimes\left\Vert \cdot\right\Vert _{\rho}\right)  \left(  \alpha\left(
f\right)  \right)  \leq\sum_{n}\left\Vert f_{n}\right\Vert _{K}\rho
^{n}=p_{K,\rho}\left(  f\right)  $ for all $K$ and $\rho<r_{y}$. Further, the
bilinear mapping
\[
\mathcal{O}\left(  U_{x}\right)  \times\mathcal{O}\left(  U_{y}\right)
\longrightarrow\mathcal{F},\quad\left(  g,\sum_{n}a_{n}z^{n}\right)
\mapsto\sum_{n}a_{n}g\left(  x\right)  y^{n}%
\]
is jointly continuous, for $p_{K,\rho}\left(  \sum_{n}a_{n}g\left(  x\right)
y^{n}\right)  =\sum_{n}\left\vert a_{n}\right\vert \left\Vert g\right\Vert
_{K}\rho^{n}=\left\Vert g\right\Vert _{K}\left\Vert \sum_{n}a_{n}%
z^{n}\right\Vert _{\rho}$ (see Lemma \ref{lemAK0}). Hence there is a unique
continuous linear map $\beta:\mathcal{O}\left(  U_{x}\right)  \widehat{\otimes
}\mathcal{O}\left(  U_{y}\right)  \rightarrow\mathcal{F}$ such that
$\beta\left(  g\otimes z^{n}\right)  =g\left(  x\right)  y^{n}$. Moreover,
\[
\left(  \beta\alpha\right)  \left(  f\right)  =\beta\left(  \sum_{n}%
f_{n}\otimes z^{n}\right)  =\sum_{n}\beta\left(  f_{n}\otimes z^{n}\right)
=\sum_{n}f_{n}\left(  x\right)  y^{n}=f
\]
for all $f\in\mathcal{F}$. Conversely, for every $u\in\mathcal{O}\left(
U_{x}\right)  \widehat{\otimes}\mathcal{O}\left(  U_{y}\right)  $ we have
$u=\sum_{i=1}^{\infty}f_{i}\otimes e_{i}$ to be an absolutely convergent
series (see \cite[Ch. III, 6.4]{Sch}) with $e_{i}=\sum_{n}a_{in}z^{n}%
\in\mathcal{O}\left(  U_{y}\right)  $. Then
\begin{align*}
\left(  \alpha\beta\right)  \left(  u\right)   &  =\alpha\left(  \sum
_{i=1}^{\infty}f_{i}\left(  x\right)  \sum_{n}a_{in}y^{n}\right)
=\alpha\left(  \sum_{n}\sum_{i}a_{in}f_{i}\left(  x\right)  y^{n}\right)
=\sum_{n}\left(  \sum_{i}a_{in}f_{i}\right)  \otimes z^{n}\\
&  =\sum_{i}f_{i}\otimes\left(  \sum_{n}a_{in}z^{n}\right)  =\sum_{i}%
f_{i}\otimes e_{i}=u,
\end{align*}
which means that $\alpha$ implements a topological isomorphism with its
inverse map $\beta$. One can use the completeness argument of $\mathcal{F}$ to
make the proof a bit shorter, but the present form clarifies the maps for the
both directions.
\end{proof}

Thus the Fr\'{e}chet (pre)sheaf $\mathcal{O}_{q}$ can be identified with a
subsheaf of $\mathcal{O}^{\mathfrak{q}}\left[  \left[  y\right]  \right]  $
due to Lemma \ref{lemQC3}. One can swap the topology order in $\mathbb{C}%
_{x}\cup\mathbb{C}_{y}$ by endowing it with the $\left(  \mathfrak{d}%
,\mathfrak{q}\right)  $\textit{-topology. }In this case we use the notation
$\mathbb{C}_{xy}^{\operatorname{op}}$ instead of $\mathbb{C}_{xy}$, and we
come up with the related sheaf $\mathcal{O}_{q}^{\operatorname{op}}$, which is
identified with a subsheaf of $\mathcal{O}^{\mathfrak{q}}\left[  \left[
x\right]  \right]  $. Namely, every $f\in\mathcal{O}_{q}^{\operatorname{op}%
}\left(  U\right)  $ can be written in the form $f=\sum_{n}x^{n}f_{n}\left(
y\right)  $ with the same family of defining seminorms.

Now we define a noncommutative multiplication over the sections of the sheaf
$\mathcal{O}_{q}$. It suffices to define it over the related presheaf, that
is, over the space $\mathcal{O}_{q}\left(  U\right)  $, where $U\subseteq
\mathbb{C}_{xy}$ is a $\left(  \mathfrak{q},\mathfrak{d}\right)  $-open
subset. Pick $f,g\in\mathcal{O}\left(  U_{x}\right)  \left[  \left[  y\right]
\right]  $ and let us define the following formal $q$-multiplication%
\[
f\cdot g=\sum_{n}\left(  \sum_{i+j=n}f_{i}\left(  x\right)  g_{j}\left(
q^{i}x\right)  \right)  y^{n}.
\]
Notice that $\left\{  q^{i}x:i\in\mathbb{Z}_{+}\right\}  \cup\left\{
0\right\}  =\left\{  x\right\}  _{q}\subseteq U_{x}$ whenever $x\in U_{x}$,
and $f_{i}\left(  x\right)  g_{j}\left(  q^{i}x\right)  $ is the standard
multiplication from the commutative algebra $\mathcal{O}\left(  U_{x}\right)
$. One can easily verify that this is an associative multiplication.

\begin{lemma}
\label{lemQC4}Let $U\subseteq\mathbb{C}_{xy}$ be a $\left(  \mathfrak{q}%
,\mathfrak{d}\right)  $\textit{-}open subset. Then the Fr\'{e}chet space
$\mathcal{O}_{q}\left(  U\right)  $ equipped with the formal $q$%
-multiplication turns out to be a Fr\'{e}chet $\widehat{\otimes}$-algebra, and
the defining family $\left\{  p_{K,\rho}\right\}  $ of seminorms on
$\mathcal{O}_{q}\left(  U\right)  $ are multiplicative whenever $K$ is running
over all $q$-compact subsets of $U_{x}$, and $\rho<r_{y}$. Moreover,
$\mathcal{O}_{q}$ is a presheaf of Fr\'{e}chet $\widehat{\otimes}$-algebras on
$\mathbb{C}_{xy}$.
\end{lemma}

\begin{proof}
Using Lemma \ref{lemQC} and Lemma \ref{lemQC3}, we conclude that the family
$\left\{  p_{K,\rho}\right\}  $ over all $q$-compact subsets $K\subseteq
U_{x}$ and $\rho<r_{y}$ is a defining family of seminorms of the Fr\'{e}chet
space $\mathcal{O}_{q}\left(  U\right)  $. Pick a $q$-compact subset
$K\subseteq U_{x}$, $\rho<r_{y}$, and $f,g\in\mathcal{O}_{q}\left(  U\right)
$. Then
\[
p_{K,\rho}\left(  f\cdot g\right)  \leq\sum_{n}\left(  \sum_{i+j=n}\left\Vert
f_{i}\right\Vert _{K}\left\Vert g_{j}\right\Vert _{K}\right)  \rho^{n}=\left(
\sum_{i}\left\Vert f_{i}\right\Vert _{K}\rho^{i}\right)  \left(  \sum
_{j}\left\Vert g_{j}\right\Vert _{K}\rho^{j}\right)  =p_{K,\rho}\left(
f\right)  p_{K,\rho}\left(  g\right)  ,
\]
that is, $f\cdot g\in\mathcal{O}_{q}\left(  U_{x}\right)  $ and $p_{K,\rho}$
is a multiplicative seminorm. Hence $\mathcal{O}_{q}\left(  U\right)  $ is a
Fr\'{e}chet $\widehat{\otimes}$-algebra. Finally, if $V\subseteq
U\subseteq\mathbb{C}_{xy}$ are $\left(  \mathfrak{q},\mathfrak{d}\right)
$\textit{-}open subsets and $f,g\in\mathcal{O}_{q}\left(  U\right)  $, then
\[
\left(  f\cdot g\right)  |_{V}=\sum_{n}\left(  \sum_{i+j=n}f_{i}\left(
x\right)  g_{j}\left(  q^{i}x\right)  \right)  |_{V_{x}}y^{n}=\sum_{n}\left(
\sum_{i+j=n}f_{i}|_{V_{x}}\left(  x\right)  g_{j}|_{V_{x}}\left(
q^{i}x\right)  \right)  y^{n}=\left(  f|_{V}\right)  \cdot\left(
g|_{V}\right)  .
\]
It means that the restriction map is a continuous homomorphism of the
Fr\'{e}chet $\widehat{\otimes}$-algebras.
\end{proof}

In the case of $U_{x}=B\left(  0,r_{x}\right)  $ one can replace the family
$\left\{  \left\Vert \cdot\right\Vert _{K}:K\subseteq U_{x}\right\}  $ of
seminorms $\mathcal{O}\left(  U_{x}\right)  $ by $\left\{  \left\Vert
\cdot\right\Vert _{\rho}:\rho<r_{x}\right\}  $ thanks to Lemma \ref{lemAK0}.
In this case, $p_{K,\rho}$ is replaced by
\[
p_{\rho}\left(  f\right)  =\sum_{n}\left\Vert f_{n}\right\Vert _{\rho_{x}}%
\rho_{y}^{n}\text{ with }\rho=\left(  \rho_{x},\rho_{y}\right)  <\left(
r_{x},r_{y}\right)  =r.
\]
Thus $\left\{  p_{\rho}:\rho<r\right\}  $ turns out to be a defining family of
multiplicative seminorms on $\mathcal{O}_{q}\left(  U\right)  $. In
particular, for the algebra of all global sections of the presheaf
$\mathcal{O}_{q}$ we obtain that
\[
\Gamma\left(  \mathbb{C}_{xy},\mathcal{O}_{q}\right)  =\mathcal{O}_{q}\left(
\mathbb{C}_{xy}\right)  =\mathcal{O}_{q}\left(  \mathbb{C}^{2}\right)
\]
is the same Fr\'{e}chet algebra considered above in Subsection \ref{sAME}.

\begin{corollary}
The stalk of the sheaf $\mathcal{O}_{q}$ at the origin admits the following
description%
\[
\left(  \mathcal{O}_{q}\right)  _{0}=\left\{  \left\langle U_{\varepsilon
},\sum_{n}f_{n}\left(  x\right)  y^{n}\right\rangle :f_{n}\in\mathcal{O}%
\left(  B\left(  0,\varepsilon_{x}\right)  \right)  ,\sum_{n}\left\Vert
f_{n}\right\Vert _{\varepsilon_{x}}\varepsilon_{y}^{n}<\infty,\varepsilon
>0\right\}  ,
\]
where $U_{\varepsilon}=B\left(  0,\varepsilon_{x}\right)  \cup B\left(
0,\varepsilon_{y}\right)  $, $\varepsilon=\left(  \varepsilon_{x}%
,\varepsilon_{y}\right)  $.
\end{corollary}

\begin{proof}
One needs to use Lemma \ref{lemQC3} and Lemma \ref{lemQC4}.
\end{proof}

By \textit{a noncommutative }$q$-\textit{plane} we mean the ringed space
$\left(  \mathbb{C}_{xy},\mathcal{O}_{q}\right)  $. We have also its opposite
plane $\left(  \mathbb{C}_{xy}^{\operatorname{op}},\mathcal{O}_{q}%
^{\operatorname{op}}\right)  $. A similar result takes place for
$\mathcal{O}_{q}^{\operatorname{op}}\left(  U\right)  $ whenever
$U\subseteq\mathbb{C}_{xy}^{\operatorname{op}}$ is a $\left(  \mathfrak{d}%
,\mathfrak{q}\right)  $-open subset. Namely, if $f=\sum_{n}x^{n}f_{n}\left(
y\right)  $ and $g=\sum_{n}x^{n}g_{n}\left(  y\right)  $ are elements of
$\mathcal{O}_{q}^{\operatorname{op}}\left(  U\right)  $, then
\[
f\cdot g=\sum_{n}x^{n}\left(  \sum_{i+j=n}f_{i}\left(  q^{j}y\right)
g_{j}\left(  y\right)  \right)  \in\mathcal{O}_{q}^{\operatorname{op}}\left(
U\right)  ,
\]
and it defines a Fr\'{e}chet $\widehat{\otimes}$-algebra structure on
$\mathcal{O}_{q}^{\operatorname{op}}\left(  U\right)  $.

\begin{corollary}
\label{corTwist}The twisting morphism $\left(  t,t^{\times}\right)  :\left(
\mathbb{C}_{xy}^{\operatorname{op}},\mathcal{O}_{q}^{\operatorname{op}%
}\right)  \rightarrow\left(  \mathbb{C}_{xy},\mathcal{O}_{q}\right)
^{\operatorname{op}}$ of the ringed spaces with
\begin{align*}
t  &  :\mathbb{C}_{xy}^{\operatorname{op}}\rightarrow\mathbb{C}_{xy},\quad
t\left(  z,w\right)  =\left(  w,z\right)  \text{,}\\
t^{\times}  &  :\left(  \mathcal{O}_{q}\right)  ^{\operatorname{op}%
}\rightarrow t_{\ast}\mathcal{O}_{q}^{\operatorname{op}},\quad t^{\times
}\left(  \sum f_{n}\left(  x\right)  y^{n}\right)  =\sum x^{n}f_{n}\left(
y\right)
\end{align*}
implements an isomorphism of the ringed spaces, where $\left(  \mathbb{C}%
_{xy},\mathcal{O}_{q}\right)  ^{\operatorname{op}}$ denotes the same ringed
space $\left(  \mathbb{C}_{xy},\mathcal{O}_{q}\right)  $ but with the opposite
multiplication in the structure sheaf $\mathcal{O}_{q}$.
\end{corollary}

\begin{proof}
The twisting map $t:\mathbb{C}_{xy}^{\operatorname{op}}\rightarrow
\mathbb{C}_{xy}$, $t\left(  x,y\right)  =\left(  y,x\right)  $ is obviously a
homeomorphism of the topological spaces. By Lemma \ref{lemQC3}, the mapping
\[
t^{\times}\left(  U\right)  :\mathcal{O}_{q}\left(  U_{x}\cup U_{y}\right)
^{\operatorname{op}}\rightarrow\mathcal{O}_{q}^{\operatorname{op}}\left(
U_{y}\cup U_{x}\right)  ,\quad t^{\times}\left(  \sum f_{n}\left(  x\right)
y^{n}\right)  =\sum x^{n}f_{n}\left(  y\right)
\]
is an isomorphism of the Fr\'{e}chet spaces, where $\mathcal{O}_{q}\left(
U_{x}\cup U_{y}\right)  ^{\operatorname{op}}$ is the same algebra
$\mathcal{O}_{q}\left(  U\right)  $ with the opposite multiplication. Notice
that
\begin{align*}
t^{\times}\left(  U\right)  \left(  f\cdot^{\operatorname{op}}g\right)   &
=t^{\times}\left(  U\right)  \left(  g\cdot f\right)  =t^{\times}\left(
U\right)  \left(  \sum_{n}\left(  \sum_{i+j=n}g_{i}\left(  x\right)
f_{j}\left(  q^{i}x\right)  \right)  y^{n}\right) \\
&  =\sum_{n}x^{n}\left(  \sum_{i+j=n}g_{i}\left(  y\right)  f_{j}\left(
q^{i}y\right)  \right)  =\sum_{n}x^{n}\left(  \sum_{j+i=n}f_{j}\left(
q^{i}y\right)  g_{i}\left(  y\right)  \right) \\
&  =t^{\times}\left(  U\right)  \left(  f\right)  \cdot t^{\times}\left(
U\right)  \left(  g\right)  ,
\end{align*}
which means that $t^{\times}$ implements a Fr\'{e}chet $\widehat{\otimes}%
$-algebra presheaf isomorphism by Lemma \ref{lemQC4}. Hence $\left(
t,t^{\times}\right)  $ is an isomorphism of the ringed spaces (see
\cite[2.1]{Harts}).
\end{proof}

In the case of $\left\vert q\right\vert >1$ we obtain the noncommutative
$q^{-1}$-plane $\left(  \mathbb{C}_{yx},\mathcal{O}_{q^{-1}}\right)  $ (see
Remark \ref{remCXCY}). Using Corollary \ref{corTwist}, we conclude that
$\left(  \mathbb{C}_{yx}^{\operatorname{op}},\mathcal{O}_{q^{-1}%
}^{\operatorname{op}}\right)  =\left(  \mathbb{C}_{yx},\mathcal{O}_{q^{-1}%
}\right)  ^{\operatorname{op}}$ up to an isomorphism of the ringed spaces.

Finally, as we have seen above there are canonical isomorphisms $\left(
\mathbb{C}_{x},\mathcal{O}_{q}|\mathbb{C}_{x}\right)  \rightarrow\left(
\mathbb{C}_{x},\mathcal{O}^{\mathfrak{q}}\right)  $ and $\left(
\mathbb{C}_{y},\mathcal{O}_{q}|\mathbb{C}_{y}\right)  \rightarrow\left(
\mathbb{C}_{y},\mathcal{O}^{\mathfrak{d}}|\mathbb{C}_{y}\right)  $ of the
ringed spaces. Notice also that if $\iota^{x}:\mathbb{C}_{x}\rightarrow
\mathbb{C}_{xy}$ is the canonical embedding then the inclusion morphism
$\iota_{\ast}^{x}\mathcal{O}^{\mathfrak{q}}\rightarrow\mathcal{O}_{q}$ is a
Fr\'{e}chet sheaf $\widehat{\otimes}$-algebra morphism. In a similar way, we
obtain the morphism $\iota_{\ast}^{y}\mathcal{O}^{\mathfrak{d}}\rightarrow
\mathcal{O}_{q}$.

\section{The decomposition of the sheaf $\mathcal{O}_{q}$ and functional
calculus}

In this section we prove the main results on the decomposition of the sheaf
$\mathcal{O}_{q}$ and the related functional calculus theorem in terms of the
Harte spectrum.

\subsection{The decomposition theorem}

The following decomposition theorem, which is similar to one from Proposition
\ref{tO}, holds for the stalks of the sheaf $\mathcal{O}_{q}$.

\begin{theorem}
\label{tdecomO}If $U\subseteq\mathbb{C}_{xy}$ is a $\left(  \mathfrak{q}%
,\mathfrak{d}\right)  $-open subset then the Fr\'{e}chet $\widehat{\otimes}%
$-algebra $\mathcal{O}_{q}\left(  U\right)  $ admits the following topological
direct sum decomposition
\[
\mathcal{O}_{q}\left(  U\right)  =\mathcal{O}^{\mathfrak{q}}\left(
U_{x}\right)  \oplus\operatorname{Rad}\mathcal{O}_{q}\left(  U\right)
\oplus\mathcal{I}_{\mathfrak{d}}\left(  U_{y}\right)
\]
into the closed subalgebras. In particular, the algebra $\mathcal{O}%
_{q}\left(  U\right)  $ is commutative modulo its Jacobson radical
$\operatorname{Rad}\mathcal{O}_{q}\left(  U\right)  $ and $\operatorname{Spec}%
\left(  \mathcal{O}_{q}\left(  U\right)  \right)  =U.$
\end{theorem}

\begin{proof}
Put $\mathcal{I}_{xy}\left(  U\right)  $ to be the set of all series
$f=\sum_{n\in\mathbb{N}}f_{n}\left(  x\right)  y^{n}\in\mathcal{O}_{q}\left(
U\right)  $ (see Lemma \ref{lemQC3}) such that $f_{n}\left(  0\right)  =0$.
Since $U_{y}$ is a disk and $\left\{  y^{n}:n\in\mathbb{N}\right\}  $ is an
absolute basis in $\mathcal{I}_{\mathfrak{d}}\left(  U_{y}\right)  $, it
follows that $\mathcal{I}_{xy}\left(  U\right)  =\mathcal{I}_{\mathfrak{q}%
}\left(  U_{x}\right)  \widehat{\otimes}\mathcal{I}_{\mathfrak{d}}\left(
U_{y}\right)  $ is a closed subspace of $\mathcal{O}_{q}\left(  U\right)  $.
Moreover, by Lemma \ref{lemQC3} and the Open Mapping Theorem for the
Fr\'{e}chet spaces \cite[3.2, Corollary 3 ]{Sch}, the topological direct sum
decomposition%
\[
\mathcal{O}_{q}\left(  U\right)  =\mathcal{O}^{\mathfrak{q}}\left(
U_{x}\right)  \oplus\mathcal{I}_{xy}\left(  U\right)  \oplus\mathcal{I}%
_{\mathfrak{d}}\left(  U_{y}\right)
\]
holds. For every $f\in\mathcal{O}_{q}\left(  U\right)  $ we have%
\[
f=\sum_{n\in\mathbb{Z}_{+}}f_{n}\left(  x\right)  y^{n}=f_{0}\left(  x\right)
+\sum_{n\in\mathbb{N}}\left(  f_{n}\left(  x\right)  -f_{n}\left(  0\right)
\right)  y^{n}+\sum_{n\in\mathbb{N}}f_{n}\left(  0\right)  y^{n}.
\]
If $f\in\mathcal{I}_{xy}\left(  U\right)  $ and $g\in\mathcal{O}_{q}\left(
U\right)  $, then $fg=\sum_{n\in\mathbb{N}}\left(  \sum_{i+j=n}f_{i}\left(
x\right)  g_{j}\left(  q^{i}x\right)  \right)  y^{n}$ with
\[
\sum_{i+j=n}f_{i}\left(  0\right)  g_{j}\left(  0\right)  =0\quad\text{for all
}n,
\]
which means that $\mathcal{I}_{xy}\left(  U\right)  $ is a closed two-sided
ideal of the Fr\'{e}chet algebra $\mathcal{O}_{q}\left(  U\right)  $. Note
also that $\mathcal{O}^{\mathfrak{q}}\left(  U_{x}\right)  $ and
$\mathcal{I}_{\mathfrak{d}}\left(  U_{y}\right)  $ are closed subalgebras of
$\mathcal{O}_{q}\left(  U\right)  $ with $\mathcal{I}_{\mathfrak{d}}\left(
U_{y}\right)  ^{+}=\mathcal{O}^{\mathfrak{d}}\left(  U_{y}\right)  $.

Further, $\mathcal{O}_{q}\left(  U\right)  /\mathcal{I}_{xy}\left(  U\right)
=\mathcal{O}^{\mathfrak{q}}\left(  U_{x}\right)  ^{\sim}\oplus\mathcal{I}%
_{\mathfrak{d}}\left(  U_{y}\right)  ^{\sim}$ is a semisimple, commutative,
Arens-Michael-Fr\'{e}chet algebra such that $\mathcal{I}_{\mathfrak{q}}\left(
U_{x}\right)  ^{\sim}\cdot\mathcal{I}_{\mathfrak{d}}\left(  U_{y}\right)
^{\sim}=\left\{  0\right\}  $. It follows that $\operatorname{Spec}\left(
\mathcal{O}_{q}\left(  U\right)  /\mathcal{I}_{xy}\left(  U\right)  \right)
=U_{x}\cup U_{y}=U$ \cite[5.3]{Hel}, and $\operatorname{Rad}\mathcal{O}%
_{q}\left(  U\right)  \subseteq\mathcal{I}_{xy}\left(  U\right)  $.

Now fix a $q$-compact subset $\omega\subseteq U_{x}$ and its $q$-open
neighborhood $V_{x}$ in $U_{x}$ with its compact closure $K$. If $f=\sum
_{n\in\mathbb{N}}f_{n}\left(  x\right)  y^{n}\in\mathcal{I}_{xy}\left(
U\right)  $ then $\sum_{n\in\mathbb{N}}\left\Vert f_{n}\right\Vert _{K}%
\rho^{n}<\infty$ for every $\rho<r_{y}$. By Lemma \ref{lemQC4}, $\mathcal{A}%
\left(  \omega\right)  \widehat{\otimes}\mathcal{A}\left(  \rho\right)  $ is a
Banach algebra being a Hausdorff (seminormed) completion of $\left(
\mathcal{O}_{q}\left(  U\right)  ,p_{\omega,\rho}\right)  $. The presence of
an absolute basis in $\mathcal{A}\left(  \rho\right)  $ allows us to conclude
that $\sigma\left(  x\right)  =\omega$ in the Banach algebra $\mathcal{A}%
\left(  \omega\right)  \widehat{\otimes}\mathcal{A}\left(  \rho\right)  $.
Namely, if $h\left(  x\right)  \left(  \sum_{n}g_{n}\left(  x\right)
y^{n}\right)  =1$ for some $h\left(  x\right)  \in\mathcal{A}\left(
\omega\right)  $, then $h\left(  x\right)  g_{0}\left(  x\right)  =1$.
Consider the closure $\mathcal{A}_{q}$ of $\mathfrak{A}_{q}$ in $\mathcal{A}%
\left(  \omega\right)  \widehat{\otimes}\mathcal{A}\left(  \rho\right)  $ with
its $x$-inverse closed hull $\mathcal{A}_{q,x}$. Notice that $\mathcal{A}_{q}$
is a Banach $q$-plane, and
\[
f|\mathcal{A}\left(  \omega\right)  \widehat{\otimes}\mathcal{A}\left(
\rho\right)  =\sum_{n\in\mathbb{N}}\left(  f_{n}|\omega\right)  \left(
x\right)  y^{n}\in\mathcal{A}_{q,x}.
\]
Using Proposition \ref{propTQA}, we deduce that $f|\mathcal{A}\left(
\omega\right)  \widehat{\otimes}\mathcal{A}\left(  \rho\right)  \in Q\left(
\mathcal{A}\left(  \omega\right)  \widehat{\otimes}\mathcal{A}\left(
\rho\right)  \right)  $. It follows that (see Lemma \ref{lemQC3})
\[
\sigma\left(  f\right)  =\cup\left\{  f|\mathcal{A}\left(  \omega\right)
\widehat{\otimes}\mathcal{A}\left(  \rho\right)  :\omega\subseteq U_{x}%
,\rho<r_{y}\right\}  =\left\{  0\right\}  ,
\]
which means that $\mathcal{I}_{xy}\left(  U\right)  \subseteq Q\left(
\mathcal{O}_{q}\left(  U\right)  \right)  $. Since $\mathcal{I}_{xy}\left(
U\right)  $ is an ideal, it follows that $\mathcal{I}_{xy}\left(  U\right)
\subseteq$ $\operatorname{Rad}\mathcal{O}_{q}\left(  U\right)  $. Hence
$\mathcal{I}_{xy}\left(  U\right)  =$ $\operatorname{Rad}\mathcal{O}%
_{q}\left(  U\right)  $.

Finally, $\operatorname{Spec}\left(  \mathcal{O}_{q}\left(  U\right)  \right)
=\operatorname{Spec}\left(  \mathcal{O}_{q}\left(  U\right)
/\operatorname{Rad}\mathcal{O}_{q}\left(  U\right)  \right)
=\operatorname{Spec}\left(  \mathcal{O}_{q}\left(  U\right)  /\mathcal{I}%
_{xy}\left(  U\right)  \right)  =U$.
\end{proof}

\begin{corollary}
The stalk $\left(  \mathcal{O}_{q}\right)  _{0}$ of the sheaf $\mathcal{O}%
_{q}$ at the origin is a noncommutative local algebra, which admits the
following decomposition
\[
\left(  \mathcal{O}_{q}\right)  _{0}=\mathcal{O}_{0}\oplus I\oplus
\mathcal{I}_{\mathfrak{d},0},
\]
where $I=\underrightarrow{\lim}\left\{  \operatorname{Rad}\mathcal{O}%
_{q}\left(  U\right)  \right\}  $ is the two-sided ideal of $\left(
\mathcal{O}_{q}\right)  _{0}$. It has the unique maximal ideal $\mathfrak{m}%
=\mathcal{I}_{\mathfrak{q},0}\oplus I\oplus\mathcal{I}_{\mathfrak{d},0}$,
which is $\operatorname{Rad}\left(  \mathcal{O}_{q}\right)  _{0}$.
\end{corollary}

\begin{proof}
Since the presheaf and the related sheaf have the same stalks at every point,
we deduce that
\begin{align*}
\left(  \mathcal{O}_{q}\right)  _{0}  &  =\underrightarrow{\lim}\left\{
\mathcal{O}_{q}\left(  U\right)  \right\}  =\underrightarrow{\lim}\left\{
\mathcal{O}\left(  U_{x}\right)  \right\}  \oplus\underrightarrow{\lim
}\left\{  \mathcal{I}_{xy}\left(  U\right)  \right\}  \oplus
\underrightarrow{\lim}\left\{  \mathcal{I}_{y}\left(  U_{y}\right)  \right\}
\\
&  =\mathcal{O}_{0}\oplus\underrightarrow{\lim}\left\{  \operatorname{Rad}%
\mathcal{O}_{q}\left(  U\right)  \right\}  \oplus\mathcal{I}_{\mathfrak{d}%
,0}=\mathcal{O}_{0}\oplus I\oplus\mathcal{I}_{\mathfrak{d},0}%
\end{align*}
thanks to Theorem \ref{tdecomO}. Notice that $\mathfrak{m}=\mathcal{I}%
_{\mathfrak{q},0}\oplus I\oplus\mathcal{I}_{\mathfrak{d},0}$ is a maximal
ideal of $\left(  \mathcal{O}_{q}\right)  _{0}$ with its residue field
$\left(  \mathcal{O}_{q}\right)  _{0}/\mathfrak{m=}\mathcal{O}_{0}%
/\mathcal{I}_{\mathfrak{q},0}=\mathbb{C}$.

If $\left\langle U,f\right\rangle \in I$ and $\left\langle V,g\right\rangle
\in\left(  \mathcal{O}_{q}\right)  _{0}$, then $\left\langle W,f\right\rangle
\in\operatorname{Rad}\mathcal{O}_{q}\left(  W\right)  $ for a $q$-open subset
$W\subseteq U\cap V$ and $1-\left\langle W,g\right\rangle \left\langle
W,f\right\rangle $ is invertible in $\mathcal{O}_{q}\left(  W\right)  $, which
in turn implies that $\left\langle U,f\right\rangle \in\operatorname{Rad}%
\left(  \mathcal{O}_{q}\right)  _{0}$. Thus $I\subseteq\operatorname{Rad}%
\left(  \mathcal{O}_{q}\right)  _{0}$. In particular, every maximal (left or
right) ideal $\mathfrak{n}$ of $\left(  \mathcal{O}_{q}\right)  _{0}$ contains
the ideal $I$, which in turn is identified with a maximal ideal of the
commutative algebra $\left(  \mathcal{O}_{q}\right)  _{0}/I=\mathcal{O}%
_{0}\oplus\mathcal{I}_{\mathfrak{d},0}$ with $\mathcal{I}_{\mathfrak{q}%
,0}\mathcal{I}_{\mathfrak{d},0}=\left\{  0\right\}  $. If $\mathcal{I}%
_{\mathfrak{q},0}\nsubseteq\mathfrak{n}$ then $f\notin\mathfrak{n}$ for some
$f\in\mathcal{I}_{\mathfrak{q},0}$, whereas $f\mathcal{I}_{\mathfrak{d}%
,0}=\left\{  0\right\}  \subseteq\mathfrak{n}$. It follows that $\mathcal{I}%
_{\mathfrak{d},0}\subseteq\mathfrak{n}$ being $\mathfrak{n}$ a prime ideal. In
particular, $\mathfrak{n}$ is identified with a maximal ideal of the quotient
algebra $\left(  \left(  \mathcal{O}_{q}\right)  _{0}/I\right)  /\mathcal{I}%
_{\mathfrak{d},0}=\mathcal{O}_{0}$. But $\mathcal{O}_{0}$ is local, therefore
$\mathfrak{n=}\mathcal{I}_{\mathfrak{q},0}$ or $\mathfrak{n=m}$.
\end{proof}

Based on Theorem \ref{tdecomO}, we deduce that the sheaf $\mathcal{O}_{q}$
admits the following decomposition
\[
\mathcal{O}_{q}=\mathcal{O}^{\mathfrak{q}}\oplus\mathcal{I}_{xy}%
\oplus\mathcal{I}_{\mathfrak{d}}%
\]
with a two-sided ideal subsheaf $\mathcal{I}_{xy}$, which is the sheaf
associated to the presheaf $\operatorname{Rad}\mathcal{O}_{q}$.

\subsection{The functional calculus and spectral mapping theorem}

The functional calculus problem is solved in the following way.

\begin{theorem}
\label{thFunCal}Let $\mathfrak{A}_{q}\rightarrow\mathcal{B}$, $x\mapsto T$,
$y\mapsto S$ be a representation given be a pair $\left(  T,S\right)  $ of
elements of a unital Banach algebra $\mathcal{B}$ such that $TS=q^{-1}ST$, and
let $U\subseteq\mathbb{C}_{xy}$ be a $\left(  \mathfrak{q},\mathfrak{d}%
\right)  $-open subset. The homomorphism $\mathfrak{A}_{q}\rightarrow
\mathcal{B}$ extends up to a continuous algebra homomorphism $\mathcal{O}%
_{q}\left(  U\right)  \rightarrow\mathcal{B}$ if and only if $\sigma\left(
T\right)  \subseteq U_{x}$ and $\sigma\left(  S\right)  \subseteq U_{y}$.
\end{theorem}

\begin{proof}
If such a calculus $\mathcal{O}_{q}\left(  U\right)  \rightarrow\mathcal{B}$
does exist, then there are continuous homomorphisms $\mathcal{O}\left(
U_{x}\right)  \rightarrow\mathcal{B}$, $x\mapsto T$ and $\mathcal{O}\left(
U_{y}\right)  \rightarrow\mathcal{B}$, $y\mapsto S$ by virtue of Theorem
\ref{tdecomO}. It follows that $\sigma\left(  T\right)  \subseteq U_{x}$ and
$\sigma\left(  S\right)  \subseteq U_{y}$ (see \cite[2.2.15]{Hel}).

Conversely, if $\sigma\left(  T\right)  \subseteq U_{x}$ and $\sigma\left(
S\right)  \subseteq U_{y}$, then there are natural continuous homomorphisms
$\mathcal{O}\left(  U_{x}\right)  \rightarrow\mathcal{B}$, $x\mapsto T$ and
$\mathcal{O}\left(  U_{y}\right)  \rightarrow\mathcal{B}$, $y\mapsto S$, which
are the holomorphic functional calculi. In particular, they are factorized
through bounded homomorphisms $\mathcal{A}\left(  K\right)  \rightarrow
\mathcal{B}$, $x\mapsto T$ and $\mathcal{A}\left(  \rho\right)  \rightarrow
\mathcal{B}$, $y\mapsto S$ for some $q$-compact subset $K\subseteq U_{x}$ and
$\rho<r_{y}$ (see Lemma \ref{lemAK0}). If $f=\sum_{n}f_{n}\left(  x\right)
y^{n}\in\mathcal{O}_{q}\left(  U\right)  $ then
\[
\left\Vert \sum_{n}f_{n}\left(  T\right)  S^{n}\right\Vert \leq\sum
_{n}\left\Vert f_{n}\left(  T\right)  \right\Vert \left\Vert S^{n}\right\Vert
\leq C\sum_{n}\left\Vert f_{n}\right\Vert _{K}\rho^{n}=Cp_{K,\rho}\left(
f\right)
\]
(see Lemma \ref{lemQC4}) for some positive $C$. It means that we come up with
the continuous algebra homomorphism $\mathcal{O}_{q}\left(  U\right)
\rightarrow\mathcal{B}$, $\sum_{n}f_{n}\left(  x\right)  y^{n}\mapsto\sum
_{n}f_{n}\left(  T\right)  S^{n}$ extending $\mathfrak{A}_{q}\rightarrow
\mathcal{B}$. Notice that the latter homomorphism is factorized through
$\mathcal{A}\left(  K\right)  \widehat{\otimes}\mathcal{A}\left(  \rho\right)
\rightarrow\mathcal{B}$ (see Proposition \ref{propTQA}).
\end{proof}

As above let $\mathcal{B}$ be a unital Banach algebra, which contains a
contractive Banach $q$-plane $\mathcal{A}_{q}$ generated by a pair $\left(
T,S\right)  $ with $TS=q^{-1}ST$. We assume that $\mathcal{B}$ is commutative
modulo its Jacobson radical $\operatorname{Rad}\mathcal{B}$. If $\mathcal{B=A}%
_{q}$ then the assumption is equivalent to $\mathcal{I}_{TS}\subseteq
\operatorname{Rad}\mathcal{B}$ by Proposition \ref{pequi}. That is the case of
an operator $q$-plane with $T$ or $S$ to be a compact operator (see
Proposition \ref{pAcom}). In this case, the (open) group $R$ of the invertible
elements of $\mathcal{B}$ defines the regularity (see \cite[Lemma
4.3]{DosievIE}), that is, $ab\in R$ iff $a,b\in R$. That regularity $R$ in
turn defines a (noncommutative) subspectrum $\tau_{R}$ on $\mathcal{B}$, which
is reduced to the known Harte spectrum $\sigma_{\operatorname{H}}$
\cite[Theorem 5.7]{DosievIE}. Moreover, the joint Harte spectrum
$\sigma_{\operatorname{H}}\left(  T,S\right)  $ of the pair $\left(
T,S\right)  $ from the algebra $\mathcal{B}$ is reduced to $\sigma
_{\operatorname{H}}\left(  T^{\sim},S^{\sim}\right)  $ (see \cite[Lemma
3.2]{DosievIE}), where $T^{\sim},S^{\sim}\in\mathcal{B}/\operatorname{Rad}%
\mathcal{B}$. But $T^{\sim}S^{\sim}=0$, therefore
\begin{align*}
\sigma_{\operatorname{H}}\left(  T^{\sim},S^{\sim}\right)   &  =\left\{
\left(  \lambda,0\right)  ,\left(  0,\mu\right)  :\lambda\in\sigma
_{\operatorname{H}}\left(  T^{\sim}\right)  ,\mu\in\sigma_{\operatorname{H}%
}\left(  S^{\sim}\right)  \right\} \\
&  =\left\{  \left(  \lambda,0\right)  ,\left(  0,\mu\right)  :\lambda
\in\sigma_{\operatorname{H}}\left(  T\right)  ,\mu\in\sigma_{\operatorname{H}%
}\left(  S\right)  \right\} \\
&  =\left(  \sigma_{\operatorname{H}}\left(  T\right)  \times\left\{
0\right\}  \right)  \cup\left(  \left\{  0\right\}  \times\sigma
_{\operatorname{H}}\left(  S\right)  \right) \\
&  =\sigma_{\operatorname{H}}\left(  T\right)  \cup\sigma_{\operatorname{H}%
}\left(  S\right)  \subseteq\mathbb{C}_{xy}%
\end{align*}
by virtue of the spectral mapping property of $\sigma_{\operatorname{H}}$.
Thus $\sigma_{\operatorname{H}}\left(  T,S\right)  =\sigma_{\operatorname{H}%
}\left(  T\right)  \cup\sigma_{\operatorname{H}}\left(  S\right)
\subseteq\mathbb{C}_{xy}$.

\begin{corollary}
\label{corFCST}Let $\mathcal{B}$ be a unital Banach algebra, which is
commutative modulo its Jacobson radical, and contains a pair $\left(
T,S\right)  $ with $TS=q^{-1}ST$, and let $U\subseteq\mathbb{C}_{xy}$ be a
$\left(  \mathfrak{q},\mathfrak{d}\right)  $-open subset. Then homomorphism
$\mathfrak{A}_{q}\rightarrow\mathcal{B}$, $x\mapsto T$, $y\mapsto S$ extends
up to a continuous algebra homomorphism $\mathcal{O}_{q}\left(  U\right)
\rightarrow\mathcal{B}$, $f\mapsto f\left(  T,S\right)  $ if and only if
$\sigma_{\operatorname{H}}\left(  T,S\right)  \subseteq U$. In this case,%
\[
\sigma_{\operatorname{H}}\left(  f\left(  T,S\right)  \right)  =f\left(
\sigma_{\operatorname{H}}\left(  T,S\right)  \right)
\]
for every tuple $f\left(  x,y\right)  $ of noncommutative holomorphic
functions from $\mathcal{O}_{q}\left(  U\right)  $.
\end{corollary}

\begin{proof}
One needs to apply Theorem \ref{thFunCal} and the previous argument with the
noncommutative Harte spectrum.

Let us prove the spectral mapping formula for a tuple $f=\left(  f_{1}%
,\ldots,f_{m}\right)  $ in $\mathcal{O}_{q}\left(  U\right)  $. By Theorem
\ref{tdecomO}, we have $f_{i}\left(  x,y\right)  =g_{i}\left(  x\right)
+r_{i}\left(  x,y\right)  +h_{i}\left(  y\right)  $ with $g_{i}\left(
x\right)  \in\mathcal{O}\left(  U_{x}\right)  $, $r_{i}\left(  x,y\right)
\in\operatorname{Rad}\mathcal{O}_{q}\left(  U\right)  $, $h_{i}\left(
y\right)  \in\mathcal{I}_{\mathfrak{d}}\left(  U_{y}\right)  $. Thus $f\left(
x,y\right)  =g\left(  x\right)  +r\left(  x,y\right)  +h\left(  y\right)  $ is
the sum of the tuples. Since the functional calculus $\mathcal{O}_{q}\left(
U\right)  \rightarrow\mathcal{B}$ is an algebra homomorphism, it follows that
$r_{i}\left(  T,S\right)  \in\operatorname{Rad}\mathcal{B}$. Then
\begin{align*}
\sigma_{\operatorname{H}}\left(  f\left(  T,S\right)  \right)   &
=\sigma_{\operatorname{H}}\left(  f\left(  T,S\right)  ^{\sim}%
\operatorname{mod}\operatorname{Rad}\mathcal{B}\right)  =\sigma
_{\operatorname{H}}\left(  g\left(  T\right)  ^{\sim}+h\left(  S\right)
^{\sim}\right) \\
&  =\sigma_{\operatorname{H}}\left(  g\left(  T^{\sim}\right)  +h\left(
S^{\sim}\right)  \right)  =\left(  g+h\right)  \left(  \sigma
_{\operatorname{H}}\left(  T^{\sim},S^{\sim}\right)  \right)  =\left(
g+h\right)  \left(  \sigma_{\operatorname{H}}\left(  T,S\right)  \right) \\
&  =g\left(  \sigma_{\operatorname{H}}\left(  T\right)  \right)  +h\left(
\sigma_{\operatorname{H}}\left(  S\right)  \right)  =f\left(  \sigma
_{\operatorname{H}}\left(  T\right)  \cup\sigma_{\operatorname{H}}\left(
S\right)  \right)  =f\left(  \sigma_{\operatorname{H}}\left(  T,S\right)
\right)  .
\end{align*}
Notice that the action of every $f_{i}\left(  x,y\right)  $ on $U$ is given by
$\lambda\left(  f_{i}\left(  x,y\right)  \right)  $ for every $\lambda\in$
$\operatorname{Spec}\left(  \mathcal{O}_{q}\left(  U\right)  \right)  $ (see
Theorem \ref{tdecomO}).
\end{proof}

\subsection{Examples\label{subsecEx}}

Let us illustrate the assertion of Corollary \ref{corFCST} in the case of
Example \ref{exCom}. Namely, consider the shift operator $T\in\mathcal{B}%
\left(  \ell_{2}\right)  $, $T\left(  e_{n}\right)  =e_{n+1}$, and the
diagonal operator $S\in\mathcal{B}\left(  \ell_{2}\right)  $, $S\left(
e_{n}\right)  =q^{n}e_{n}$. The pair $\left(  T,S\right)  $ generates the
operator $q$-plane $\mathcal{B=A}_{q}$ in $\mathcal{B}\left(  \ell_{2}\right)
$, which is commutative modulo its Jacobson radical by Proposition
\ref{pAcom}. Notice that $\mathbb{D}_{1}=\sigma\left(  T\right)
\subseteq\sigma_{\operatorname{H}}\left(  T\right)  $ and $\left\{  1\right\}
_{q}=\left\{  q^{n}:n\in\mathbb{Z}_{+}\right\}  \cup\left\{  0\right\}
=\sigma\left(  S\right)  \subseteq\sigma_{\operatorname{H}}\left(  S\right)
$. But $\left\Vert T\right\Vert =1$, therefore for every $\lambda
\notin\mathbb{D}_{1}$ we have $\left(  \lambda-T\right)  ^{-1}=\sum
_{n\in\mathbb{Z}_{+}}\lambda^{-n-1}T^{n}\in\mathcal{A}_{q}$, which means that
$\sigma_{\operatorname{H}}\left(  T\right)  =\mathbb{D}_{1}$ too. In a similar
way, we have $\left\Vert S\right\Vert =1$ and $\sigma_{\operatorname{H}%
}\left(  S\right)  \subseteq\mathbb{D}_{1}$. Hence
\[
\mathbb{D}_{1}\cup\left\{  1\right\}  _{q}\subseteq\sigma_{\operatorname{H}%
}\left(  T,S\right)  =\mathbb{D}_{1}\cup\sigma_{\operatorname{H}}\left(
S\right)  \subseteq\mathbb{D}_{1}\cup\mathbb{D}_{1}\subseteq\mathbb{C}_{xy},
\]
and $\sigma_{\operatorname{H}}\left(  T,S\right)  $ is a quasicompact subset
of $\mathbb{C}_{xy}$. Let us consider the following $\left(  \mathfrak{q}%
,\mathfrak{d}\right)  $-open subset $U\subseteq\mathbb{C}_{xy}$ with
$U_{x}=B\left(  0,\left(  3/2\right)  ^{1/2}\right)  =U_{y}$, and let
$f\left(  x,y\right)  =\ln\left(  \frac{3}{2}+xy\right)  $ be an element of
the algebra $\mathcal{O}_{q}\left(  U\right)  $. Notice that
\begin{align*}
f\left(  x,y\right)   &  =\ln\left(  \frac{3}{2}+xy\right)  =\ln\left(
\frac{3}{2}\right)  +\sum_{n=1}^{\infty}\frac{\left(  -1\right)  ^{n+1}}%
{n}\left(  \frac{2}{3}\right)  ^{n}\left(  xy\right)  ^{n}\\
&  =\ln\left(  \frac{3}{2}\right)  +\sum_{n=1}^{\infty}\frac{\left(
-1\right)  ^{n+1}}{n}\left(  \frac{2}{3}\right)  ^{n}q^{n\left(  n-1\right)
/2}x^{n}y^{n}\in\mathcal{O}\left(  U_{x}\right)  \oplus\operatorname{Rad}%
\mathcal{O}_{q}\left(  U\right)
\end{align*}
by Theorem \ref{tdecomO}. Using Corollary \ref{corFCST}, we deduce that
$f\left(  T,S\right)  =\ln\left(  \frac{3}{2}+TS\right)  \in\mathcal{A}_{q}$
and%
\[
\sigma_{\operatorname{H}}\left(  f\left(  T,S\right)  \right)  =f\left(
\sigma_{\operatorname{H}}\left(  T,S\right)  \right)  =\left\{  \left(
\ln\left(  \frac{3}{2}\right)  ,0\right)  \right\}  \subseteq\mathbb{C}_{x}.
\]

Now consider the $\left(  \mathfrak{q},\mathfrak{d}\right)  $-open subset
$U\subseteq\mathbb{C}_{xy}$ with $U_{x}=B\left(  0,3/2\right)  =U_{y}$, and
put
\[
f\left(  x,y\right)  =\ln\left(  \frac{3}{2}+x\right)  +\sum_{n=1}^{\infty
}\left(  \frac{2}{3}\right)  ^{n}\left(  \ln\left(  \frac{3}{2}+\frac{1}%
{n}+x\right)  -\ln\left(  \frac{3}{2}+\frac{1}{n}\right)  \right)  y^{n}%
+\frac{y}{y-3/2}%
\]
to be a formal series in $\mathcal{O}\left(  U_{x}\right)  \left[  \left[
y\right]  \right]  $. Note that
\[
f_{n}\left(  x\right)  =\left(  \frac{2}{3}\right)  ^{n}\left(  \ln\left(
\frac{3}{2}+\frac{1}{n}+x\right)  -\ln\left(  \frac{3}{2}+\frac{1}{n}\right)
\right)  =\left(  \frac{2}{3}\right)  ^{n}\sum_{k\in\mathbb{N}}\frac{\left(
-1\right)  ^{k+1}}{k}\left(  \frac{3}{2}+\frac{1}{n}\right)  ^{-k}x^{k}%
\]
with
\[
\left\Vert f_{n}\right\Vert _{\rho_{x}}\leq\left(  \frac{2}{3}\right)
^{n}\sum_{k\in\mathbb{N}}\frac{1}{k}\left(  \frac{3}{2}+\frac{1}{n}\right)
^{-k}\rho_{x}^{k}\leq-\left(  \frac{2}{3}\right)  ^{n}\ln\left(  1-\frac
{\rho_{x}}{\frac{3}{2}+\frac{1}{n}}\right)
\]
for all $n\in\mathbb{N}$ and $\rho_{x}<3/2$, that is, $\left\{  f_{n}\left(
x\right)  \right\}  \subseteq\mathcal{I}_{\mathfrak{q}}\left(  U_{x}\right)
$. If $r\left(  x,y\right)  =\sum_{n=1}^{\infty}f_{n}\left(  x\right)  y^{n}$
and $\rho<3/2$, then
\[
p_{\rho}\left(  r\left(  x,y\right)  \right)  =\sum_{n}\left\Vert
f_{n}\right\Vert _{\rho_{x}}\rho_{y}^{n}\leq\sum_{n}-\left(  \frac{2}%
{3}\right)  ^{n}\ln\left(  1-\frac{\rho_{x}}{\frac{3}{2}+\frac{1}{n}}\right)
\rho_{y}^{n}%
\]
with
\[
\lim\sup_{n}\left(  -\left(  \frac{2}{3}\right)  ^{n}\ln\left(  1-\frac
{\rho_{x}}{\frac{3}{2}+\frac{1}{n}}\right)  \right)  ^{1/n}=\frac{2}{3}.
\]
Hence $p_{\rho}\left(  f\right)  <\infty$ for all $\rho<3/2$. By Theorem
\ref{tdecomO}, we conclude that $r\left(  x,y\right)  \in\operatorname{Rad}%
\mathcal{O}_{q}\left(  U\right)  $ and
\[
f\left(  x,y\right)  =\ln\left(  \frac{3}{2}+x\right)  +r\left(  x,y\right)
+\frac{y}{y-3/2}\in\mathcal{O}\left(  U_{x}\right)  \oplus\operatorname{Rad}%
\mathcal{O}_{q}\left(  U\right)  \oplus\mathcal{I}_{y}\left(  U_{y}\right)  ,
\]
that is, $f\left(  x,y\right)  \in\mathcal{O}_{q}\left(  U\right)  $. Using
Corollary \ref{corFCST}, we deduce that
\[
\sigma_{\operatorname{H}}\left(  f\left(  T,S\right)  \right)  =\left\{
\left(  \ln\left(  \frac{3}{2}+z\right)  ,0\right)  ,\left(  0,\frac{w}%
{w-3/2}\right)  :z\in\mathbb{D}_{1},w\in\sigma_{\operatorname{H}}\left(
S\right)  \right\}  \subseteq\mathbb{C}_{xy}.
\]

\subsection{Concluding remarks on the Taylor spectrum and localizations}

The quantum plane $\mathfrak{A}_{q}$ possesses (see \cite{Tah}, \cite{WM}) the
following free $\mathfrak{A}_{q}$-bimodule resolution
\begin{equation}
0\leftarrow\mathfrak{A}_{q}\overset{\pi}{\longleftarrow}\mathfrak{A}%
_{q}\otimes\mathfrak{A}_{q}\overset{d_{0}}{\longleftarrow}\mathfrak{A}%
_{q}\otimes\mathbb{C}^{2}\otimes\mathfrak{A}_{q}\overset{d_{1}}{\longleftarrow
}\mathfrak{A}_{q}\otimes\wedge^{2}\mathbb{C}^{2}\otimes\mathfrak{A}%
_{q}\leftarrow0, \label{Ures}%
\end{equation}
with the mapping $\pi\left(  a\otimes b\right)  =ab$ and the differentials
\begin{align*}
d_{0}\left(  a\otimes e_{1}\otimes b\right)   &  =a\otimes xb-ax\otimes
b,\quad d_{0}\left(  a\otimes e_{2}\otimes b\right)  =a\otimes yb-ay\otimes
b,\\
d_{1}\left(  a\otimes e_{1}\wedge e_{2}\otimes b\right)   &  =a\otimes
e_{2}\otimes xb-qax\otimes e_{2}\otimes b-qa\otimes e_{1}\otimes yb+ay\otimes
e_{1}\otimes b,
\end{align*}
where $a,b\in\mathfrak{A}_{q}$ and $\left(  e_{1},e_{2}\right)  $ is the
standard basis in $\mathbb{C}^{2}$. One of the central results of \cite{Pir}
asserts that the canonical embedding $\mathfrak{A}_{q}\rightarrow
\mathcal{O}_{q}\left(  \mathbb{C}_{xy}\right)  $ is a localization in the
sense of Taylor \cite{Tay2}. Using the resolution (\ref{Ures}), we derive that
the following complex%
\[
0\leftarrow\mathcal{O}_{q}\left(  \mathbb{C}_{xy}\right)  \overset{\pi
}{\longleftarrow}\mathcal{O}_{q}\left(  \mathbb{C}_{xy}\right)
^{\widehat{\otimes}2}\overset{d_{0}}{\longleftarrow}\mathcal{O}_{q}\left(
\mathbb{C}_{xy}\right)  \widehat{\otimes}\mathbb{C}^{2}\widehat{\otimes
}\mathcal{O}_{q}\left(  \mathbb{C}_{xy}\right)  \overset{d_{1}}{\longleftarrow
}\mathcal{O}_{q}\left(  \mathbb{C}_{xy}\right)  \widehat{\otimes}\wedge
^{2}\mathbb{C}^{2}\widehat{\otimes}\mathcal{O}_{q}\left(  \mathbb{C}%
_{xy}\right)  \leftarrow0
\]
is admissible, which provides a free $\mathcal{O}_{q}\left(  \mathbb{C}%
_{xy}\right)  $-bimodule resolution of $\mathcal{O}_{q}\left(  \mathbb{C}%
_{xy}\right)  $ (see \cite[3.2]{HelHom}). For the differentials of the
resolution one can use the following matrix representations
\[
d_{0}=\left[
\begin{array}
[c]{cc}%
1\otimes L_{x}-R_{x}\otimes1 & 1\otimes L_{y}-R_{y}\otimes1
\end{array}
\right]  \text{, }d_{1}=\left[
\begin{array}
[c]{c}%
R_{y}\otimes1-q1\otimes L_{y}\\
1\otimes L_{x}-qR_{x}\otimes1
\end{array}
\right]  ,
\]
where $L$ and $R$ indicate to the left and right multiplication operators on
$\mathcal{O}_{q}\left(  \mathbb{C}_{xy}\right)  $. It follows that every
Fr\'{e}chet $\mathcal{O}_{q}\left(  \mathbb{C}_{xy}\right)  $-bimodule
possesses a similar resolution. By Lemma \ref{lemQC4}, the Fr\'{e}chet
$\widehat{\otimes}$-algebra $\mathcal{O}_{q}\left(  U\right)  $ over a
$\left(  \mathfrak{q},\mathfrak{d}\right)  $\textit{-}open subset
$U\subseteq\mathbb{C}_{xy}$ has the same $\mathcal{O}_{q}\left(  U\right)
$-bimodule resolution. If $X$ is a left Fr\'{e}chet $\mathcal{O}_{q}\left(
U\right)  $-module, then by applying the functor $\circ\underset{\mathcal{O}%
_{q}\left(  U\right)  }{\widehat{\otimes}}X$ to the just indicated bimodule
resolution of $\mathcal{O}_{q}\left(  U\right)  $ we obtain a similar free
left $\mathcal{O}_{q}\left(  U\right)  $-module resolution for $X$. Namely,
the following complex
\begin{equation}
0\leftarrow X\overset{\pi^{\left(  X\right)  }}{\longleftarrow}\mathcal{O}%
_{q}\left(  U\right)  \widehat{\otimes}X\overset{d_{0}^{\left(  X\right)
}}{\longleftarrow}\mathcal{O}_{q}\left(  U\right)  \widehat{\otimes}%
\mathbb{C}^{2}\widehat{\otimes}X\overset{d_{1}^{\left(  X\right)
}}{\longleftarrow}\mathcal{O}_{q}\left(  U\right)  \widehat{\otimes}\wedge
^{2}\mathbb{C}^{2}\widehat{\otimes}X\leftarrow0, \label{reX}%
\end{equation}
is admissible, where $\pi^{\left(  X\right)  }\left(  f\otimes\xi\right)
=f\xi$ and
\begin{align*}
d_{0}^{\left(  X\right)  }\left(  f\otimes e_{1}\otimes\xi\right)   &
=f\otimes x\xi-ax\otimes\xi,\quad d_{0}\left(  f\otimes e_{2}\otimes
\xi\right)  =f\otimes y\xi-fy\otimes\xi,\\
d_{1}\left(  f\otimes e_{1}\wedge e_{2}\otimes\xi\right)   &  =f\otimes
e_{2}\otimes x\xi-qfx\otimes e_{2}\otimes\xi-qf\otimes e_{1}\otimes
y\xi+fy\otimes e_{1}\otimes\xi,
\end{align*}
for all $f\in\mathcal{O}_{q}\left(  U\right)  $ and $\xi\in X$.

Now let $X$ be a left Banach $\mathfrak{A}_{q}$-module such that the
generators $x$ and $y$ of $\mathfrak{A}_{q}$ act on $X$ as the bounded linear
operators $T,S\in\mathcal{B}\left(  X\right)  $ such that $TS=q^{-1}ST$. Since
$\mathcal{O}_{q}\left(  \mathbb{C}_{xy}\right)  $ is the Arens-Michael
envelope of $\mathfrak{A}_{q}$ (see Subsection \ref{sAME}), $X$ turns out to
be a left Banach $\mathcal{O}_{q}\left(  \mathbb{C}_{xy}\right)  $-module by
extending the original $\mathfrak{A}_{q}$-module action on $X$.

One can introduce the spectrum of a left Banach $\mathcal{O}_{q}\left(
\mathbb{C}_{xy}\right)  $-module $X$ within the general framework of the
spectral theory of a left module over the algebra of global sections of a
Fr\'{e}chet sheaf \cite{Dos}, \cite{DosJOT10}. In the commutative case that
approach is due to M. Putinar \cite{Put}. That is the case of $q=1$.

\textit{The resolvent set }$\operatorname{res}\left(  T,S\right)  $ of these
operators (or the module $X$) is defined as a set of those $\gamma
\in\operatorname{Spec}\left(  \mathcal{O}_{q}\left(  \mathbb{C}_{xy}\right)
\right)  $ such that $\mathbb{C}\left(  \gamma\right)  \perp_{\mathcal{O}%
_{q}\left(  \mathbb{C}_{xy}\right)  }X$. The set
\[
\sigma\left(  T,S\right)  =\operatorname{Spec}\left(  \mathcal{O}_{q}\left(
\mathbb{C}_{xy}\right)  \right)  \backslash\operatorname{res}\left(
T,S\right)
\]
is called\textit{ the joint spectrum of the operator pair} $\left(
T,S\right)  $ (or the module $X$).

Recall that $\mathbb{C}\left(  \gamma\right)  \perp_{\mathcal{O}_{q}\left(
\mathbb{C}_{xy}\right)  }X$ means that $\operatorname{Tor}_{k}^{\mathcal{O}%
_{q}\left(  \mathbb{C}_{xy}\right)  }\left(  \mathbb{C}\left(  \gamma\right)
,X\right)  =\left\{  0\right\}  $ for all $k\geq0$. The homology groups
$\operatorname{Tor}_{k}^{\mathcal{O}_{q}\left(  \mathbb{C}_{xy}\right)
}\left(  \mathbb{C}\left(  \gamma\right)  ,X\right)  $ can be calculated by
means of the resolution (\ref{reX}). By applying the functor $\mathbb{C}%
\left(  \gamma\right)  \underset{\mathcal{O}_{q}\left(  \mathbb{C}%
_{xy}\right)  }{\widehat{\otimes}}\circ$ to the resolution (\ref{reX}), we
derive that $\operatorname{Tor}_{k}^{\mathcal{O}_{q}\left(  \mathbb{C}%
_{xy}\right)  }\left(  \mathbb{C}\left(  \gamma\right)  ,X\right)  $ are the
homology groups of the following complex
\[
0\leftarrow X\overset{d_{0,\gamma}}{\longleftarrow}\mathbb{C}^{2}\otimes
X\overset{d_{1,\gamma}}{\longleftarrow}\wedge^{2}\mathbb{C}^{2}\otimes
X\leftarrow0
\]
where
\begin{align*}
d_{0,\gamma}\left(  e_{1}\otimes\xi\right)   &  =\left(  T-\gamma\left(
x\right)  \right)  \xi,\quad d_{0,\gamma}\left(  e_{2}\otimes\eta\right)
=\left(  S-\gamma\left(  y\right)  \right)  \eta,\\
d_{1,\gamma}\left(  e_{1}\wedge e_{2}\otimes\xi\right)   &  =-e_{1}%
\otimes\left(  \gamma\left(  y\right)  -qS\right)  \xi+e_{2}\otimes\left(
T-q\gamma\left(  x\right)  \right)  \xi.
\end{align*}
One can use its cochain version in terms of the column and row representations
of its differentials. Namely we come up with the following parametrized over
$\mathbb{C}_{xy}$ Banach space complex
\begin{equation}
0\rightarrow X\overset{d_{\gamma}^{0}}{\longrightarrow}X\oplus
X\overset{d_{\gamma}^{1}}{\longrightarrow}X\rightarrow0 \label{Xkos}%
\end{equation}
with the differentials
\[
d_{\gamma}^{0}=\left[
\begin{tabular}
[c]{l}%
$\gamma\left(  y\right)  -qS$\\
$T-q\gamma\left(  x\right)  $%
\end{tabular}
\ \ \right]  \quad\text{and}\quad d_{\gamma}^{1}=\left[
\begin{tabular}
[c]{ll}%
$T-\gamma\left(  x\right)  $ & $S-\gamma\left(  y\right)  $%
\end{tabular}
\ \ \right]  .
\]
One can consider the spectra (see \cite{Tay}, \cite{DosJOT1} and
\cite{DosAJM}) of the parametrized Banach space complex (\ref{Xkos}). By
Proposition \ref{tO}, we have $\operatorname{Spec}\left(  \mathcal{O}%
_{q}\left(  \mathbb{C}_{xy}\right)  \right)  =\mathbb{C}_{xy}=\mathbb{C}%
_{x}\cup\mathbb{C}_{y}$. It follows that
\begin{align*}
\sigma\left(  T,S\right)   &  =\sigma_{x}\left(  T,S\right)  \cup\sigma
_{y}\left(  T,S\right)  \text{ with}\\
\sigma_{x}\left(  T,S\right)   &  =\sigma\left(  T,S\right)  \cap
\mathbb{C}_{x}\text{ and }\sigma_{y}\left(  T,S\right)  =\sigma\left(
T,S\right)  \cap\mathbb{C}_{y}.
\end{align*}
\quad Thus we have the analog of the Taylor spectrum of the pair $\left(
T,S\right)  $, but lack of noncommutative Fr\'{e}chet $\widehat{\otimes}%
$-algebra sheaf makes it restrictive in applications. One can use the formal
completion of $\mathcal{O}_{q}$ to apply the general framework of the
noncommutative functional calculus from \cite{DosJOT10}, but that would impose
an extra condition for $TS$ to be a nilpotent operator instead of a
quasinilpotent one (see Lemma \ref{lAq}). As in the case of a nilpotent Lie
algebra of operators \cite{DosPOMI}, a general picture of an operator quantum
plane calculus can be depicted by means of the presheaf of noncommutative
Fr\'{e}chet $\widehat{\otimes}$-algebras $\mathcal{O}_{q}$ shown above.


\begin{thebibliography}{99}                                                                                               %


\bibitem {BurComA}N. Bourbaki, Commutative algebra, Ch. 1-7, Moscow (MIR) 1971.

\bibitem {BourST}N. Bourbaki, Spectral theory, Moscow (MIR) 1972.

\bibitem {DosJOT1}A. A. Dosiev, \textit{Spectra of infinite parametrized
Banach complexes,} J. Operator Theory, 48 (2002) 585--614.

\bibitem {DosPOMI}A. A. Dosiev, \textit{Algebras of power series in elements
of a Lie algebra and Slodkowski spectra, }J. Math. Sciences, 124 (2) (2004) 4886-4908.

\bibitem {Dos}A. A. Dosiev, \textit{Cohomology of sheaves of}
\textit{Fr\'{e}chet algebras and spectral theory, }Funct. Anal. Appl. 39 (3)
(2005) 225-228.

\bibitem {DosievIE}A. A. Dosiev, \textit{Regularities in noncommutative Banach
algebras,} Integr. Equ. Oper. Theory 61 (2008) 341-364.

\bibitem {DosIzv}A. A. Dosi, \textit{Non-commutative holomorphic functions in
elements of a Lie algebra and the absolute basis problem, }Izvestiya Math. 73
(6) (2009) 1149--1171.

\bibitem {DosJOT10}A. A. Dosi, \textit{Taylor functional calculus for
supernilpotent Lie algebra of operators,} J. Operator Theory 63 (1) (2010) 101-126.

\bibitem {DosAJM}A. A. Dosi, \textit{A survey of spectra of parametrized
Banach space complexes, }Azerb. J. Math. 1 (1) (2011) 3-56.

\bibitem {Goor}K. R. Goodearl, \textit{Quantized coordinate rings and related
Noetherian algebras,} Proceedings of the 35th Symposium on Ring Theory and
Representation Theory (Okayama, 2002) 19-45.

\bibitem {Ll}G. R. Luecke, \textit{A note on spectral continuity and on
spectral properties of essentially }$G_{1}$\textit{ operators, }Pacific J.
Math. 69 (1) (1977) 141-149.

\bibitem {M}Yu. I. Manin, \textit{Some remarks on Koszul algebras and quantum
groups}, Ann. Inst. Fourier (Grenoble) 37 (1987), 191--205.

\bibitem {Harts}R. Hartshorne, Algebraic geometry, Grad. Texts Math. Springer 1977.

\bibitem {Hel}A. Ya. Helemskii, Banach and polynormed algebras: general
theory, representations, homology, Nauka, Moscow, 1989.

\bibitem {HelHom}A. Ya. Helemskii, The homology of Banach and topological
algebras, Math. Appl. (41), Kluwer Acad. Publ. 1989.

\bibitem {Pir06}A. Yu. Pirkovskii, \textit{Stably flat completions of
universal enveloping algebras,} Dissertationes Math. 441 (2006) 1--60.

\bibitem {Pir16}A. Yu. Pirkovskii, \textit{Arens-Michael enveloping algebras
and analytic smash products,} Proc. Amer. Math. Soc. 134 (9) (2006) 2621--2631.

\bibitem {Pir}A. Yu. Pirkovskii, \textit{Arens-Michael envelopes, homological
epimorphisms and relatively quasifree algebras, }Trans. Moscow Math. Soc. 69
(2008) 27--104.

\bibitem {Pir09}A. Yu. Pirkovskii, \textit{Homological dimensions of complex
analytic and smooth quantum tori,} Math. methods and Appl. Proc. of the 18th
RSSU (2009) 119--142.

\bibitem {Put}M. Putinar, \textit{Functional calculus with sections of an
analytic space}, J. Oper. Th. 4 (1980) 297-306.

\bibitem {Sch}H. Schaefer, Topological vector spaces, Springer-Verlag New
York, Heidelberg,Berlin (1970).

\bibitem {Tah}L. A. Takhtajan, \textit{Noncommutative cohomologies of the
quantum toros,} Funct. Anal. Appl. 23 (2) (1989) 75-76.

\bibitem {Tay}J. L. Taylor, \textit{A joint spectrum for several commuting
operators, }J. Funct. Anal., 6 (1970) 172-191.

\bibitem {Tay2}J. L. Taylor, \textit{A general framework for a multi-operator
functional calculus,} Adv. Math. 9 (1972) 183-252.

\bibitem {WM}M. Wambst, \textit{Complexes de Koszul quantiques,} Annales de
l'Institut Fourier, 43 (4) (1993) 1089-1156.
\end{thebibliography}
\end{document}